\documentclass{aic}

\aicAUTHORdetails{%
  title = {$11/4$-Colorability of Subcubic Triangle-Free Graphs}, 
  author = {Zden\v{e}k Dvo\v{r}\'ak, Bernard Lidick\'y and Luke Postle},
  plaintextauthor = {Zden\v{e}k Dvo\v{r}\'ak, Bernard Lidick\'y, Luke Postle},
  plaintexttitle = {11/4-colorability of Subcubic Triangle-free Graphs}, 
  keywords = {fractional coloring, subcubic graphs, triangle-free graphs},
}   

\aicEDITORdetails{%
   year={2025},
   number={5},
   received={29 April 2022},   
   published={23 April 2025},  
   doi={10.19086/aic.2025.5},      
}   

\usepackage[utf8]{inputenc}
\usepackage{amsmath,amssymb,amsthm,amsfonts,stmaryrd}
\usepackage{epsfig}
\usepackage{mathrsfs,enumitem}
\usepackage{microtype}
\usepackage{tikz}
\usetikzlibrary{calc}
\usepackage{fancyvrb}

\newenvironment{subproof}{%
  \begin{proof}[Subproof]%
}{%
  \end{proof}%
}

\tikzset{vtx/.style={inner sep=2.5pt, outer sep=0pt, circle, fill=white,draw=black},
nail/.style={inner sep=2.5pt, outer sep=0pt, rectangle, fill=black, draw=black},
vtxS/.style={inner sep=2.5pt, outer sep=0pt, circle, fill=white,draw=black,double},
nailS/.style={inner sep=2.5pt, outer sep=0pt, rectangle, fill=black, draw=black,double},
}

\newtheorem{theorem}{Theorem}[section]
\newtheorem{corollary}[theorem]{Corollary}
\newtheorem{definition}[theorem]{Definition}
\newtheorem{lemma}[theorem]{Lemma}

\newtheorem{conjecture}[theorem]{Conjecture}
\newtheorem{observation}[theorem]{Observation}
\newtheorem{megaclaim}{Assertion}

\newcommand{\starg}[5]{standard argument for the configuration $#2$ attaching at $#3$ and the replacement graph #1 gives an e-graph $#4$ and its proper induced sub-e-graph $#5\in\CC_0$}
\newcommand{\stex}[2]{configuration $#1$ attaching at $#2$ is excluded by reducibility}
\newcommand{\II}{\mathcal{I}}
\newcommand{\PP}{\mathcal{P}}
\newcommand{\CC}{\mathcal{C}}
\newcommand{\kk}{\ensuremath{K_4^+}}

\newcommand{\vc}[1]{\ensuremath{\vcenter{\hbox{#1}}}}

\date{}

\begin{document}
\begin{frontmatter}[classification=text]

\title{$11/4$-Colorability of Subcubic Triangle-Free Graphs\footnote{In memory of Robin Thomas, who inspired much more than just this paper.}}


\author[zd]{Zdeněk Dvořák\thanks{This article is part of a project that has received funding from the European Research Council (ERC) under the
European Union’s Horizon 2020 research and innovation programme (grant agreement No 810115).}}
\author[bl]{Bernard Lidický\thanks{Supported in part by NSF grant DSM-2152490 and Scott Hanna professorship.}}
\author[lp]{Luke Postle\thanks{Supported by Canada Research Chair in Graph Theory. Partially supported by NSERC under Discovery Grant No. 2019-04304, the Ontario Early Researcher Awards program and the Canada Research Chairs program.}}

\begin{abstract}
We prove that up to two exceptions, every connected subcubic triangle-free graph  has fractional chromatic number at most $11/4$.
This is tight unless further exceptional graphs are excluded, and improves the known bound on
the fractional chromatic number of subcubic triangle-free planar graphs.
\end{abstract}
\end{frontmatter}

\section{Introduction}

Fractional coloring was introduced in 1973~\cite{planfr5} as an
approach for either disproving, or giving more evidence for the~Four Color Theorem.
While this largely failed, the topic of fractional coloring turned out to be interesting
on its own, in part due to applications in the study of the independence number.
For a real
number $k$, a graph $G$ is \emph{fractionally $k$-colorable} if any of the following equivalent~\cite{ScheinermanUllman2011}
conditions holds:
\begin{itemize}
\item There exists an assignment $\varphi$ of sets of measure $1$ to vertices of $G$ such that
$\varphi(u)\cap \varphi(v)=\emptyset$ for every edge $uv\in E(G)$, and the measure of
$\bigcup_{v\in V(G)} \varphi(v)$ is at most $k$.
\item For some positive integers $a$ and $b$ such that $a/b\le k$, there
exists an assignment $\varphi$ of subsets of $\{1,\ldots,a\}$ of size $b$ to vertices of $G$ such that
$\varphi(u)\cap \varphi(v)=\emptyset$ for every edge $uv\in E(G)$.
\item For every assignment of non-negative weights to vertices of $G$, there is an independent set in $G$
that contains at least $(1/k)$-fraction of the total weight.
\end{itemize}
Note that either of the first two conditions directly implies that for an integer $k$, a properly $k$-colorable graph
is also fractionally $k$-colorable.  The \emph{fractional chromatic number} of $G$ is defined as the infimum of the values $k$ such that $G$ is fractionally $k$-colorable
(actually, this infimum is achieved, and thus we could write ``minimum'' in this definition~\cite{ScheinermanUllman2011}).
By considering the uniform weight assignment, the last of the equivalent conditions shows that every fractionally
$k$-colorable graph on $n$ vertices contains an independent set of size at
least $n/k$. Conversely, in case graphs from a certain class contain large independent sets,
it is natural to ask whether they have bounded fractional chromatic number.

The independence number of triangle-free graphs is of particular interest as one of the most basic
instances of the Ramsey theory---the fact that the Ramsey number $R(3,t)$ is $\Theta(t^2/\log t)$~\cite{akomsem,kim1995ramsey}
is equivalent to saying that the minimum possible independence number of a triangle-free graph on $n$ vertices is $\Theta(\sqrt{n\log n})$.
In terms of the maximum degree, it is known~\cite{akomsem} that an $n$-vertex triangle-free graph of maximum degree at most $\Delta$ has
an independent set of size $\Omega(n\log \Delta/\Delta)$.  Actually, they also have the chromatic number $O(\Delta/\log \Delta)$, but
the proof of this fact is substantially more involved~\cite{johansson1996asymptotic,molloy2019list,martinsson2021simplified}.

The bounds mentioned in the previous paragraph are asymptotic and do not give any information for small values of $\Delta$.
In this paper, we are interested in the case of the maximum degree $\Delta\le 3$, i.e., \emph{subcubic} triangle-free graphs.
Culminating a series of previous results, Staton~\cite{Sta79} proved that every $n$-vertex subcubic triangle-free graph
has an independent set of size at least $5n/14$; this bound is optimal, since the
generalized Petersen graph $P(7,2)$ has $14$ vertices and no independent set of
size greater than $5$, as observed by Fajtlowicz~\cite{fajtlowicz1978size}.  In fact, $P(7,2)$ and another $14$-vertex
graph (see Figure~\ref{fig-forb}) are the only connected graphs for which this bound is tight.
Fraughnaugh and Locke~\cite{fralo} proved that a connected $n$-vertex subcubic triangle-free graph
has an independent set of size at least $(11n-4)/30$, improving upon Staton's bound for $n>14$.
Based on earlier experimental evidence of Bajnok and Brinkmann~\cite{bajnok1998independence},
Fraughnaugh and Locke~\cite{fralo} also conjectured an improvement of this bound: An $n$-vertex subcubic triangle-free graph
avoiding graphs $F_{14}^{(1)}$, $F_{14}^{(2)}$, $F_{11}$, $F_{22}$, $F_{19}^{(1)}$, and $F_{19}^{(2)}$ as subgraphs
has an independent set of size at least $3n/8$.  Here $F_{14}^{(1)}$ and $F_{14}^{(2)}$ are the two graphs depicted in Figure~\ref{fig-forb},
and $F_{11}$, $F_{22}$, $F_{19}^{(1)}$ and $F_{19}^{(2)}$ are the graphs depicted in Figure~\ref{fig-forb2} which also have
independence number smaller than $3/8$ times their number of vertices.
This conjecture was recently confirmed by Cames van Batenburg et al.~\cite{van2019large}.
Let us remark that this result is tight, as there exist infinitely many $3$-connected subcubic triangle-free graphs $G$ of girth five
with independence number $3|V(G)|/8$.

The study of these questions from the perspective of the fractional chromatic number was first proposed by
Heckman and Thomas.  In the paper where they gave a new proof of Staton's result, they suggest the following natural strengthening.
\begin{conjecture}[Heckman and Thomas~\cite{thoheck}]
\label{conjecture}
Every subcubic triangle-free graph is fractionally $14/5$-colorable.
\end{conjecture}
Furthermore, in~\cite{HeTh06} they proved that every $n$-vertex planar subcubic triangle-free graph has an independent set of size at least $3n/8$
(since the graphs $F_{14}^{(1)}$, \ldots, $F_{19}^{(2)}$ are non-planar, this is a special case of the result
of Cames van Batenburg et al.~\cite{van2019large}), and gave the corresponding fractional chromatic number conjecture.
\begin{conjecture}[Heckman and Thomas~\cite{HeTh06}]
\label{conjecture-planar}
Every subcubic triangle-free planar graph is fractionally $8/3$-colorable.
\end{conjecture}
Conjecture~\ref{conjecture} was proved by Dvořák et al.~\cite{fracsub}.  Cames van Batenburg et al.~\cite{van2019large}
proposed the natural strengthening, which would imply Conjecture~\ref{conjecture-planar}.
\begin{conjecture}[Cames van Batenburg et al.~\cite{van2019large}]
\label{conjecture-forb}
Every subcubic triangle-free graph avoiding $F_{14}^{(1)}$, $F_{14}^{(2)}$, $F_{11}$, $F_{22}$, $F_{19}^{(1)}$, and $F_{19}^{(2)}$ as subgraphs
is fractionally $8/3$-colorable.
\end{conjecture}

In this paper, we present the first step towards this conjecture.

\begin{theorem}\label{thm-mainfr}
Let $G$ be a subcubic triangle-free graph.  If no component of $G$ is isomorphic to the graphs
$F_{14}^{(1)}$ and $F_{14}^{(2)}$ depicted in Figure~\ref{fig-forb},
then $G$ has fractional chromatic number at most $11/4$.
\end{theorem}
Note that the bound $11/4$ is the best possible if only the two subgraphs $F_{14}^{(1)}$ and $F_{14}^{(2)}$ are forbidden,
as the graphs $F_{11}$ and $F_{22}$ have fractional chromatic number $11/4$.  Conversely, let us remark that $F_{14}^{(1)}$ and $F_{14}^{(2)}$
both have fractional chromatic number $14/5>11/4$, and thus we cannot avoid forbidding them.
Since $F_{14}^{(1)}$ and $F_{14}^{(2)}$ are both non-planar, we obtain the following partial result towards Conjecture~\ref{conjecture-planar}.
\begin{corollary}
Every subcubic triangle-free planar graph is fractionally $11/4$-colorable.
\end{corollary}

Let us end the introduction by proposing the following conjecture, intermediate between Theorem~\ref{thm-mainfr}
and Conjecture~\ref{conjecture-forb}.
\begin{conjecture}
Every subcubic triangle-free graph avoiding $F_{14}^{(1)}$, $F_{14}^{(2)}$, $F_{11}$, and $F_{22}$
as subgraphs is fractionally $19/7$-colorable.
\end{conjecture}

In the following section, we present a strengthening of Theorem~\ref{thm-mainfr} that enables us to carry out an inductive argument;
this strengthening is motivated by the proof method of Heckman and Thomas~\cite{thoheck} and Dvořák et al~\cite{fracsub}.
In the following sections, we consider the properties of a hypothetical minimal counterexample to this strengthening,
gradually obtaining more information on its structure.
\begin{itemize}
\item First, we argue that it can have only trivial 2-edge-cuts (where one side of the cut is a path of vertices of degree two)
and use this to conclude that the minimal counterexample cannot have nailed vertices, and that any adjacent vertices of degree two
must be contained in a 5-cycle.
\item Next, we show that every vertex has at most one neighbor of degree two.
\item Then, we show that the minimum counterexample does not contain \kk{} as an induced subgraph. This significantly simplifies
the verification that the e-graphs arising from further reductions do not have critical induced sub-e-graphs.
\item In the next step, we reduce 5-cycles containing two vertices of degree two.  As we have argued before that
any adjacent vertices of degree two belong to a 5-cycle and that no vertex has two neighbors of degree two,
this implies that the distance between distinct vertices of degree two is at least three.
\item Next, we show that vertices at distance at most one from a 4-cycle have degree at least three.
\item We now enter the final stages of the proof.  Up to this point, we generally used straightforward reductions,
replacing the configuration of interest with a smaller one, arguing that the resulting e-graph is valid and does not
contain critical induced sub-e-graphs and thus has an $11/4$-coloring, and showing that this $11/4$-coloring can be extended
to the original configuration.

We now reduce vertices of degree two, for which this approach is not sufficient.
Instead, we try two different reductions, neither of which quite works---we obtain two fractional colorings $\varphi_1$
and $\varphi_2$ which assign fewer colors than required to some of the vertices.  However, they also assign more colors than required
to other vertices, and we argue that a convex combination of $\varphi_1$ and $\varphi_2$ is an $11/4$-coloring of the whole graph.
\item Using similar ideas, we then eliminate 4-cycles.
\end{itemize}
Hence, at this point (Section~\ref{sec-final}), we know that the minimum counterexample is 3-regular and has girth at least five.
We now come to the core of the argument, very similar to the one used in~\cite{fracsub}.
For each vertex $v$ of a minimum counterexample $G$, we delete $v$ and the neighbors of $v$
and find an $11/4$-coloring of the resulting e-graph $G_v$ without nailed vertices.  We then convexly combine these $|V(G)|$ colorings to obtain an $11/4$-coloring
of the whole graph, which gives a contradiction.  Compared to~\cite{fracsub}, we need to work much harder to argue that
the graph $G_v$ does not contain critical induced sub-e-graphs.  Indeed, the critical e-graphs $F_{14}^{(1)}$ and
$F_{14}^{(2)}$ arise in the analysis performed in this step.

\subsection*{Programs}

Many of the arguments in this paper are computer-assisted; the programs we used to verify the claims
can be found at \href{https://iuuk.mff.cuni.cz/~rakdver/elevenfour/}{\url{https://iuuk.mff.cuni.cz/~rakdver/elevenfour/}}.
To decrease the chance of errors, two of the authors wrote independent programs using different frameworks:
\begin{itemize}
\item The program by Zdeněk Dvořák is written in C++, using the \href{https://www.bugseng.com/parma-polyhedra-library}{Parma Polyhedra Library}
to enumerate vertices of polyhedrons and \href{https://www.gurobi.com/}{Gurobi} to solve linear programs.
While the Parma Polyhedra Library operates in exact arithmetics, Gurobi uses floating-point arithmetics.
However, to eliminate the issue of rounding errors, we convert the obtained solutions (of primal or dual programs, depending
on the outcome) to rational numbers and verify their validity in exact arithmetics.
\item The program by Bernard Lidický is written in SageMath. It can use variety of solvers shipped with SageMath, including Gurobi and/or Parma Polyhedral Library.
\end{itemize}

\section{Preliminaries}

Following the ideas of \cite{fracsub,thoheck}, we are going to prove a strengthening of Theorem~\ref{thm-mainfr}
where we put additional constraints on the coloring of vertices of degree less than three.  When coloring the given graph $G$,
we generally choose 
an induced subgraph $H$ of $G$, find its coloring inductively, then extend the coloring to $G$,
taking advantage of these additional constraints on the coloring of $H$.  However, occasionally it is necessary
not to enforce these constraints at some of the vertices of $H$, and to treat the vertices according to their original degree in $G$
rather than their current degree in $H$.  To deal with this issue, we introduce the following definition.

A \emph{graph with external degrees} (an \emph{e-graph} for short) is a graph $G$ (the \emph{underlying graph} of the e-graph)
together with a function $d_G:V(G)\to\mathbb{Z}_0^+$ such that $d_G(v)\ge \deg v$ for every $v\in V(G)$. 
We say $\deg v$ is the \emph{degree} and $d_G(v)$ is the \emph{external degree} of the vertex $v$ of the e-graph.
We use $\deg_G v$ when we need to specify that the degree is in $G$. 
A vertex $v\in V(G)$ is \emph{nailed} if $d_G(v)>\deg v$.
We say an e-graph $G$ is \emph{subcubic} if $d_G(v)\le 3$ for every $v\in V(G)$, and \emph{cubic} if $d_G(v)=3$ for every $v\in V(G)$;
let us remark the underlying graph of a cubic e-graph may contain vertices of degree less than three.
In contrast, we say an e-graph $G$ is \emph{3-regular} if all its vertices have degree three.
Note that if $G$ is subcubic and has minimum degree at least two, then the function $d_G$ is uniquely determined by the set of
nailed vertices of $G$ (each nailed vertex $v$ has degree two and $d_G(v)=3$, while all other vertices $u$ satisfy
$d_G(u)=\deg u$); in figures, nailed vertices are drawn in black and all other vertices in white.
We say that $G$ is \emph{valid} if $G$ is subcubic, triangle-free, and $d_G(v)\ge 2$ for every $v\in V(G)$.
A \emph{sub-e-graph} of $G$ is an e-graph consisting of a subgraph of $G$ and the restriction of $d_G$ to its vertex set.
For $Y\subseteq G$, let $G[Y]$ denote the e-graph consisting of the subgraph of $G$ induced by $Y$ and the restriction of $d_G$ to $Y$;
we say that $G[Y]$ is an \emph{induced sub-e-graph} of $G$.  If $Y\neq V(G)$, we say $G[Y]$ is a \emph{proper induced sub-e-graph}.
Note that an induced sub-e-graph of a valid e-graph is itself valid.

For a measurable set $S$ of real numbers, let $|S|$ denote the measure of $S$.
For a real number $r>0$, by a \emph{set $r$-coloring} of an e-graph $G$, we mean a function $\varphi$ assigning a measurable subset of the interval $[0,r)$
to every vertex $v$ of $G$ such that $\varphi(u)\cap \varphi(v)=\emptyset$ for every $uv\in E(G)$.
Let us now introduce a key definition.
\begin{definition}
Let $\varphi$ be a set $11$-coloring of a subcubic e-graph $G$.
We say $\varphi$ is an \emph{$11/4$-coloring} of $G$ if $|\varphi(v)|=7-d_G(v)$ for every $v\in V(G)$.
A subcubic e-graph $G$ is \emph{$11/4$-colorable} if $G$ has an $11/4$-coloring.
\end{definition}
That is, each vertex receives a color set of measure at least four, and we require sets of larger measure to be assigned to vertices
with $d_G(v)<3$. In particular
\[
|\varphi(v)|= \begin{cases} 4 & \text{ if } d_G(v) = 3, \\
5 & \text{ if } d_G(v) = 2.\\
\end{cases}
\]
Observe that if the e-graph $G$ is cubic, then it has an $11/4$-coloring if and only if the underlying graph of $G$
is fractionally $11/4$-colorable.
A subcubic e-graph $G$ is \emph{critical} if $G$ does not have an $11/4$-coloring, but every proper induced sub-e-graph of $G$
has an $11/4$-coloring. 
Observe that valid critical graphs have a minimum degree at least 2.
Critical e-graphs characterize $11/4$-colorability in the following sense.
\begin{observation}\label{obs-cricol}
A subcubic e-graph $G$ is $11/4$-colorable if and only if $G$ does not contain a critical induced sub-e-graph.
\end{observation}

\begin{figure}
\begin{center}
\begin{tikzpicture}[scale=0.8]
\draw (0,0) node[vtx](a){} -- (0,1) node[vtx](b){} -- ++(0,1)node[vtx](c){}
-- ++(0,1)node[vtx](d){}
-- (1.5,1.5) node[vtx](e){}--(a) -- (-1.5,1.5) node[vtx](f){}--(d)
(e) to[bend right=90,looseness=2] 
node[vtx,pos=0.3](d){}
node[vtx,pos=0.7](d){}
(f)
;
\end{tikzpicture}
\end{center}
\caption{The graph \kk.}\label{fig-kk}
\end{figure}
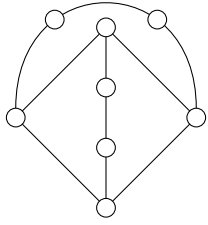

\begin{figure}
\begin{center}
\begin{tikzpicture}[scale=0.8]
\draw
\foreach \x in {0,...,6}{
(90+\x*52:2) node[vtx](x\x){}
--
(90+\x*52:1) node[vtx](y\x){}
}
(x0)--(x1)--(x2)--(x3)--(x4)--(x5)--(x6)--(x0)
(y0)--(y2)--(y4)--(y6)--(y1)--(y3)--(y5)--(y0)
;
\draw(0,-2.5) node{$F_{14}^{(1)}$};
\end{tikzpicture}
\hspace{1 cm}
\begin{tikzpicture}[scale=0.8]
\draw
\foreach \x in {0,...,7}{
(90+22.5+\x*45:2) node[vtx](x\x){}
}
(x0)--(x1)--(x2)--(x3)--(x4)--(x5)--(x6)--(x7)--(x0)
;
\path (x0) -- 
node[vtx,pos=0.25](z1){}  
node[vtx,pos=0.5](z2){}  
node[vtx,pos=0.75](z3){}  
(x3);
\path (x7) -- 
node[vtx,pos=0.25](w1){}  
node[vtx,pos=0.5](w2){}  
node[vtx,pos=0.75](w3){}  
(x4);
\draw 
(x0)--(z1)--(z2)--(w2)--(w1)--(x7)
(x2)--(z2)
(w2)--(x5)
(z3)--(x1) (z3)--(w1) (z3)--(x4)
(w3)--(x6) (w3)--(z1) (w3)--(x3)
;
\draw(0,-2.5) node{$F_{14}^{(2)}$};
\end{tikzpicture}
\end{center}
\caption{Valid critical cubic e-graphs $F_{14}^{(1)}$ and $F_{14}^{(2)}$ .}\label{fig-forb}
\end{figure}
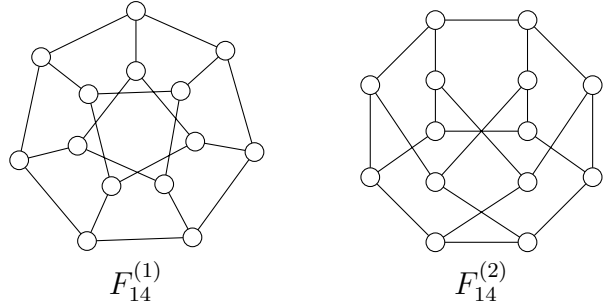

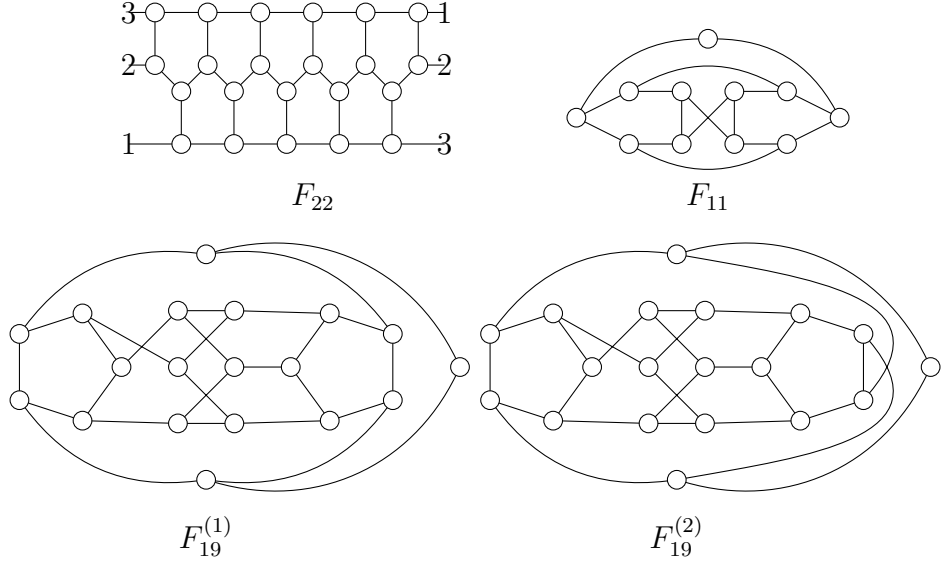
\begin{figure}
\begin{center}
\begin{tikzpicture}[scale=0.7]
\draw
\foreach \x in {0,...,4}{
(\x,0) node[vtx](x\x){} -- ++ (0,1) node[vtx](y\x){}
-- ++(0.5,0.5) node[vtx](z\x){} 
-- ++(0,1) node[vtx](w\x){} 
}
;
\foreach \x in {0,...,3}{
\pgfmathtruncatemacro{\xx}{\x+1}
\draw 
(x\x) -- (x\xx)  (z\x) -- (y\xx) (w\x) -- (w\xx)
; 
}
\draw (y0) -- ++(-0.5,0.5) node[vtx](z){} --  ++(0,1) node[vtx](w){}--(w0);
\draw
(w4)--++(0.5,0) node{1}
(z4)--++(0.5,0) node{2}
(x4)--++(1,0) node{3}
(w)--++(-0.5,0) node{3}
(z)--++(-0.5,0) node{2}
(x0)--++(-1,0) node{1}
;
\draw (2.5,-1) node{$F_{22}$};
\end{tikzpicture}
\hskip 3em
\begin{tikzpicture}[scale=0.7]
\draw
\foreach \x in {0,...,3}{
(\x,1) node[vtx](y\x){}
(\x,0) node[vtx](x\x){}
}
(-1,0.5) node[vtx](xl){}
(4,0.5) node[vtx](xr){}
(1.5,2) node[vtx](yt){}
;
\draw
(xl)--(x0)--(x1)--(y2)--(y3)--(xr)--(x3)--(x2)--(y1)--(y0)--(xl) to[bend left] (yt) (yt) to[bend left] (xr)
(x1)--(y1) (x2)--(y2)
(x0) to[bend right] (x3)
(y0) to[bend left] (y3)
;
\draw (1.5,-1) node{$F_{11}$};
\end{tikzpicture}\\
\newcommand{\Fnineteen}{
\foreach \x in {1,2,3,4,5}{
\draw (0,0)++(-72+72*\x:1) node[vtx](a\x){};
\draw (5,0)++(180-72+72*\x:1) node[vtx](b\x){};
}
\foreach \x in {1,2,3}{
\draw (2,-2+\x) node[vtx](aa\x){};
\draw (3,-2+\x) node[vtx](bb\x){};
}
\draw
(2.5,2) node[vtx](x){}
(2.5,-2) node[vtx](y){}
(7,0) node[vtx](z){}
;
\draw
(a1)--(a2)--(a3)--(a4)--(a5)--(a1)
(b1)--(b2)--(b3)--(b4)--(b5)--(b1)
(aa1)--(bb1)--(aa2)--(bb3)--(aa3)--(bb2)--(aa1)
(a1)--(aa3) (a2)--(aa2) (a5)--(aa1)
(b1)--(bb2) (b2)--(bb1) (b5)--(bb3)
(a3) to[bend left] (x)
(a4) to[bend right] (y)
(x) to[bend left=40] (z)
(y) to[bend right=40] (z)
;
}
\begin{tikzpicture}[scale=0.75]
\Fnineteen
\draw
(x) to[bend left] (b4)
(y) to[bend right] (b3)
(2.5,-3) node{$F_{19}^{(1)}$}
;
;
\end{tikzpicture}
\begin{tikzpicture}[scale=0.75]
\Fnineteen
\draw
(x) to[out=-10,in=50,looseness=1.4] (b3)
(y) to[out=10,in=-50,looseness=1.4] (b4)
(2.5,-3) node{$F_{19}^{(2)}$}
;
\end{tikzpicture}
\end{center}
\caption{Subcubic graphs $F_{22}$ and $F_{11}$ with fractional chromatic number $\frac{11}{4}$ 
and $F_{19}^{(1)}$ and $F_{19}^{(2)}$ with fractional chromatic number $\frac{19}{7}$.
}\label{fig-forb2}
\end{figure}

Let \kk{} be the graph obtained from the clique $K_4$ by subdividing twice the edges of a perfect matching; see Figure~\ref{fig-kk}.
Let $\CC$ denote the set of all pairwise non-isomorphic valid critical e-graphs.  Let us now state the strengthening of
Theorem~\ref{thm-mainfr} that we are about to prove in the rest of the paper.
\begin{theorem}\label{thm-crit}
The set $\CC$ contains exactly 176 e-graphs.  Out of these, only the two $14$-vertex e-graphs depicted in Figure~\ref{fig-forb} are cubic.
Furthermore:
\begin{itemize}
\item[(a)] Each e-graph in $\CC$ has at most two nailed vertices, and if it has two, then its underlying graph is either $C_5$ or \kk.
\item[(b)] If $G\in \CC$ has exactly one nailed vertex, then $G$ contains at least three non-nailed vertices of degree two.
\item[(c)] No graph in $\CC$ has exactly three vertices of degree two.
\end{itemize}
\end{theorem}

This implies our main result.
\begin{proof}[Proof of Theorem~\ref{thm-mainfr}.]
We can assume $G$ is connected, as otherwise we can color each component separately.
Let us extend $G$ to a valid cubic e-graph by defining $d_G(v)=3$ for every $v\in V(G)$.
By the assumptions, $G$ is not isomorphic to either of the e-graphs depicted in Figure~\ref{fig-forb}.
Moreover, the e-graphs depicted in Figure~\ref{fig-forb} are 3-regular, and since $G$ is connected and
has maximum degree at most three, it does not contain either of them as a proper induced sub-e-graph.
Since $G$ is cubic, all its induced sub-e-graphs are cubic, and thus Theorem~\ref{thm-crit} implies that $G$ does not contain a critical induced sub-e-graph.
By Observation~\ref{obs-cricol}, $G$ has an $11/4$-coloring, and thus the underlying graph of $G$ has fractional chromatic number at most~$11/4$.
\end{proof}

Let $\CC_0$ be the set of e-graphs listed in the Appendix (which we obtained by a computer-assisted enumeration).  Using computer,
we also verified that all e-graphs in $\CC_0$ are valid, subcubic, critical (so $\CC_0\subseteq\CC$), and pairwise non-isomorphic,
$|\CC_0|=176$, the only cubic e-graphs in $\CC_0$ are the two drawn in Figure~\ref{fig-forb}, and the following conditions
analogous to those from Theorem~\ref{thm-crit} hold:
\begin{itemize} 
\item[(a0)] Each e-graph in $\CC_0$ has at most two nailed vertices, and if it has two, then its underlying graph is either $C_5$ or \kk.
\item[(b0)] If $G\in \CC_0$ has exactly one nailed vertex, then $G$ contains at least three non-nailed vertices of degree two.
\item[(c0)] No graph in $\CC_0$ has exactly three vertices of degree two.
\end{itemize}
Therefore, Theorem~\ref{thm-crit} is implied by the following lemma.

\begin{lemma}\label{lemma-main}
Every e-graph $G\in \CC$ belongs to $\CC_0$.
\end{lemma}

The rest of the paper is devoted to the proof of Lemma~\ref{lemma-main}.
We proceed by contradiction; we say a graph $G$ is a \emph{minimum counterexample} if
$G$ is an e-graph in $\CC\setminus \CC_0$ with the smallest number of vertices and subject to that
the smallest number of nailed vertices.
Thus, $G$ is valid and critical, and in particular, $G$ does not have an $11/4$-coloring.

Note that the subcubic e-graphs with at most two nailed vertices and underlying graphs $C_5$ or \kk{} are not $11/4$-colorable,
and they all belong to $\CC_0$.  Since a minimum counterexample is critical, we obtain the following useful observation.
\begin{observation}\label{obs-3nails}
If $G$ is a minimum counterexample, then every induced sub-e-graph of $G$
whose underlying graph is $C_5$ or \kk{} has at least three nailed vertices.
In particular, if $G$ has no nailed vertices, then every $5$-cycle in $G$ contains at least three
vertices of degree three.
\end{observation}

Let us also note the following simple but useful observation.
\begin{observation}\label{obs-ncnail}
Let $G$ be an e-graph and let $H$ be an induced sub-e-graph of $G$.  Every vertex $v\in V(H)$
incident with an edge $e\in E(G)\setminus E(H)$ is nailed in $H$.
\end{observation}

\subsection{Coloring and reducibility}

The initial part of the proof of Lemma~\ref{lemma-main} consists of a series of reducible configuration arguments.
In the ordinary proper coloring setting, such an argument proceeds along the following lines.
Suppose $G=G_1\cup G_2$, where $G_1$ and $G_2$ are proper induced subgraphs of $G$, and let $S=G_1\cap G_2$.
Suppose moreover that $G_1$ is a quite small graph, and thus we can directly verify that it satisfies the
following condition: Every proper coloring of $S$ (by the given number of colors) extends to a proper coloring of $G_1$.
This implies that $G$ has a proper coloring by the given number of colors if and only if $G_2$ has one, and thus the argument
can be finished by applying induction.

In the fractional coloring setting, the argument is slightly more involved, since
$S$ typically has infinitely many fractional colorings by the given measure of colors.
However, through the standard linear programming reformulation of fractional coloring, it can be
seen that the fractional colorings of $S$ (up to measure-preserving transformations) form
a polytope, and to verify that all of them extend to fractional colorings of $G_1$, it suffices
to verify this is the case for the colorings that form the vertices of this polytope.
In this section, we describe this idea in more detail.

Let $r>0$ be a real number.  The set $r$-colorings of an e-graph $G$ can be convexly combined, in the following sense.
Let $\varphi_1$, \ldots, $\varphi_n$ be set $r$-colorings of $G$, and let $\lambda_1,\ldots,\lambda_n\ge 0$
be real numbers summing to $1$.  For $i=1,\ldots, n$, choose arbitrarily a linear bijection $f_i$ from $[0,r)$
to a sub-interval of $[0,r)$ of length $\lambda_ir$, so that the intervals $f_1([0,r)]$, \ldots, $f_n([0,r))$ are pairwise
disjoint.  Let $\sum_{i=1}^n\lambda_i\varphi_i$ denote the function $\varphi$ defined by $\varphi(v)=\bigcup_{i=1}^n f_i(\varphi_i(v))$
for every $v\in V(G)$.
\begin{observation}\label{obs-convex}
Let $G$ be an e-graph and let $r>0$ be a real number.
Let $\varphi_1$, \ldots, $\varphi_n$ be set $r$-colorings of $G$, and let $\lambda_1,\ldots,\lambda_n\ge 0$
be real numbers summing to $1$.  Let $\varphi=\sum_{i=1}^n\lambda_i\varphi_i$.  Then $\varphi$ is a set $r$-coloring of $G$
and $|\varphi(v)|=\sum_{i=1}^n \lambda_i|\varphi_i(v)|$ for every $v\in V(G)$.
\end{observation}

Let $G$ be a subcubic e-graph and let $\II(G)$ denote the set of all independent sets of $G$.  Note that $G$ has an $11/4$-coloring if and only if we can assign a non-negative
real number $x(I)$ to every set $I\in\II(G)$ so that
\begin{itemize}
\item for every $v\in V(G)$, we have $\sum_{I\ni v} x(I)=7-d_G(v)$, and
\item $\sum_{I\in \II(G)} x(I)=11$.
\end{itemize}
Indeed, if $\varphi$ is an $11/4$-coloring, for a color $c\in [0,11)$ let $\varphi^{-1}(c)=\{v\in V(G):c\in \varphi(c)\}$
be the (independent) set of vertices on which $c$ appears, and for each independent set $I$, set $x(I)$ to be the measure of the set
$\{c:\varphi^{-1}(c)=I\}$.  
The converse is similarly easy.
Let $p$ be an mapping from $\II(G)$ to measurable subsets of $[0,11)$ such that
$|p(I)| = x(I)$ for all $I \in \II(G)$ and $p$ is a partition of $[0,11)$.
An $11/4$-coloring $\varphi$ of $G$ can be constructed by letting $\varphi(v) = \cup_{I \ni v} p(i)$
for each vertex $v\in V(G)$.

We say that this function $x:\II(G)\to\mathbb{R}_0^+$ is
the \emph{LP representation} of the coloring $\varphi$ (as the constraints on $x$ form a linear program).
Note that two $11/4$-colorings have the same LP representation exactly if they differ only by a measure-preserving transformation.

The functions $x\in\mathbb{R}^{\II(G)}$ satisfying the above constraints form a polytope, which we denote by $\PP(G)$, i.e.,
$$\PP(G) = \left\{\begin{array}{ll}
x\in\mathbb{R}^{\II(G)}:&x(I)\ge 0\text{ for all $I\in\II(G)$},\\ 
&\sum_{I\in \II(G)} x(I)=11,\\
&\sum_{I\ni v} x(I)=7-d_G(v)\text{ for all $v \in V(G)$}
\end{array}\right\}
$$
For a set $S\subseteq V(G)$, the \emph{restriction} of $x$ to $S$ is the function $y:\II(G[S])\to\mathbb{R}_0^+$
defined by
\begin{equation}\label{eq-restr}
y(J)=\sum_{I\in\II(G), I\cap S=J} x(I)
\end{equation}
for $J\in\II(G[S])$.  Note that if $x$ is the LP representation of an $11/4$-coloring $\varphi$,
then $y$ is the LP representation of the restriction of $\varphi$ to $G[S]$.  If $y$ is the restriction of the LP-representation
of an $11/4$-coloring of $G$, then we say that $y$ \emph{extends} to $G$.  Observe that $y$ extends to $G$ if and only if
there exists $x\in\PP(G)$ satisfying the linear equations (\ref{eq-restr}), and thus extendability of $y$ can be tested
algorithmically via linear programming.

We often deal with various conditions on $11/4$-colorings $\varphi$ of $G[S]$, of the form $|\varphi(u)\cap \varphi(v)|\le \alpha$
or $|\varphi(u)\cup \varphi(v)|\le \alpha$, for vertices $u,v\in S$ and a non-negative real number $\alpha$.
In the LP representation $y$ of $\varphi$, these correspond to the linear constraints $\sum_{J\supseteq\{u,v\}} y(J)\le \alpha$
or $\sum_{J\cap \{u,v\}\neq\emptyset} y(J)\le \alpha$.  In greater generality, the LP representation $y$ of $\varphi$ may be known
to belong to a certain polytope $P\subseteq\PP(G[S])$.
\begin{definition}
Let $G$ be a subcubic e-graph, let $S$ be a subset of its vertices, and let $P$ be a subpolytope of $\PP(G[S])$.
The e-graph $G$ is a \emph{reducible configuration subject to $P$} if every vertex of the polytope $P$ extends to $G$.
\end{definition}
This implies that every $11/4$-coloring of $G[S]$ satisfying
the constraints expressed by $P$ extends to an $11/4$-coloring of $G$.
To simplify the statements, we will usually specify $P$ only by listing the conditions on the $11/4$-colorings $\varphi$ of $G[S]$ it represents;
e.g., by ``$G$ is a reducible configuration subject to $|\varphi(u)\cap \varphi(v)|\le 1$'', we mean ``$G$ is a reducible configuration
subject to the polytope $P$ defined as the intersection of $\PP(G[S])$ with the half-space
$\sum_{J\supseteq\{u,v\}} y(J)\le 1$'', and so on.  Let us now state the key property of reducible configurations.

\begin{lemma}\label{lemma-redu}
Let $G=G_1\cup G_2$, where $G_2$ is an induced subgraph of $G$, and let $S=V(G_1)\cap V(G_2)$.
Let $P\subseteq\PP(G_1[S])$ be a polytope and suppose that $G_1$ is a reducible configuration subject to $P$.
Let $x_0$ be the LP representation of an $11/4$-coloring of $G_2$, and let $y\in\PP(G[S])$ be the restriction
of $x_0$ to $S$.
Let $y'\in \PP(G_1[S])$ be obtained from $y$ by setting $y'(J)=0$ for every $J\in\II(G_1[S])\setminus \II(G[S])$.
If $y'\in P$, then $G$ has an $11/4$-coloring.
\end{lemma}
\begin{proof}
Let $y_1$, \ldots, $y_m$ be the vertices of $P$.  Note that $y'\in P$ is a convex combination of the vertices, and thus
there exist $\alpha_1,\ldots,\alpha_m\ge 0$ such that $\sum_{i=1}^m\alpha_i=1$ and $\sum_{i=1}^m\alpha_iy_i=y'$.
For $i=1,\ldots, m$, since $G_1$ is a reducible configuration subject to $P$, there exists an $11/4$-coloring of $G_1$
with the LP representation $x_i$ whose restriction to $S$ is equal to $y_i$.
Then $x'=\sum_{i=1}^m\alpha_ix_i$ is the LP representation of an $11/4$-coloring of $G_1$ with restriction $y'$.
Let $x:\II(G)\to\mathbb{R}_0^+$ be defined by setting
$$x(I)=\frac{x'(I\cap V(G_1))x_0(I\cap V(G_2))}{y(I\cap S)}$$ for every $I\in \II(G)$
such that $y(I\cap S)\neq 0$, and $x(I)=0$ for $I\in \II(G)$ such that $y(I\cap S)=0$
(let us remark that in the latter case we also have $x_0(I\cap V(G_2))=x'(I\cap V(G_1))=0$).
It is easy to verify that $x$ is the LP representation of an $11/4$-coloring of $G$.
\end{proof}

Observe that the reducibility of a configuration can be tested algorithmically,
since the (finitely many) vertices of $P$ can be enumerated and their extendability can be verified via linear programming
as noted above.  Throughout the paper, whenever we claim a configuration is reducible, we perform such a verification
in our programs, and we will not mention this explicitly.

When arguing reducibility, we occasionally do not know whether all vertices in $S$ are pairwise distinct, or whether
they have external degree two or three.  For example, we consider a configuration consisting of a 4-cycle $K$ of vertices of degree
three.  The e-graph $G_1$ consists of $K$, the set $S$ of the neighbors of the vertices of $K$ outside of $K$, and the
edges between $K$ and $S$; in this situation, we do not know anything about the degrees of vertices of $S$,
or even whether $S$ consists of two, three, or four vertices.  So, it would seem we need to go over all these
cases and separately test their reducibility.  However, typically it suffices to prove the reducibility in the most restrictive
case that the vertices in $S$ are pairwise distinct and those that can have external degree two do: We are arguing that
a precoloring $\varphi$ of $S$ subject to certain constraints extends to $S$.  If two vertices of $S$ are identified
we can instead consider the situation where they are distinct but receive the same color set.  If a vertex has external degree three,
it is typically possible to add colors of measure one to its color set so that the constraints under which we prove the reducibility are satisfied,
obtaining a valid precoloring for the configuration where the vertex has external degree two.

To apply Lemma~\ref{lemma-redu}, we need to ensure that the $11/4$-coloring of $G_2$ (which exists if $G$ is critical and $G_2\neq G$)
satisfies the restrictions expressed by $P$.  To do so, we often find an $11/4$-coloring of a super-e-graph $G'_2$ of $G_2$, rather
than just of $G_2$ itself (a more involved argument is then needed to show that $G_2'$ has an $11/4$-coloring, the criticality
of $G$ is no longer sufficient).  The e-graph $G'_2$ is obtained from $G$ by replacing the known configuration $G_1$ by a
smaller set configuration.  More precisely, the replacement operation is defined as follows. Let $F$ and $R$ be e-graphs
sharing a subset $B=V(F)\cap V(R)$ of their (boundary) vertices, where $R[B]\subseteq F[B]$; these e-graphs specify
the original configuration $F$ and its replacement $R$.  Let $G=G_1\cup G_2$ and $G'_2=G'_1\cup G_2$ be e-graphs such that $G_2$
is an induced sub-e-graph of $G'_2$ and $V(G_1)\cap V(G_2)=V(G'_1)\cap V(G_2)$; let $S=V(G_1)\cap V(G_2)$.
If there exist isomorphisms $f_F$ of $F$ to $G_1$ and $f_R$ of $R$ to $G'_1$ such that
$f_F(B)=S$ and 
$f_F(v)=f_R(v)$ for all $v \in B$,
 then we say that $G'_2$ is obtained from $G$ by
\emph{replacing a sub-e-graph matching $F$ by $R$}\footnote{A careful reader might wonder whether
the inequality $R[B]\subseteq F[B]$ is in the right direction; when replacing the reducible configuration
by a smaller one, it may be useful to add new edges between the boundary vertices, but it is rarely
possible to drop them.  Indeed, we generally use the replacement operation in the opposite direction, where $F$ is
a smaller graph we used to replace the reducible configuration $R$, and we need to work out what graphs may arise from those
in $\CC_0$ by the inverse to such a replacement.  See the following section for more details.}, see Figure~\ref{fig:replace} for an illustration.
\begin{figure}
\begin{center}
\begin{tikzpicture}
\draw
(0,0) coordinate (a) 
(1,0) coordinate (b) 
(2,0) coordinate (c) 
(3,0) coordinate (d) 
;
\draw[fill=gray!20!white]
(a)  to[bend right] (b) to[bend right] (c) to[bend right] (d) to[bend left=90, looseness=2](a);
; 
\draw[fill=gray!20!white]
(a)  to[bend left] (b) to[bend left] (c) to[bend left] (d) to[bend right=90, looseness=1](a);
; 
\draw
(1.5,0.4) node{$G_1 \cong F$}
(1.5,-0.9) node{$G_2$}
(1.5,-2.1) node{$G$}
(a) node[vtxS]{}
(b) node[vtxS]{}
(c) node[vtxS]{}
(d) node[vtxS]{}
;
\begin{scope}[xshift = 4cm]
\draw
(0,0) coordinate (a) 
(1,0) coordinate (b) 
(2,0) coordinate (c) 
(3,0) coordinate (d) 
;
\draw[fill=gray!20!white]
(a)  to[bend right] (b) to[bend right] (c) to[bend right] (d) to[bend left=90, looseness=2](a);
; 
\draw[fill=gray!20!white]
(a)  to[bend left] (b) to[bend left] (c) to[bend left] (d) to[bend right=90, looseness=1](a);
; 
\draw
(1.5,0.4) node{$G'_1 \cong R$}
(1.5,-0.9) node{$G_2$}
(1.5,-2.1) node{$G_2'$}
(a) node[vtxS]{}
(b) node[vtxS]{}
(c) node[vtxS]{}
(d) node[vtxS]{}
;\end{scope}
\end{tikzpicture}
\end{center}
\caption{Replacing a sub-e-graph matching $F$ by $R$.}\label{fig:replace}
\end{figure}

\subsection{Common arguments}

We now introduce a shorthand for a line of reasoning we frequently use throughout the proof of Lemma~\ref{lemma-main}.
The setting is as follows.
\begin{itemize}
\item $G$ is a minimum counterexample,
\item $G_1$ is a sub-e-graph of $G$ and $S$ is a set of vertices of $G_1$ such that no edge of $E(G)\setminus E(G_1)$ has an end in $V(G_1)\setminus S$,
\item $P$ is (a list of conditions determining) a subpolytope of $\PP(G_1[S])$ as described in the previous section, and
\item $H$ is an e-graph with $V(H)\cap V(G)=S$, $H[S]\supseteq G_1[S]$, and $d_H(v)=d_G(v)$ for all $v\in S$.
\end{itemize}
Note there exists a unique induced sub-e-graph $G_2$ of $G$ such that $G=G_1\cup G_2$ and $S=V(G_1)\cap V(G_2)$.
Let $G'$ be the e-graph $G_2\cup H$.

As a specific example for such a setting, consider the situation depicted in Figure~\ref{fig-2deg2-2}: We want to show that the minimum counterexample $G$
cannot contain a vertex $v$ of degree three with two neighbors $u$ and $w$ of degree two.  Let $z$, $u'$, and $w'$ be the neighbors
of $v$, $u$, and $w$, respectively, not contained in $\{u,v,w\}$.  We consider the e-graph $G'$ obtained from $G-\{u,v,w\}$
by adding vertices $a$ and $b$, where $d_{G'}(a)=2$ and $d_{G'}(b)=3$, and edges $az$, $ab$, $bu'$ and $bw'$.
Hence, $S=\{u',w',z\}$ is the set of vertices in which these configurations attach to the rest of the graph,
$G_1$ is the configuration that we want to exclude formed by $S\cup \{u,v,w\}$ and the edges of $G$ incident with
$u$, $v$, and $w$, and $H$ is the replacement configuration formed by $S\cup\{a,b\}$ and the edges of $G'$ incident with $a$ and $b$.
The polytope $P$ will express the conditions that an $11/4$-coloring $\varphi$ of $G_1[S]$ satisfies
$|\varphi(u')\cup \varphi(v')|\le 7$ and $|\varphi(z)\cap(\varphi(u')\cup \varphi(w'))|\le 2$;
it is easy to check that any $11/4$-coloring $\varphi$ of $H$ must satisfy these conditions.

Suppose we verify that the following conditions are satisfied.
\begin{itemize}
\item[(i)] The e-graph $G'$ is valid and $|V(G')|<|V(G)|$.
\item[(ii)] The e-graph $G_1$ is a reducible configuration subject to $P$.
\item[(iii)] Every $11/4$-coloring of $H$ satisfies the conditions $P$.
\item[(iv)] For every e-graph in $\CC_0$, replacing any sub-e-graph matching $H$ by $G_1$ results in an e-graph that either is not critical
or belongs to $\CC_0$.
\end{itemize}
Since $G$ does not have an $11/4$-coloring, the condition (ii) and Lemma~\ref{lemma-redu} imply that no $11/4$-coloring of $G_2$ satisfies $P$.
By (iii), it follows that no $11/4$-coloring of $G_2$ extends to $G'$, and thus $G'$ does not have an $11/4$-coloring.
Hence, $G'$ contains a critical induced sub-e-graph~$F$.  By (i), $F$ is valid and $|V(F)|<|V(G)|$, and thus
the minimality of $G$ implies $F\in \CC_0$.  We claim that $F$ is actually a proper induced sub-e-graph of $G'$, that is, $F\neq G'$.
Indeed, note that $G$ is obtained from $G'$ by replacing a subgraph matching $H$ by $G_1$.  If $G'=F$,
then since $F\in\CC_0$, (iv) would imply that $G$ either is not critical or belongs to $\CC_0$, a contradiction to the assumption that $G$ is a minimum counterexample.

In the proof of Lemma~\ref{lemma-main}, we repeatedly use this line of reasoning, followed by a further analysis of the e-graph $F$.
We use the statement
``\emph{The \starg{$H$ enforcing $P$}{G_1}{S}{G'}{F}.}''
as a shorthand for the argument described above, including the claim that we verified the conditions (i)--(iv).
Let us remark that the following properties of $F$ are often useful in the analysis following the standard argument:
\begin{itemize}
\item $F\in \mathcal{C}_0$, and thus it satisfies the conditions (a0), (b0), and (c0).
\item $F$ is a proper induced sub-e-graph of $G'$, and thus $|V(F)|< |V(G')|<|V(G)|$.
\item Since both $F$ and $G$ are critical, $F$ is not a proper induced sub-e-graph of $G$, and thus
the intersection of $F$ with $H$ is non-empty.
\end{itemize}
Let us now comment on how we show that the conditions (i)--(iv) hold.
The condition (i) is usually straightforward to verify (it suffices to check we do not increase the degrees of vertices
and do not create triangles by the replacement).  The condition (ii) is verified by computer
(or by hand for very simple e-graphs $G_1$) using the procedure we described after Lemma~\ref{lemma-redu}.
The condition (iv) is also generally verified by a computer-assisted enumeration, or by hand in
simple cases when $H$ contains at least two nailed vertices, so that (a0) applies to e-graphs containing $H$.
The following easy observation is often used to show the validity of the condition (iii).
\begin{observation}\label{obs-constraints}
Let $H$ be an e-graph and let $S$ be a subset of vertices of $H$.  Let $\varphi$ be an $11/4$-coloring of $H$.
For any $A\subseteq S$, if all vertices in $A$ are adjacent to a vertex $z\in V(H)\setminus S$, then
$$\Bigl|\bigcup_{v\in A} \varphi(v)\Bigr|\le d_H(z)+4.$$
Let $B$ and $C$ be non-empty subsets of $S$.  If all vertices in $B$ are adjacent to a vertex $z_1\in V(H)\setminus S$,
all vertices in $C$ are adjacent to a vertex $z_2\in V(H)\setminus S$, and $z_1z_2\in E(H)$, then
$$\Bigl|\Bigl(\bigcup_{v\in B} \varphi(v)\Bigr)\cap \Bigl(\bigcup_{v\in C} \varphi(v)\Bigr)\Bigr|\le d_H(z_1)+d_H(z_2)-3.$$
\end{observation}
\begin{proof}
In the former case, $\bigcup_{v\in A} \varphi(v)\subseteq [0,11)\setminus\varphi(z)$ and
$|[0,11)\setminus\varphi(z)|=d_H(z)+4$.  In the latter case,
$\Bigl(\bigcup_{v\in B} \varphi(v)\Bigr)\cap \Bigl(\bigcup_{v\in C} \varphi(v)\Bigr)\subseteq [0,11)\setminus(\varphi(z_1)\cup \varphi(z_2))$
and $|[0,11)\setminus(\varphi(z_1)\cup \varphi(z_2))|=11-|\varphi(z_1)|-|\varphi(z_2)|=d_H(z_1)+d_H(z_2)-3$.
\end{proof}
We call all inequalities among the color sets of vertices of $S$ implied by this observation
the \emph{trivial constraints} of $H$, and we say that a vertex $z\in V(H)\setminus S$ \emph{participates in trivial constraints}
if $z$ is contained in a path in $H$ of length at most three with both ends in $S$.

\bigskip

Let us now introduce another shorthand, applied in the following setting.
\begin{itemize}
\item $G$ is a valid critical e-graph,
\item $G_1$ is a sub-e-graph of $G$ and $S$ is a set of vertices of $G_1$ such that no edge of $E(G)\setminus E(G_1)$ has an end in $V(G_1)\setminus S$, and
\item a vertex $z\in V(G_1)\setminus S$ does not participate in the trivial constraints of $G_1$.
\end{itemize}
Suppose we verify (using computer or by hand for very simple e-graphs $G_1$) that
\begin{itemize}
\item[(r)] the e-graph $G_1$ is a reducible configuration subject to its trivial constraints.
\end{itemize}
Note there exists a unique proper induced sub-e-graph $G_2$ of $G$ such that $G=G_1\cup G_2$ and $S=V(G_1)\cap V(G_2)$.
Since $G$ is critical, the e-graph $G-z$ has an $11/4$-coloring $\varphi$, and since $z$ does not participate
in the trivial constraints of $G_1$, the restriction of $\varphi$ to $V(G_2)$ satisfies the trivial constraints of $G_1$.
By the condition (r) and Lemma~\ref{lemma-redu}, it follows that $G$ has an $11/4$-coloring.  This is a contradiction, since $G$ is critical.
This shows that the configuration $G_1$ cannot appear in $G$; we use
the statement ``\emph{The \stex{G_1}{S}}.'' as a shorthand for this argument.

Let us remark that in the figures illustrating either of the two arguments, the vertices of $S$ are drawn with double line.

\section{Connectivity and nailed vertices}

We now start our work on restricting the structure of minimum counterexamples.
Since a minimum counterexample is critical, it is connected.  In fact, it is straightforward
to argue it is $2$-connected, as we see next.

\begin{lemma}\label{lemma-conn}
Let $G$ be a critical subcubic e-graph.  For every clique $K$ in $G$ in $G$,
the graph $G-K$ is connected.  Consequently, $G$ is $2$-connected and if $|V(G)|>2$, then it has minimum degree at least two.
\end{lemma}
\begin{proof}
If $G-K$ is not connected, then let $G_1$ and $G_2$ be proper induced sub-e-graphs of $G$ intersecting in $K$.
By the criticality of $G$, both $G_1$ and $G_2$ have an $11/4$-coloring, and by applying a suitable measure-preserving
bijection on the $11/4$-coloring of $G_2$, we can assume that the two $11/4$-colorings assign the same color sets
to vertices of $K$.  This gives an $11/4$-coloring of $G$, which is a contradiction.
\end{proof}

We need the following observation on the coloring of paths.
\begin{observation}\label{obs-cpath}
Let $P$ be a subcubic e-graph whose underlying graph is a path $v_0\ldots v_k$ of length $k\ge 1$.
An $11/4$-coloring $\varphi$ of $v_0$ and $v_k$ extends to an $11/4$-coloring of $P$ if and only if
\begin{itemize}
\item $k$ is odd and $$|\varphi(v_0)\cap\varphi(v_k)|\le \Bigl(\sum_{i=1}^{k-1} d_P(v_i)\Bigr)-3(k-1)/2\text{, or}$$
\item $k$ is even and $$|\varphi(v_0)\cup\varphi(v_k)|\le \Bigl(\sum_{i=1}^{k-1} d_P(v_i)\Bigr)-(3k-14)/2.$$
\end{itemize}
In particular, if either $k\ge 3$ is odd and $|\varphi(v_0)\cap\varphi(v_k)|\le 1$, or
$k\ge 2$ is even and $|\varphi(v_0)\cup\varphi(v_k)|\le 6$, then $\varphi$ extends to an $11/4$-coloring of $P$.
\end{observation}
\begin{proof}
By induction on $k$.  In the basic case $k=1$, the assumption gives $|\varphi(v_0)\cap\varphi(v_1)|=0$, as needed
for the set coloring of the path.  Hence, assume that $k\ge 2$.

If $k$ is even, choose $\varphi'(v_{k-1})\subseteq [0,11)\setminus\varphi(v_k)$
of measure $7-d_P(v_{k-1})$ with $|\varphi(v_0)\cap\varphi'(v_{k-1})|$ minimum
and let $\varphi'(v_0)=\varphi(v_0)$.  By the induction hypothesis, we see that
$\varphi$ extends to $P$ if and only if $\varphi'$ extends to $P-v_k$.
Observe that
\begin{align*}
|\varphi'(v_0)\cap\varphi'(v_{k-1})|&=\max\bigl(0, (7-d_P(v_{k-1}))-(11-|\varphi(v_0)\cup\varphi(v_k)|)\bigr)\\
&=\max(0, |\varphi(v_0)\cup\varphi(v_k)|-d_P(v_{k-1})-4).
\end{align*}
If $|\varphi(v_0)\cup\varphi(v_k)|\le \bigl(\sum_{i=1}^{k-1} d_P(v_i)\bigr)-(3k-14)/2$,
then 
\begin{align*}
|\varphi(v_0)\cup\varphi(v_k)|-d_P(v_{k-1})-4&\le \Bigl(\sum_{i=1}^{k-2} d_P(v_i)\Bigr)-\frac{3k-14}{2}-4\\
&=\Bigl(\sum_{i=1}^{k-2} d_P(v_i)\Bigr)-\frac{3(k-2)}{2}.
\end{align*}
Moreover, note that $\bigl(\sum_{i=1}^{k-2} d_P(v_i)\bigr)-3(k-2)/2\ge 0$, since $d_P(v_i)\ge \deg_P v_i=2>3/2$ for $1\le i\le k-2$.
We conclude that $|\varphi'(v_0)\cap\varphi'(v_{k-1})|\le\bigl(\sum_{i=1}^{k-2} d_P(v_i)\bigr)-3(k-2)/2$,
and thus $\varphi'$ extends to an $11/4$-coloring of $P-v_k$ by the induction hypothesis.  Conversely,
if $|\varphi(v_0)\cup\varphi(v_k)|>\bigl(\sum_{i=1}^{k-1} d_P(v_i)\bigr)-(3k-14)/2$,
then
$$|\varphi'(v_0)\cap\varphi'(v_{k-1})|\ge |\varphi(v_0)\cup\varphi(v_k)|-d_P(v_{k-1})-4>\Bigl(\sum_{i=1}^{k-2} d_P(v_i)\Bigr)-\frac{3(k-2)}{2},$$
and thus $\varphi'$ does not extend to an $11/4$-coloring of $P-v_k$ by the induction hypothesis.

The case $k\ge 3$ is odd is dealt with analogously. We choose $\varphi'(v_{k-1})\subseteq [0,11)\setminus\varphi(v_k)$
of measure $7-d_P(v_{k-1})$
with $|\varphi(v_0)\cup\varphi'(v_{k-1})|$ minimum and let $\varphi'(v_0)=\varphi(v_0)$, and note that by the induction hypothesis,
$\varphi$ extends to $P$ if and only if $\varphi'$ extends to $P-v_k$.
Observe that
$$|\varphi'(v_0)\cup\varphi'(v_{k-1})|=\max\bigl(|\varphi(v_0)|,7-d_P(v_{k-1})+|\varphi(v_0)\cap\varphi(v_k)|\bigr).$$
If $|\varphi(v_0)\cap\varphi(v_k)|\le \bigl(\sum_{i=1}^{k-1} d_P(v_i)\bigr)-3(k-1)/2$, then
$$7-d_P(v_{k-1})+|\varphi(v_0)\cap\varphi(v_k)|\le \Bigl(\sum_{i=1}^{k-2} d_P(v_i)\Bigr)-\frac{3(k-1)-14}{2}.$$
Moreover, note that the expression on the right-hand side is greater or equal to
$2-(3\cdot 2-14)/2=6\ge 7-d_P(v_0)=|\varphi(v_0)|$.
We conclude that $|\varphi'(v_0)\cup\varphi'(v_{k-1})|\le\bigl(\sum_{i=1}^{k-2} d_P(v_i)\bigr)-(3(k-1)-14)/2$,
and thus $\varphi'$ extends to an $11/4$-coloring of $P-v_k$ by the induction hypothesis.  Conversely,
if $|\varphi(v_0)\cap\varphi(v_k)|>\bigl(\sum_{i=1}^{k-1} d_P(v_i)\bigr)-3(k-1)/2$, then
$|\varphi'(v_0)\cup\varphi'(v_{k-1})|>\bigl(\sum_{i=1}^{k-2} d_P(v_i)\bigr)-(3(k-1)-14)/2$,
and thus $\varphi'$ does not extend to an $11/4$-coloring of $P-v_k$ by the induction hypothesis.

The final claim follows from the fact that $d_P(v_i)\ge \deg_P v_i=2>3/2$ for $1\le i\le k-1$, and thus the right-hand
sides of the conditions from the statement of this observation are greater or equal to $1$ if $k\ge 3$ is odd
and to $6$ if $k\ge 2$ is even.
\end{proof}

By Lemma~\ref{lemma-conn}, minimum counterexamples are $2$-edge-connected.  We now
further restrict $2$-edge-cuts.

\begin{lemma}\label{lemma-3conn}
Let $G$ be a minimum counterexample, and let $\{A_1,A_2\}$ be a partition of $V(G)$ into
non-empty parts.  If $G$ contains exactly two edges between $A_1$ and $A_2$, then
there exists $i\in\{1,2\}$ such that all vertices in $A_i$ have degree two.
\end{lemma}
\begin{proof}
Suppose for a contradiction that each of $A_1$ and $A_2$ contains a vertex of degree three.  Then there exist subsets $A'_1\subseteq A_1$,
$A'_2\subseteq A_2$, and disjoint induced paths $P_1$ and $P_2$ in $G$ such that for $i\in\{1,2\}$
the path $P_i$ has ends $x_i\in A'_1$ and $y_i\in A'_2$ of degree three and is otherwise disjoint
from $A'_1\cup A'_2$, and $G=G[A'_1]\cup G[A'_2]\cup P_1\cup P_2$.  Let us remark that $x_1\neq x_2$ and $y_1\neq y_2$,
as otherwise the edge incident with $x_1$ or $y_1$ not contained in $P_1\cup P_2$ would form a bridge in $G$.

Let $G_1$ be the e-graph obtained from $G[A'_1]$ by adding a path
$x_1abx_2$, with $d_{G_1}(a)=d_{G_1}(b)=3$, and let $G_2$ be defined analogously.
Suppose first that both $G_1$ and $G_2$ are $11/4$-colorable, and let $\varphi_1$ and $\varphi_2$ be the
restriction of their $11/4$-colorings to $A'_1$ and $A'_2$, respectively.  By Observation~\ref{obs-constraints} (or Observation~\ref{obs-cpath}),
we have $|\varphi_1(x_1)\cap \varphi_1(x_2)|\le 3$ and $|\varphi_2(y_1)\cap \varphi_2(y_2)|\le 3$.

Observe that by applying a measure-preserving bijection to $\varphi_2$, we can without loss of generality assume
any one of the following three conditions holds, as needed:
\begin{itemize}
\item[(i)] $|\varphi_1(x_1)\cup \varphi_2(y_1)|\le 6$ and $|\varphi_1(x_2)\cup \varphi_2(y_2)|\le 6$;
\item[(ii)] $|\varphi_1(x_1)\cap \varphi_2(y_1)|=0$ and $|\varphi_1(x_2)\cup \varphi_2(y_2)|\le 6$;
\item[(iii)] $|\varphi_1(x_1)\cap \varphi_2(y_1)|=0$ and $|\varphi_1(x_2)\cap \varphi_2(y_2)|=0$.
\end{itemize}
\begin{subproof}
We show how to enforce the condition (ii); the conditions (i) and (iii) are dealt with similarly.
For $r\in\{0,3\}$, let $\psi_{1,r}$ be the set $11$-coloring of $\{x_1,x_2\}$ defined by
$\psi_{1,r}(x_1)=[0,4)$ and $\psi_{1,r}(x_2)=[4-r,8-r)$, so that $|\psi_{1,r}(x_1)\cap \psi_{1,r}(x_2)|=r$.
Letting $\lambda_1=|\varphi_1(x_1)\cap \varphi_1(x_2)|/3$, we can without loss of generality assume
$\lambda_1\psi_{1,3}+(1-\lambda_1)\psi_{1,0}$ is the restriction of $\varphi_1$ to $\{x_1,x_2\}$.
Define $\psi_{2,0}$, $\psi_{2,3}$ and $\lambda_2$ analogously so that
$\lambda_2\psi_{2,3}+(1-\lambda_2)\psi_{2,0}$ is the restriction of $\varphi_2$ to $\{y_1,y_2\}$.

Observe that for $r_1,r_2\in\{0,3\}$, there exists a set $11$-coloring $\psi_{2,r_2,(ii),r_1}$ of $\{y_1,y_2\}$
obtained by applying a measure-preserving bijection to $\psi_{2,r_2}$ and such that
$|\psi_{1,r_1}(x_1)\cap \psi_{2,r_2,(ii),r_1}(y_1)|=0$ and $|\psi_{1,r_1}(x_2)\cup \psi_{2,r_2,(ii),r_1}(y_2)|\le 6$.
Indeed, we can let
\begin{itemize}
\item $\psi_{2,0,(ii),0}(y_1)=[7,11)$ and $\psi_{2,0,(ii),0}(y_2)=[3,7)$,
\item $\psi_{2,3,(ii),0}(y_1)=[5,9)$ and $\psi_{2,3,(ii),0}(y_2)=[4,8)$,
\item $\psi_{2,0,(ii),3}(y_1)=[7,11)$ and $\psi_{2,0,(ii),3}(y_2)=[1,5)$,
\item $\psi_{2,3,(ii),3}(y_1)=[4,8)$ and $\psi_{2,3,(ii),3}(y_2)=[3,7]$.
\end{itemize}
Now let $\alpha_{0,0}=(1-\lambda_1)(1-\lambda_2)$, $\alpha_{0,3}=(1-\lambda_1)\lambda_2$,
$\alpha_{3,0}=\lambda_1(1-\lambda_2)$ and $\alpha_{3,3}=\lambda_1\lambda_2$.
Then
$$\psi_1=\sum_{r_1,r_2\in\{0,3\}} \alpha_{r_1,r_2}\psi_{1,r_1}=\lambda_1\psi_{1,3}+(1-\lambda_1)\psi_{1,0}$$
is the restriction of $\varphi_1$ to $\{x_1,x_2\}$, and
$$\psi_2=\sum_{r_1,r_2\in\{0,3\}} \alpha_{r_1,r_2}\psi_{2,r_2,(ii),r_1}$$
is obtained by applying a measure-preserving bijection $\theta$ to
$$\sum_{r_1,r_2\in\{0,3\}} \alpha_{r_1,r_2}\psi_{2,r_2}=\lambda_2\psi_{2,3}+(1-\lambda_2)\psi_{2,0},$$
which is the restriction of $\varphi_2$ to $\{y_1,y_2\}$.  Moreover, by linearity,
we have
\begin{align*}
|\varphi_1(x_1)\cap \theta^{-1}(\varphi_2(y_1))|&=|\psi_1(x_1)\cap \psi_2(y_1)|\\
&=\sum_{r_1,r_2\in\{0,3\}} \alpha_{r_1,r_2}|\psi_{1,r_1}(x_1)\cap \psi_{2,r_2,(ii),r_1}(y_1)|=0\\
|\varphi_1(x_2)\cup \theta^{-1}(\varphi_2(y_2))|&=|\psi_1(x_2)\cup \psi_2(y_2)|\\
&=\sum_{r_1,r_2\in\{0,3\}} \alpha_{r_1,r_2}|\psi_{1,r_1}(x_2)\cup \psi_{2,r_2,(ii),r_1}(y_2)|\le 6.
\end{align*}
Hence, (ii) holds after applying the measure-preserving bijection $\theta^{-1}$ to~$\varphi_2$.
\end{subproof}
We can by symmetry assume that if $P_1$ has even length, then so does $P_2$.
By applying a measure-preserving bijection if necessary, we can without loss
of generality assume that if $P_1$ and $P_2$ both have even length, then (i) holds, if $P_1$ has odd length and $P_2$ has even length then (ii) holds,
and if $P_1$ and $P_2$ both have odd length, then (iii) holds.
But then for $i\in \{1,2\}$, Observation~\ref{obs-cpath} implies that $\varphi_1\cup \varphi_2$ extends to an $11/4$-coloring
of $P_i$, thus giving an $11/4$-coloring of $G$.  This is a contradiction.

Hence, we can by symmetry assume $G_1$ does not have an $11/4$-coloring,
and thus $G_1$ contains a critical induced sub-e-graph $F_1$.
Recall that since $G$ is $2$-connected, we have $x_1\neq x_2$, and thus $G_1$ (and consequently also $F_1$) is triangle-free.
Note that since $\deg y_1=3$ and $\deg y_2=3$, we have $|A'_2|>2$, and thus $|V(F_1)|\le |V(G_1)|<|V(G)|$.
Since $G$ is a minimum counterexample, it follows that $F_1\in\CC_0$.  Since $G$ is critical, $F_1$ is not an induced sub-e-graph of $G$,
and thus $a,b\in V(F_1)$.  Hence $F_1$ contains two nailed vertices $a$ and $b$, and by (a0) no other vertex of $F_1$ is nailed
and the underlying graph of $F_1$ is either $C_5$ or \kk.  Since no vertex of $F_1$ other than $a$ or $b$ is nailed and $G$ is connected,
we conclude using Observation~\ref{obs-ncnail} that $G_1=F_1$.  Furthermore, since $\deg x_1=3$, $G_1$ cannot be a $5$-cycle, and thus the underlying graph of $G_1$ is \kk.

\begin{figure}
\begin{center}
\begin{tikzpicture}[scale=0.8]
\draw (0,0) node[vtx,label=below:$x_2$](a){} 
 ++(0,3)node[vtx,label=above:{$x_1$}](x){}
-- (1.5,1.5) node[vtx](e){}--(a) -- (-1.5,1.5) node[vtx](f){}--(x)
(e) to
node[vtx,pos=0.33](y){}
node[vtx,pos=0.66](d){}
(f)
(a) -- ++(1.5,0) node[vtxS,label=below:$y'_2$]{}
(x) -- ++(1.5,0) node[vtxS,label=above:$y'_1$]{}
;

\begin{scope}[xshift=5cm]
\draw
(0,0) node[vtxS,label=below:$y'_2$](a){} 
 ++(0,3)node[vtxS,label=above:{$y'_1$}](x){}
 (-1,1.5) node[nail,label=left:$z$](z){}
 (a)--(z)--(x)
 ;
\end{scope}

\end{tikzpicture}
\end{center}
\caption{The reducible configuration $G_3$ and the replacement graph $H_3$.}\label{fig-3conn-1}
\end{figure}
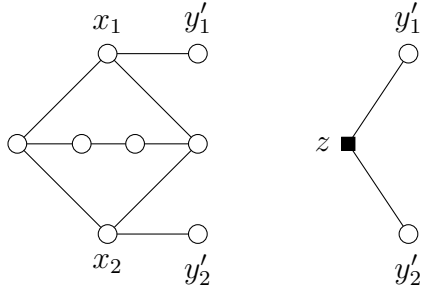

For $i\in\{1,2\}$, let $y'_i$ be the neighbor of $x_i$ in $P_i$.
Let $G_3$ be the e-graph obtained from $G[A'_1]$ by adding the edges $x_1y'_1$ and $x_2y'_2$.
Since $G_1$ is \kk{} with only $a$ and $b$ nailed, $G_3$ is the e-graph depicted in Figure~\ref{fig-3conn-1}.
Let $S_3=\{y'_1,y'_2\}$ and let $H_3$ be the e-graph with the vertex set $S_3\cup\{z\}$,
edges $y'_1z$ and $y'_2z$, and $d_{H_3}(z)=3$.
The \starg{$H_3$ enforcing $|\varphi(y'_1)\cup \varphi(y'_2)|\le 7$}{G_3}{S_3}{G'}{F_3}
(note that by Lemma~\ref{lemma-conn}, we have $y'_1y'_2\not\in E(G)$, and thus $G'$ is triangle-free).
Since $G$ is critical, $F_3$ is not an induced sub-e-graph of $G$, and thus $z\in V(F_3)$.
However, Lemma~\ref{lemma-conn} implies that $G'$ is $2$-connected, and since $F_3\neq G'$, it follows
by Observation~\ref{obs-ncnail} that $F_3$ contains at least two nailed vertices
(those with a neighbor in $V(G')\setminus V(F_3)$) in addition to $z$, which contradicts (a0).
\end{proof}

Let us now derive two simple consequences of Lemma~\ref{lemma-3conn}.

\begin{figure}
\begin{center}
\begin{tikzpicture}[scale=0.8]
\draw (0,0) node[vtx](a){} -- (0,1) node[vtx](b){} -- ++(0,1)node[vtx](c){}
-- ++(0,1)node[vtx](d){}
-- (1.5,1.5) node[vtx](e){}--(a) -- (-1.5,1.5) node[vtx](f){}--(d)
(e) to[bend right=90,looseness=2] 
node[nail,pos=0.3](d){}
node[vtx,pos=0.7](d){}
(f)
;
\draw[dashed]
(b) to[out=200,in=200,looseness=2.5] 
  node[vtx,pos=0.5]{} 
 (d)
;
\begin{scope}
\clip ($(b)!0.5!(d)$) circle(1.7cm);
\draw
(b) to[out=200,in=200,looseness=2.5] 
  node[vtx,pos=0.2]{} 
 (d)
;
\end{scope}
\end{tikzpicture}
\hskip 1em
\begin{tikzpicture}[scale=0.8]
\draw
\foreach \x  in {2,3,4,5}{
(18+72*\x:1.5) node[vtx](x\x){}
}
(18+72*1:1.5) node[nail](x1){}
(x1)--(x2)--(x3)--(x4)--(x5)--(x1)
;
\draw[dashed]
(x2) to[out=-30,in=210,looseness=1.5] 
  node[vtx,pos=0.6]{} 
 (x5)
;
\begin{scope}
\clip (x2) circle(1.2cm);
\draw
(x2) to[out=-30,in=210,looseness=1.5] 
  node[vtx,pos=0.2]{} 
 (x5)
;
\end{scope}
\begin{scope}
\clip (x5) circle(0.9cm);
\draw
(x2) to[out=-30,in=210,looseness=1.5] (x5);
\end{scope}
\end{tikzpicture}
\hskip 1em
\begin{tikzpicture}[scale=0.8]
\draw
\foreach \x  in {1,2,4,5}{
(18+72*\x:1.5) node[vtx](x\x){}
}
(18+72*3:1.5) node[nail](x3){}
(x1)--(x2)--(x3)--(x4)--(x5)--(x1)
;
\draw[dashed]
(x2) to[out=-30,in=210,looseness=1.5] 
  node[vtx,pos=0.6]{} 
 (x5)
;
\begin{scope}
\clip (x2) circle(1.2cm);
\draw
(x2) to[out=-30,in=210,looseness=1.5] 
  node[vtx,pos=0.2]{} 
 (x5)
;
\end{scope}
\begin{scope}
\clip (x5) circle(0.9cm);
\draw
(x2) to[out=-30,in=210,looseness=1.5] (x5);
\end{scope}
\end{tikzpicture}
\end{center}
\caption{The graphs arising in the final case analysis in Lemma~\ref{lemma-nails}.}\label{fig-nailfinal}
\end{figure}

\begin{lemma}\label{lemma-nails}
A minimum counterexample has no nailed vertices.
\end{lemma}
\begin{proof}
Suppose for a contradiction that $v$ is a nailed vertex of a minimum counterexample $G$, and thus $\deg v=2$ and $d_G(v)=3$.
Let $G'$ be the e-graph obtained from $G$ by setting $d_{G'}(v)=2$.  Since $G$ is not $11/4$-colorable,
$G'$ is not $11/4$-colorable, and thus it contains a critical induced sub-e-graph $F$.
Note that $|V(F)|\le |V(G)|$ and $F$ has fewer nailed vertices than $G$, and thus the minimality of $G$ implies $F\in \CC_0$.  By a computer-assisted enumeration, we verified
that for every e-graph in $\CC_0$, nailing a vertex results in an e-graph that either is not critical
or belongs to $\CC_0$.  Since $G\not\in \CC_0$ is critical, it is not obtained by nailing a vertex of $F\in\CC_0$, and thus $F\neq G'$.  Since $G$ is $2$-edge-connected
by Lemma~\ref{lemma-conn}, at least two vertices of $F$ have a neighbor in $V(G)\setminus V(F)$,
and thus at least two vertices of $F$ are nailed.  By (a0), exactly two vertices $v_1$ and $v_2$ of $F$ are nailed
and the underlying graph of $F$ is $C_5$ or \kk.  By Lemma~\ref{lemma-conn}, we have $v_1v_2\not\in E(G)$.
By Lemma~\ref{lemma-3conn}, $G$ is obtained from $F$ by adding a path $P$ of vertices of degree two between $v_1$ and $v_2$
and nailing $v$.  However, a straightforward case analysis shows that
all valid e-graphs arising from $C_5$ or \kk{} in this way and satisfying the conclusion of Observation~\ref{obs-3nails}
(depicted in Figure~\ref{fig-nailfinal}) are $11/4$-colorable, which is a contradiction.
Let us remark that although the added path $P$ may have arbitrary length and its vertices can be nailed arbitrarily,
Observation~\ref{obs-cpath} implies that to verify the $11/4$-colorability of the e-graphs depicted in Figure~\ref{fig-nailfinal},
it suffices to consider the case that $P$ has length two or three and its vertices are not nailed.
\end{proof}

\newcommand{\baseKfour}{
\clip(-2.4,-0.3) rectangle (1.8,3.7);
\draw (0,0) node[vtx](a){} -- (0,1) node[vtx](b){} -- ++(0,1)node[vtx](c){}
-- ++(0,1)node[vtx](x){}
-- (1.5,1.5) node[vtx](e){}--(a) -- (-1.5,1.5) node[vtx](f){}--(x)
(e) to[bend right=90,looseness=2] 
node[vtx,pos=0.3](y){}
node[vtx,pos=0.7](d){}
(f)
;
\draw[dashed]
(b) to[out=200,in=200,looseness=2.5] 
  node[vtx,pos=0.5]{} 
 (d)
;
\begin{scope}
\clip ($(b)!0.5!(d)$) circle(1.7cm);
\draw
(b) to[out=200,in=200,looseness=2.5] 
  node[vtx,pos=0.2]{} 
 (d)
;
\end{scope}}
\begin{figure}
\begin{center}
\begin{tikzpicture}[scale=0.7]
\baseKfour
\draw  (e) to[bend right=90,looseness=2] 
node[vtx,pos=0.15]{}
node[vtx,pos=0.3]{}
node[vtx,pos=0.5]{}
node[vtx,pos=0.7]{}
(f);
\end{tikzpicture}
\hskip 1em
\begin{tikzpicture}[scale=0.7]
\baseKfour
\draw  (e) to[bend right=90,looseness=2] 
node[vtx,pos=0.92]{}
node[vtx,pos=0.3]{}
node[vtx,pos=0.82]{}
node[vtx,pos=0.7]{}
(f);
\end{tikzpicture}
\hskip 1em
\begin{tikzpicture}[scale=0.7]
\baseKfour
\draw  (e) to[bend right=90,looseness=2]  node[vtx,pos=0.3]{} node[vtx,pos=0.7]{} (f);
\draw (a) -- node[vtx,pos=0.3]{}  node[vtx,pos=0.7]{} (e) ;
\end{tikzpicture}
\\
%
\begin{tikzpicture}[scale=0.7]
\baseKfour
\draw  (e) to[bend right=90,looseness=2]  node[vtx,pos=0.3]{} node[vtx,pos=0.7]{} (f);
\draw (a) -- node[vtx,pos=0.3]{}  node[vtx,pos=0.8]{} (f) ;
\end{tikzpicture}
\hskip 1em
\begin{tikzpicture}[scale=0.7]
\baseKfour
\draw  (e) to[bend right=90,looseness=2]  node[vtx,pos=0.3]{} node[vtx,pos=0.7]{} (f);
\draw (e) -- node[vtx,pos=0.3]{}  node[vtx,pos=0.7]{} (x) ;
\end{tikzpicture}
\hskip 1em
\begin{tikzpicture}[scale=0.7]
\clip(-2.4,-0.3) rectangle (1.8,3.7);
\draw (0,0) node[vtx](a){} 
--    node[vtx,pos=0.2](b){}  node[vtx,pos=0.4](c){} node[vtx,pos=0.6](bb){}  node[vtx,pos=0.8](cc){}   (0,3)node[vtx](x){}
-- (1.5,1.5) node[vtx](e){}--(a) -- (-1.5,1.5) node[vtx](f){}--(x)
(e) to[bend right=90,looseness=2] 
node[vtx,pos=0.3](y){}
node[vtx,pos=0.7](d){}
(f)
;
\draw
(b) to[bend left=50,looseness=2] 
  node[vtx,pos=0.8]{} 
 (cc)
;
\begin{scope}
\clip (b) circle(1.1cm);
\draw[dashed]
(b) to[bend left=50,looseness=2] 
  node[vtx,pos=0.2]{} 
 (cc)
;
\end{scope}
\end{tikzpicture}
\\
\begin{tikzpicture}[scale=0.7]
\clip(-1.8,-1.5) rectangle (1.8,1.9);
\draw
\foreach \x  in {1,2,3,4,5}{
(18+72*\x:1.5) node[vtx](x\x){}
}
(x1)--(x2)--(x3)--(x4)--(x5)--(x1)
;
\draw[dashed]
(x2) to[out=-30,in=210,looseness=1.5] 
  node[vtx,pos=0.6]{} 
 (x5)
;
\begin{scope}
\clip (x2) circle(1.2cm);
\draw
(x2) to[out=-30,in=210,looseness=1.5] 
  node[vtx,pos=0.2]{} 
 (x5)
;
\end{scope}
\begin{scope}
\clip (x5) circle(0.9cm);
\draw
(x2) to[out=-30,in=210,looseness=1.5] (x5);
\end{scope}
\draw (x3) -- node[vtx,pos=0.3]{}  node[vtx,pos=0.7]{} (x4) ;
\end{tikzpicture}
\hskip 1em
\begin{tikzpicture}[scale=0.7]
\clip(-1.8,-1.5) rectangle (1.8,1.9);
\draw
\foreach \x  in {1,2,3,4,5}{
(18+72*\x:1.5) node[vtx](x\x){}
}
(x1)--(x2)--(x3)--(x4)--(x5)--(x1)
;
\draw[dashed]
(x2) to[out=-30,in=210,looseness=1.5] 
  node[vtx,pos=0.6]{} 
 (x5)
;
\begin{scope}
\clip (x2) circle(1.2cm);
\draw
(x2) to[out=-30,in=210,looseness=1.5] 
  node[vtx,pos=0.2]{} 
 (x5)
;
\end{scope}
\begin{scope}
\clip (x5) circle(0.9cm);
\draw
(x2) to[out=-30,in=210,looseness=1.5] (x5);
\end{scope}
\draw (x1) -- node[vtx,pos=0.3]{}  node[vtx,pos=0.7]{} (x2) ;
\end{tikzpicture}
\hskip 1em
\begin{tikzpicture}[scale=0.7]
\clip(-1.8,-1.5) rectangle (1.8,1.9);
\draw
\foreach \x  in {1,2,3,4,5}{
(18+72*\x:1.5) node[vtx](x\x){}
}
(x1)--(x2)--(x3)--(x4)--(x5)--(x1)
;
\draw[dashed]
(x3) to[out=70,in=110,looseness=1.5] 
  node[vtx,pos=0.6]{} 
 (x4)
;
\begin{scope}
\clip (x3) circle(1.2cm);
\draw
(x3) to[out=70,in=110,looseness=1.5] 
  node[vtx,pos=0.2]{} 
 (x4)
;
\end{scope}
\begin{scope}
\clip (x4) circle(0.9cm);
\draw
(x3) to[out=70,in=110,looseness=1.5] (x3);
\end{scope}
\draw (x3) -- node[vtx,pos=0.3]{}  node[vtx,pos=0.7]{} (x4) ;
\end{tikzpicture}

\end{center}
\caption{The graphs arising in the final case analysis in Lemma~\ref{lemma-adj2}.}\label{fig-adj2}
\end{figure}

\begin{lemma}\label{lemma-adj2}
Let $G$ be a minimum counterexample and let $uv\in E(G)$.  If $\deg u = \deg v=2$, then
$u$ and $v$ are contained in a $5$-cycle.  Consequently, $G$ does not contain a path
of three vertices of degree two.
\end{lemma}
\begin{proof}
Suppose for a contradiction that $uv$ is not an edge of a $5$-cycle in $G$.  Let $u',v'\not\in\{u,v\}$ be the other neighbors of $u$ and $v$, respectively.
Then the e-graph $G'=G-\{u,v\}+u'v'$ is triangle-free.  Every $11/4$-coloring $\varphi$ of $G'$ would satisfy
$\varphi(u')\cap \varphi(v')=\emptyset$, and thus by Observation~\ref{obs-cpath}, it would extend to an $11/4$-coloring of $G$.
Therefore, $G'$ is not $11/4$-colorable, and thus $G'$ contains
a critical induced sub-e-graph $F$.  Since $G$ is critical, $F$ cannot be an induced sub-e-graph of $G$, and thus $u'v'\in E(F)$.

The minimality of $G$ implies $F\in \CC_0$.  By a computer-assisted enumeration, we verified
that for every e-graph in $\CC_0$, replacing an edge by a path of two vertices of degree two
results in an e-graph that either is not critical or belongs to $\CC_0$.  Since $G\not\in \CC_0$ is critical,
we have $F\neq G'$.  By Observation~\ref{obs-ncnail} and Lemma~\ref{lemma-conn}, $F$ contains at least two nailed vertices.
By (a0), it follows that $F$ contains exactly two nailed vertices $x$ and $y$ and the underlying graph of $F$ is $C_5$ or \kk.
By Lemma~\ref{lemma-3conn}, $G$ is obtained from $F$ by replacing an edge by a path of two vertices of degree two
and adding a path of vertices of degree two between $x$ and $y$.  By Lemma~\ref{lemma-conn}, we have $xy\not\in E(G)$.
A straightforward case analysis shows that none of such e-graphs (depicted in Figure~\ref{fig-adj2}) is critical, which is a contradiction.

If $G$ contained a path of three vertices of degree two, then these vertices would be contained in a 5-cycle $K$,
and by Lemma~\ref{lemma-conn}, we would have $G=K$.  However, all critical e-graphs whose underlying graph is a 5-cycle belong to~$\CC_0$.
\end{proof}

Combining Lemmas~\ref{lemma-3conn} and \ref{lemma-adj2}, we obtain the following.
\begin{corollary}\label{cor-3conn}
Let $G$ be a minimum counterexample, and let $\{A_1,A_2\}$ be a partition of $V(G)$ to
non-empty parts.  If $G$ contains exactly two edges between $A_1$ and $A_2$, then
there exists $i\in\{1,2\}$ such that all vertices in $A_i$ have degree two and $|A_i|\le 2$.
\end{corollary}

\section{Two neighbors of degree two}

Next, we want to show that every vertex has at most one neighbor of degree two.
Let us start by excluding the special case of a vertex with two neighbors of degree two contained in a common 5-cycle
(together with some related configurations).
Notice that Observation~\ref{obs-3nails} and Lemma~\ref{lemma-nails} imply that every $5$-cycle contains at most two vertices of degree two.

\begin{lemma}\label{lemma-5c3nailsdeg2}
Let $G$ be a minimum counterexample and let $C$ be a $5$-cycle in $G$ with exactly three vertices of degree three.
Then the vertices of degree two in $C$ are adjacent and every vertex $v\in V(G)\setminus V(C)$ with a neighbor in $C$ has degree three.
\end{lemma}
\begin{proof}
Firstly, we claim that every vertex $v\in V(G)\setminus V(C)$ of degree two with a neighbor $v_1$ in $C$
has exactly one neighbor in $C$ and both neighbors of $v$ have degree three.
\begin{subproof}
Indeed, let $v'$ be the neighbor of $v$ distinct from $v_1$.  By Lemma~\ref{lemma-conn},
$v'\not\in V(C)$, since if $v'$ belonged to $C$, then the degree three vertex of $C$ distinct from $v_1$ and $v'$
would be incident with a bridge.  Suppose for a contradiction that $\deg v'=2$.
Then $v'$ does not have a neighbor in $C$, since $G$ is $2$-connected and exactly three vertices of $C$
have degree three.  By Lemma~\ref{lemma-adj2}, $v'$ has a neighbor $x$
of degree three, and $x$ and $v_1$ have a common neighbor $v_2$ belonging to $C$.
Let $w$ be the degree three vertex of $C$ distinct from $v_1$ and $v_2$.  By Corollary~\ref{cor-3conn},
$x$ and $w$ are joined by an edge or by a path of at most two vertices of degree two.
This leaves only finitely many choices for the e-graph $G$.
As we verified using computer, among these e-graphs, all critical ones are contained in $\CC_0$.
This is a contradiction, implying that $\deg v'=3$.
\end{subproof}

Let $C=v_1v_2v_3v_4v_5$, where $\deg v_1=\deg v_2=3$.  For each vertex $v_i$ of $C$ of degree three, let $v'_i$ denote
the neighbor of $v_i$ outside of $C$, and if $\deg v'_i=2$, then let $v''_i$ be the neighbor of $v'_i$ distinct from $v_i$.
Let $Q$ consist of $V(C)$ and all vertices of degree two with a neighbor in $V(C)$.  Let $G_1$ be the sub-e-graph of $G$ consisting
of all edges incident with $Q$ and all vertices incident with these edges, and let $S=V(G_1)\setminus Q$.

If $\deg v_4=3$ and $\deg v'_1=2$ or $\deg v'_2=2$, then the \stex{G_1}{S}.
If $\deg v_4=\deg v'_1=\deg v'_2=3$, then let $H_1$ be the e-graph with the vertex set $S\cup\{z_1,z_2\}$, edges $v'_1z_1$, $z_1z_2$, and $z_2v'_2$,
and $d_{H_1}(z_1)=2$ and $d_{H_1}(z_2)=3$.
The \starg{$H_1$ enforcing $|\varphi(v'_1)\cap \varphi(v'_2)|\le 2$}{G_1}{S}{G'}{F_1}.
Since $G$ is critical, $F_1$ is not an induced sub-e-graph of $G$, and thus $z_1,z_2\in V(F_1)$.  Furthermore, Corollary~\ref{cor-3conn} implies that $G'$ is $2$-connected, and since $F_1\neq G'$,
Observation~\ref{obs-ncnail} implies that $F_1$ has at least two nailed vertices in addition to $z_2$, which contradicts (a0).

Hence, we can assume $\deg v_4=2$ and by symmetry $\deg v_3=3$.  Suppose now for a contradiction
that at least one of the vertices $v'_1$, $v'_2$, and $v'_3$ has
degree two.  For $i\in \{1,2,3\}$, let $w_i=v'_i$ if $\deg v'_i=3$ and $w_i=v''_i$ otherwise,
so that $S=\{w_1,w_2,w_3\}$.
If say $w_1=w_2$, then by Corollary~\ref{cor-3conn} applied to the cut formed by the edge incident with $w_1$ not belonging to $G_1$
and the edge $v_3v'_3$, we also have $w_2=w_3$ and $G$ consists of a 5-cycle and three paths of length at most two to $w_1$;
but a straightforward case analysis shows that all such critical e-graphs belong to $\CC_0$.
Hence, we can by symmetry assume $w_1$, $w_2$, and $w_3$ are pairwise distinct.
Furthermore, we can by symmetry assume $\deg v'_1\ge\deg v'_3$.
Let $H_2$ be the e-graph with the vertex set $S\cup\{z_1,z_2\}$, edges $w_1z_1$, $z_1z_2$, and $z_2w_i$,
where $i=3$ if $\deg v'_2=2$ and $i=2$ otherwise, and $d_{H_2}(z_1)=2$ and $d_{H_2}(z_2)=3$.
The \starg{$H_2$ enforcing $|\varphi(w_1)\cap \varphi(w_i)|\le 2$}{G_1}{S}{G''}{F_2}.
Clearly $z_1,z_2\in V(F_2)$.  Observe that Corollary~\ref{cor-3conn} implies that $G''$ is $2$-connected.
Since $F_2\neq G''$,
Observation~\ref{obs-ncnail} implies that $F_2$ has at least two nailed vertices in addition to $z_2$, which contradicts (a0).
\end{proof}

We now proceed to excluding a vertex with two neighbors of degree two in general.

\begin{lemma}\label{lemma-2deg2}
No vertex in a minimum counterexample has two neighbors of degree two.
\end{lemma}
\begin{proof}
Let $G$ be a minimum counterexample and let $uvw$ be a path in $G$, and suppose for a contradiction that $\deg u=\deg w=2$.
Let $u'$ and $w'$ be the neighbors of $u$ and $w$ distinct from $v$.  By Lemma~\ref{lemma-adj2},
we have $\deg v=3$; let $z$ be the neighbor of $v$ distinct from $u$ and $w$.  Note that $u'\neq w'$,
as otherwise $u$ and $w$ would have the same neighborhood, contradicting the criticality of~$G$.
Moreover, $u'w'\not\in E(G)$ by Lemma~\ref{lemma-5c3nailsdeg2}.  We claim that $\deg z=3$.
\begin{subproof}
Suppose for a contradiction that $\deg z=2$, and let $z'$ be the neighbor of $z$ distinct from $v$.
The argument from the previous paragraph implies that $u'\neq z'\neq w'$ and $\{u',w',z'\}$ is an independent set.
Note that $\deg u'=\deg w'=\deg z'=3$ by Lemma~\ref{lemma-adj2} and Observation~\ref{obs-3nails}.
Let $\vec{Z}$ be the auxiliary directed graph with the vertex set $\{u',w',z'\}$,
where for distinct vertices $x,y\in V(\vec{Z})$, we have $(x,y)\in E(\vec{Z})$ if and only if $x$ is contained in a unique $4$-cycle $xrts\subseteq G-\{u,w,z,v\}$,
$r$ and $s$ are joined by a path of length three, $\{r,s,t\}\cap \{u',w',z'\}=\emptyset$, and $ty\in E(G)$;
see Figure~\ref{fig-2deg2-1} for an illustration.
Note that for a fixed vertex $x$, such a vertex $y$ is uniquely determined by these conditions, and thus $\vec{Z}$ has maximum outdegree at most one.
Furthermore, observe that $\vec{Z}$ contains at most one of the edges $(u',w')$ and $(z',u')$.
Therefore, we can by symmetry assume $(u',w'),(w',u')\not\in E(\vec{Z})$.

\begin{figure}
\begin{center}
\begin{tikzpicture}
\draw 
(0,0) node[vtx,label=right:$v$](v){}
--++(90:1)node[vtx,label=right:$w$](w){}
--++(90:1)node[vtx,label=right:$w'$](w'){}
(v)
--++(-30:1.5)node[vtx,label=right:$u$](u){}
--++(270:1)node[vtx,label=right:{$u'=y$}](u'){}
(v)
--++(210:1.5)node[vtx,label=left:$z$](z){}
--++(270:1)node[vtx,label=left:{$z'=x$}](z'){}
(z') --++(45:1) node[vtx,label=above:$r$](r){}
(z') --++(-45:1) node[vtx,label=below:$s$](s){}
--++(45:1) node[vtx,label=below:$t$](t){} --(r)
(t)--(u')
(w') -- ++(45:0.5)
(w') -- ++(135:0.5)
(u') -- ++(270:0.5)
(s) --
node[pos=0.333,vtx](a){}
node[pos=0.666,vtx](b){}
 (r) ++(0:0.5)
 ;
 \draw[dotted]
(a) -- ++(0:0.3)
(b) -- ++(0:0.3)
;
\draw (0,-3) node{$G$};

\begin{scope}[xshift=7cm]
\draw 
(90:2)node[vtx,label=right:$w'$](w'){}
(-30:2)node[vtx,label=right:$u'$](u'){}
(210:2)node[vtx,label=left:$z'$](z'){}
;
\draw[-latex](z')--(u');
\draw (0,-3) node{$\vec{Z}$};
\end{scope}
\end{tikzpicture}
\end{center}
\caption{An edge of the auxiliary graph $\vec{Z}$.}\label{fig-2deg2-1}
\end{figure}
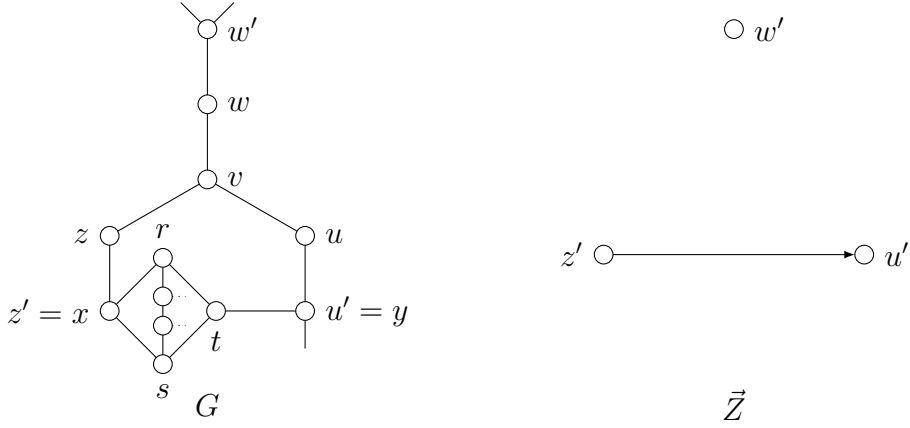

Let $S_0=\{u',w',z'\}$ and let $G_0$ be the sub-e-graph of $G$ with the vertex set $S_0\cup\{u,v,w,z\}$
and all edges incident with $u,v,w,z$.  Let $H_0$ be the e-graph with the vertex set $S_0\cup \{c\}$,
with $c$ adjacent to $u'$ and $w'$ and with $d_{H_0}(c)=2$.
The \starg{$H_0$ enforcing $|\varphi(u')\cup\varphi(w')|\le 6$}{G_0}{S_0}{G''}{F_0}.
Note that since $G$ is critical, $F_0$ is not a proper induced sub-e-graph of $G$, and consequently $c\in V(F_0)$.
Note that Lemma~\ref{lemma-conn} implies that $G''$ is $2$-connected,
and thus $F_0$ contains at least two nailed vertices by Observation~\ref{obs-ncnail}.  By (a0),
it follows that the underlying graph of $F_0$ is either $C_5$ or \kk.
\begin{itemize}
\item In the former case, $F_0=u'cw'ab$ for some vertices $a$ and $b$, $u'$ and $w'$ are nailed in $F_0$
since they have degree three in $G$, and $\deg_G a = \deg_G b = 2$.  By Lemma~\ref{lemma-adj2}, the path $u'abw'$ is contained in a $5$-cycle in $G$,
and since $\deg u=2$, this $5$-cycle contradicts Lemma~\ref{lemma-5c3nailsdeg2}.
\item In the latter case, since the underlying graph of $F_0$ is \kk{} and $\deg_{F_0} c=2$, we can by symmetry assume $\deg_{F_0} w'=2$ (and thus $w'$ is nailed in $F_0$)
and $\deg_{F_0} u'=3$.  Then $u'$ is contained in a $4$-cycle of non-nailed vertices of $F_0$, necessarily distinct from the vertex $z'$ which is nailed in $G_0$.
This 4-cycle together with the rest of $F_0-c$ provides the structure implying that $(u',w')\in E(\vec{Z})$, which is a contradiction.
\end{itemize}
\end{subproof}

Therefore, we have $\deg z=3$.  Suppose now that $\deg u'=2$.  By Lemma~\ref{lemma-adj2}, $uu'$ is contained in a $5$-cycle $C$;
since $u'w'\not\in E(G)$, $C$ contains the path $u'uvz$.  By Observation~\ref{obs-3nails}, $C$ contains exactly three vertices of
degree three.  However, $v\in V(C)$ is adjacent to a vertex $w\not\in V(C)$ of degree two,
contradicting Lemma~\ref{lemma-5c3nailsdeg2}.  Therefore $\deg u'=3$, and symmetrically $\deg w'=3$.

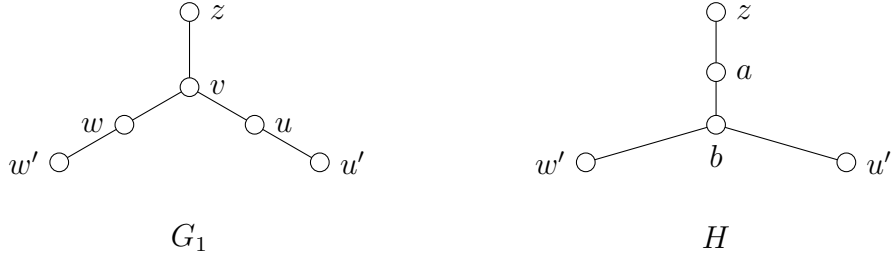
\begin{figure}
\begin{center}
\begin{tikzpicture}
\draw 
(0,0) node[vtx,label=right:$v$](v){}
--++(90:1)node[vtxS,label=right:$z$](z){}
(v)
--++(-30:1)node[vtx,label=right:$u$](u){}
--++(-30:1)node[vtxS,label=right:$u'$](u'){}
(v)
--++(210:1)node[vtx,label=left:$w$](w){}
--++(210:1)node[vtxS,label=left:$w'$](w'){}
 ;
\draw (0,-2) node{$G_1$};
\begin{scope}[xshift=7cm]
\draw 
(90:1)node[vtxS,label=right:$z$](z){}
(-30:2)node[vtxS,label=right:$u'$](u'){}
(210:2)node[vtxS,label=left:$w'$](w'){}
(270:0.5) node[vtx,label=below:$b$](b){}
(90:0.2) node[vtx,label=right:$a$](a){}
(w')--(b)--(a)--(z) (b)--(u')
;
\draw (0,-2) node{$H$};
\end{scope}
\end{tikzpicture}
\end{center}
\caption{The e-graph $G_1$ and the replacement e-graph $H$.}\label{fig-2deg2-2}
\end{figure}
Let $S=\{u',w',z\}$ and let $G_1$ be the sub-e-graph of $G$ induced by
$S\cup\{u,v,w\}$.  Let $H$ be the e-graph with the vertex set $S\cup\{a,b\}$, with edges $bu'$, $bw'$, $ab$, and $az$,
and with $d_H(a)=2$ and $d_H(b)=3$, see Figure~\ref{fig-2deg2-2}.
The \starg{$H$ enforcing $|\varphi(u')\cup \varphi(w')|\le 7$ and $|(\varphi(u')\cup \varphi(w'))\cap \varphi(z)|\le 2$}{G_1}{S}{G'}{F}.
Since $G$ is critical, $F$ is not a proper induced sub-e-graph of $G$.  Moreover, $F$ has minimum degree at least two,
and thus $V(F)\setminus V(G)\neq\{a\}$.  Therefore, $b\in V(F)$.
Since $G$ is 2-edge-connected, so is $G'$, and thus Observation~\ref{obs-ncnail} implies that $F$ has at least two nailed vertices.
By (a0), it follows that the underlying graph of $F$ is either $C_5$ or \kk.
In the former case, since $\deg_{G'} u'=\deg_{G'} w'=\deg_{G'} z=\deg_{G'} b=3$, $F$ would contain at least three nailed vertices
($b$ and two of the vertices $u'$, $w'$, and $z$), contradicting (a0).  Therefore, the underlying graph of $F$ is \kk.

If $|V(F)\cap \{u',w'\}|=1$, say $V(F)\cap \{u',w'\}=\{u'\}$, then $G$ contains an induced sub-e-graph $F'$ with underlying graph \kk{}
obtained from $F$ by replacing the path $u'baz$ by the path $u'uvz$.  By Observation~\ref{obs-3nails}, at least three vertices of $F'$
are nailed, implying that at least three vertices of $F$ are nailed (the nailed vertex $v$ of $F'$ corresponds to the nailed vertex $b$
of $F$).  This contradicts (a0).  Hence, we have $u',w'\in V(F)$.

Since $\deg_G u'=\deg_G w'=\deg_G z=3$, Corollary~\ref{cor-3conn} implies that $G-\{u,v,w\}$ is $2$-edge-connected.
If $a\not\in V(F)$, then $b$ is a nailed vertex of $F$, and since $F$ contains at most one more nailed vertex,
Observation~\ref{obs-ncnail} implies that $F=G'-a$.  This is not possible, since the underlying graph of $F$ is \kk, but $b$ would have degree two in $F$ and both its neighbors
in $F$ would have degree three.
Therefore, $a,z\in V(F)$; note that $\deg_F z=2$, since $\deg_F a=2$ and $\deg_F b=3$ and the underlying graph of $F$ is \kk.  By Corollary~\ref{cor-3conn} and Lemma~\ref{lemma-adj2}, we conclude that $G$ is obtained from $F$ (a copy of \kk with two nailed vertices, one of which is $z$) by deleting $a$ and $b$, adding $G_1$, and adding a vertex of degree
two adjacent to $z$ and to another vertex of $F$ of degree two (there is only one choice by symmetry).  However, this graph (depicted in Figure~\ref{fig-2deg2final}) is $11/4$-colorable, which is a contradiction.
\end{proof}

\begin{figure}
\begin{center}
\begin{tikzpicture}
\draw 
(0,0) node[vtx,label=right:$v$](v){}
--++(90:1)node[vtx,label=right:$z$](z){}
--++(90:1)node[vtx,label=right:$ $](zz){}
(v)
--++(-30:1)node[vtx,label=right:$u$](u){}
--++(-30:1)node[vtx,label=right:$u'$](u'){}
(v)
--++(210:1)node[vtx,label=left:$w$](w){}
--++(210:1)node[vtx,label=left:$w'$](w'){}
(w') to[bend left] (zz) 
(u') to[bend right] (zz) 
(w') to[bend right]  node[vtx,pos=0.333](ww){}  node[vtx,pos=0.666](uu){}   (u')
(ww) to[bend left=10] node[vtx,pos=0.25]{}  (z)
 ;
\end{tikzpicture}
\end{center}

\caption{The graph arising in the final case analysis in Lemma~\ref{lemma-2deg2}.}\label{fig-2deg2final}
\end{figure}
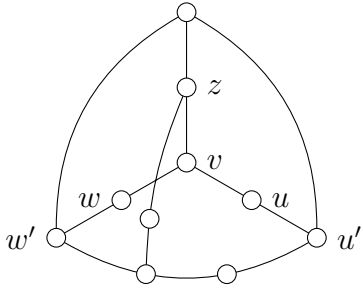

\section{Subdivided $K_4$}

In this section, we aim to show that the minimum counterexample does not contain \kk{} as an induced subgraph.
Before that, let us get rid of another simple configuration.

\begin{lemma}\label{lemma-C4only3vtxs}
Let $G$ be a minimum counterexample. If $C=u_1u_2u_3u_4$ is a 4-cycle in $G$,
then $\deg_G u_i =3$ for $i \in \{1,2,3,4\}$.
\end{lemma}
\begin{proof}
Suppose for a contradiction that $\deg u_1=2$.
Lemmas~\ref{lemma-adj2} and \ref{lemma-2deg2} imply that $\deg u_i = 3$ for $i \in \{2,3,4\}$;
let $v_i$ denote the neighbor of $u_i$ not in $C$.  These neighbors are pairwise distinct, since $G$ is triangle-free
and if $v_2=v_4$, then $u_2$ and $u_4$ would have the same neighborhood, which is not possible in a critical graph.
Lemma~\ref{lemma-2deg2} implies $\deg v_2=\deg v_4=3$; note that $\deg v_3=2$ is possible.
Let $S=\{v_2,v_3,v_4\}$ and let $G_1$ be the sub-e-graph of $G$ with the vertex set $V(C)\cup S$ and with the edge
set consisting of the edges incident with $V(C)$.
The \stex{G_1}{S}.
\end{proof}

Next, let us exclude \kk{} with three nailed vertices.

\begin{lemma}\label{lemma-k4three}
Let $G$ be a minimum counterexample. 
If $H$ is an induced sub-e-graph of $G$ and the underlying graph of $H$ is \kk, then all vertices of $H$ have degree three in $G$. 
\end{lemma}
\begin{proof}
By Observation~\ref{obs-3nails}, at most one vertex of $H$ has degree two in $G$.  Suppose $H$ contains a vertex
whose degree is two in $G$.  Let $u$, $v$, and $w$ be the vertices of $H$ that have degree two in $H$ and degree three in $G$,
where $vw\in E(H)$.  Let $u'$ be the neighbor of $u$ not in $H$,
and note that $\deg u'=3$ by Lemma~\ref{lemma-2deg2}.
Let $S=\{u',v,w\}$ and let $G_1$ be the sub-e-graph of $G$ consisting of $H$ and the edge $u'u$.
The \stex{G_1}{S}.
\end{proof}

Let us now proceed with the main result of this section.

\begin{lemma}\label{lemma-nok4}
If $G$ is a minimum counterexample, then $G$ does not contain any induced sub-e-graph with the underlying graph \kk.
\end{lemma}
\begin{proof}
Let $H$ be an induced sub-e-graph of $G$ with the underlying graph \kk, and let $v_1v_2,v_3v_4\in E(H)$ be the edges joining vertices
whose degree in $H$ is two.  Lemma~\ref{lemma-k4three} implies that for $i\in\{1,\ldots, 4\}$, $v_i$ has a neighbor $v'_i$
outside of~$H$.  Let us remark that $v'_i$ might have degree two, and $\{v'_1,v'_2\}\cap \{v'_3,v'_4\}\neq\emptyset$ is possible.
Let $S=\{v'_1,v'_2,v'_3,v'_4\}$, let $G_1$ be the sub-e-graph of $G$ with the vertex set $V(H)\cup S$ and containing all edges of $G$ incident with $V(H)$,
and let $H_1$ be the e-graph with the vertex set $S\cup\{x_1,x_2,x_3,x_4\}$, edges $v'_ix_i$ for $i\in\{1,\ldots,4\}$, $x_1x_2$, and $x_3x_4$,
and with $d_{H_1}(x_1)=d_{H_1}(x_2)=2$ and $d_{H_1}(x_3)=d_{H_1}(x_4)=3$, see Figure~\ref{fig-nok4-1}.
The \starg{$H_1$ enforcing by Observation~\ref{obs-cpath} $|\varphi(v'_1)\cap \varphi(v'_2)|\le 1$ and $|\varphi(v'_3)\cap \varphi(v'_4)|\le 3$}{G_1}{S}{G'}{F}.
Note that $G'$ is connected, since otherwise either we would obtain a contradiction with Lemma~\ref{lemma-conn},
or the edges of $H$ from $\{v_1,v_2\}$ to $V(H)\setminus \{v_1,v_2\}$ would form a $2$-edge-cut contradicting Corollary~\ref{cor-3conn}.
Since $F\neq G'$, if $x_3,x_4\in V(F)$, then $F$ would contain at least three nailed vertices ($x_3$, $x_4$, and at least one more
by Observation~\ref{obs-ncnail}), contradicting (a0).
Consequently $x_3,x_4\not\in V(F)$.  Furthermore, $F$ is not an induced sub-e-graph of $G$, and thus $x_1,x_2\in V(F)$.

\begin{figure}
\begin{center}
\begin{tikzpicture}
\draw (0,0) node[vtx](a){} -- (0,1) node[vtx,label=left:$v_1$](b){} -- ++(0,1)node[vtx,label=left:$v_2$](c){}
-- ++(0,1)node[vtx](d){}
-- (2,1.5) node[vtx](e){} to[bend left] (a) (a) -- (-2,1.5) node[vtx](f){}--(d)
(e) to[bend right=90,looseness=2] 
node[vtx,pos=0.3,label=below:$v_3$](v3){}
node[vtx,pos=0.7,label=below:$v_4$](v4){}
(f)
(b) --++(0.5,0) node[vtxS,label=below:{$v_1'$}]{}
(c) --++(0.5,0) node[vtxS,label=below:{$v_2'$}]{}
(v3) --++(0,1) node[vtxS,label=right:{$v_3'$}]{}
(v4) --++(0,1) node[vtxS,label=left:{$v_4'$}]{}
;
\draw (0,-1) node{$G_1$};
\end{tikzpicture}
\hskip 3em
\begin{tikzpicture}
\path    (0,1) node[vtx,label=left:$x_1$](b){}  ++(0,1)node[vtx,label=left:$x_2$](c){}
(2,1.5) to[bend right=90,looseness=2] 
node[nail,pos=0.3,label=below:$x_3$](v3){}
node[nail,pos=0.7,label=below:$x_4$](v4){}
 (-2,1.5)
;
\draw
(b)--(c)
(b) --++(0.5,0) node[vtxS,label=right:{$v_1'$}]{}
(c) --++(0.5,0) node[vtxS,label=right:{$v_2'$}]{}
(v3) --++(0,1) node[vtxS,label=right:{$v_3'$}]{}
(v4) --++(0,1) node[vtxS,label=left:{$v_4'$}]{}
(v3)--(v4)
;
\draw (0,-1) node{$H_1$};
\end{tikzpicture}
\end{center}
\caption{The e-graph $G_1$ and the replacement e-graph $H_1$.}\label{fig-nok4-1}
\end{figure}
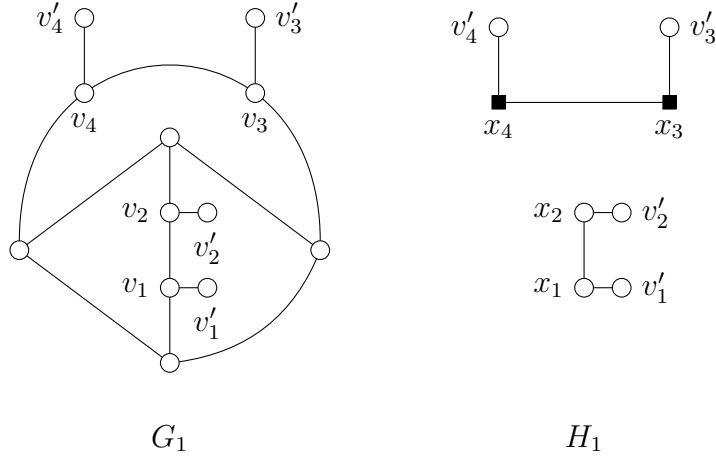

By Observation~\ref{obs-ncnail}, $F$ has at least one nailed vertex.
Suppose first that $F$ has exactly one nailed vertex $z$; in $G$, $z$ has a neighbor $z'\in V(G)\setminus V(F)$.
Let $G_2$ be the sub-e-graph of $G$ consisting of $H$, $F-\{x_1,x_2\}$, the edges
$v_1v'_1$ and $v_2v'_2$, and the edge $zz'$, see Figure~\ref{fig-nok4-2}. 
Note that $z'\not\in V(H)$, since if say $z'=v_3$, then $v_4v'_4$ would be a bridge in $G$ separating $G_2$, contradicting Lemma~\ref{lemma-conn}.
Let $S_2=\{v_3,v_4,z'\}$.  Note that $\deg z'=2$ is possible.  
Since $F\in \CC_0$, all possible e-graphs $G_2$ can be enumerated, and the \stex{G_2}{S_2}.

\begin{figure}[h]
\begin{center}
\begin{tikzpicture}h
\draw (0,0) node[vtx](a){} -- (0,1) node[vtx,label=left:$v_1$](b){} -- ++(0,1)node[vtx,label=left:$v_2$](c){}
-- ++(0,1)node[vtx](d){}
to[bend left] (4,1.5) node[vtx](e){} to[bend left] (a) (a) -- (-2,1.5) node[vtx](f){}--(d)
(e) to[bend right=90,looseness=1.5] 
node[vtxS,pos=0.3,label=below:$v_3$](v3){}
node[vtxS,pos=0.7,label=below:$v_4$](v4){}
(f)
;
\draw[fill=gray!20!white](1.6,1.5) ellipse (1.2cm and 0.9cm);
\draw
(b) --++(0.8,0) node[vtx,label=below:{$v_1'$}]{}
(c) --++(0.8,0) node[vtx,label=above:{$v_2'$}]{}
(2.3,1.5) node[vtx,label=left:$z$]{} -- ++(0.9,0) node[vtxS,label=above:{$z'$}]{}
;
\end{tikzpicture}
\end{center}
\caption{The e-graph $G_2$. The shaded part is $F-\{x_1,x_2\}$.}\label{fig-nok4-2}
\end{figure}

By (a0), it follows that $F$ contains two nailed vertices and the underlying graph of $F$ is either $C_5$ or \kk.
Let $z_1$ and $z_2$ be the nailed vertices of $F$.  Suppose that $z_1z_2\in E(F)$.
Let $S_3=\{z_1,z_2,v_3,v_4\}$ and let $G_3$ be the sub-e-graph of $G$ consisting of $H$, $F-\{x_1,x_2\}$, and
the edges $v_1v'_1$ and $v_2v'_2$.  Then the \stex{G_3}{S_3}.

Therefore $z_1z_2\not\in E(F)$.  Since $F$ contains adjacent non-nailed vertices $x_1$ and $x_2$
of degree two, the underlying graph $F$ is either $C_5$ or \kk,
and $\deg_F z_1=\deg_F z_2=2$, it follows that $F$ is a 5-cycle, $z_1=v'_1$ and $z_2=v'_2$.
Let $z'_1$ and $z'_2$ be the neighbors of $z_1$ and $z_2$ outside of $V(F)$ in $G$.
Let $z$ be the common neighbor of $z_1$ and $z_2$ in $F$; since $F$ has only two nailed vertices, we have $\deg_G z=2$.
By Lemma~\ref{lemma-2deg2}, we have $\deg z'_1=\deg z'_2=3$.  If $\{z'_1,z'_2\}\cap \{v_3,v_4\}\neq\emptyset$,
then Corollary~\ref{cor-3conn} implies $\{z'_1,z'_2\}=\{v_3,v_4\}$,
and thus $G$ is the e-graph depicted in Figure~\ref{fig-nok4-3}.
However, this e-graph belongs to $\CC_0$, which is a contradiction.
Therefore, $z'_1,z'_2\not\in V(H)$.  Furthermore, $z'_1\neq z'_2$ by Lemma~\ref{lemma-C4only3vtxs}, since $z_1$ and $z_2$
have a common neighbor $z$ of degree two.  We claim that $z'_1z'_2\in E(G)$.

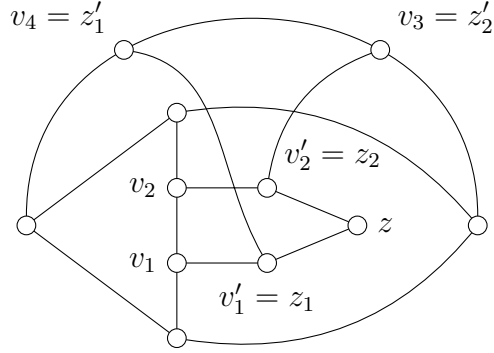
\begin{figure}
\begin{center}
\begin{tikzpicture}
\draw (0,0) node[vtx](a){} -- (0,1) node[vtx,label=left:$v_1$](b){} -- ++(0,1)node[vtx,label=left:$v_2$](c){}
-- ++(0,1)node[vtx](d){}
to[bend left] (4,1.5) node[vtx](e){} to[bend left] (a) (a) -- (-2,1.5) node[vtx](f){}--(d)
(e) to[bend right=90,looseness=1.5] 
node[vtx,pos=0.3,label=above right:{$v_3=z_2'$}](v3){}
node[vtx,pos=0.7,label=above left:{$v_4=z_1'$}](v4){}
(f)
;
\draw
(b) --++(1.2,0) node[vtx,label=below:{$v_1'=z_1$}](z1){}
(c) --++(1.2,0) node[vtx,label=above right:{$v_2'=z_2$}](z2){}
(2.4,1.5) node[vtx,label=right:$z$](z){} 
(z1)--(z)
(z2)--(z)
(z2) to[bend left] (v3)
(z1) to[out=120,in=-10] (v4)
;
\end{tikzpicture}
\end{center}
\caption{The critical e-graph arising in Lemma~\ref{lemma-nok4}.}\label{fig-nok4-3}
\end{figure}

\begin{subproof}
Suppose for a contradiction that $z'_1z'_2\not\in E(G)$.  Let $S_4=\{z'_1,z'_2,v_3,v_4\}$ and let $G_4$ be the sub-e-graph of $G$ consisting of
$H$, $F-\{x_1,x_2\}$ and the edges $v_1v'_1$, $v_2v'_2$, $z_1z'_1$ and $z_2z'_2$.
Let $H_4$ be the e-graph with the vertex set $S\cup\{a\}$ and edges $v_3v_4$, $az'_1$, and $az'_2$, with $d_{H_4}(a)=3$,
see Figure~\ref{fig-nok4-4}.  The \starg{$H_4$ enforcing $|\varphi(z'_1)\cup\varphi(z'_2)|\le 7$}{G_4}{S_4}{G''}{F_4} (note that $F_4\neq G''$ follows from (a0)
and the fact that $G''$ has three nailed vertices $a$, $v_3$, and $v_4$).  Since $G$ is critical, $F_4$ is not an induced sub-e-graph of $G$, and thus $a\in V(F_4)$.
By Lemma~\ref{lemma-conn}, $G''$ is connected (as otherwise $\{v_3,v_4\}$ would be a cut in $G$), and thus Observation~\ref{obs-ncnail} implies that
$F_4$ contains a nailed vertex different from $a$.  By (a0), $F_4$ contains exactly one nailed vertex different from $a$ and the underlying graph of $F_4$ is $C_5$ or
\kk.

\begin{figure}[h!]
\begin{center}
\begin{tikzpicture}
\draw (0,0) node[vtx](a){} -- (0,1) node[vtx,label=left:$v_1$](b){} -- ++(0,1)node[vtx,label=left:$v_2$](c){}
-- ++(0,1)node[vtx](d){}
to[out=0,in=100] (4,1.5) node[vtx](e){} to[out=270,in=0] (a) (a) -- (-2,1.5) node[vtx](f){}--(d)
(e) to[bend right=90,looseness=1.5] 
node[vtxS,pos=0.3,label=below:$v_3$](v3){}
node[vtxS,pos=0.7,label=below:$v_4$](v4){}
(f)
;
\draw
(b) --++(1,0) node[vtx,label=below:{$v_1'=z_1$}](z1){}
(c) --++(1,0) node[vtx,label=above:{$v_2'=z_2$}](z2){}
(z1) -- ++(1,0) node[vtxS,label=below right:{$z_1'$}]{}
(z2) -- ++(1,0) node[vtxS,label=above right:{$z_2'$}]{}
(1.5,1.5) node[vtx,label=right:{$z$}](z){} 
(z1)--(z)--(z2)
(1.5,-1)node{$G_4$}
;
\end{tikzpicture}
\hskip 3em
\begin{tikzpicture}
\draw
(-1,1.5) node[nailS,label=below:{$v_4$}]{}
--
(1,1.5) node[nailS,label=below:{$v_3$}]{}

(1,0) node[nail,label=left:$a$](a){} -- ++(40:1) node[vtxS,label=above:{$z_2'$}]{}
(a) -- ++(-40:1) node[vtxS,label=below:{$z_1'$}]{}
(0,-2)node{$H_4$}
;
\end{tikzpicture}
\end{center}
\caption{The e-graph $G_4$ and the replacement e-graph $H_4$.}\label{fig-nok4-4}
\end{figure}
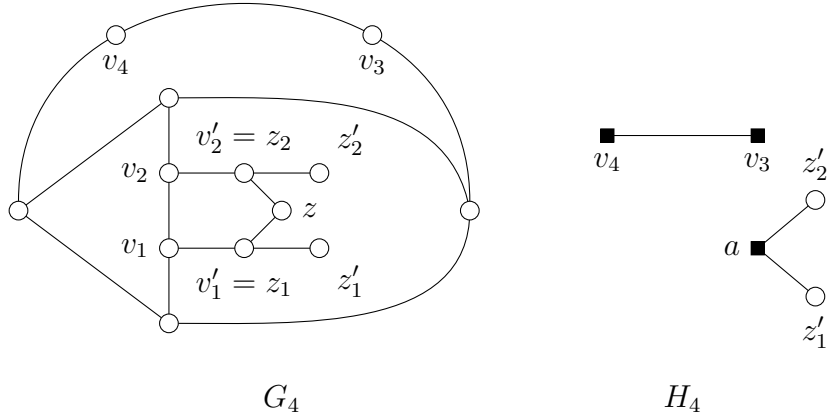

Since $\deg a=2$ and $\deg_G z'_1=\deg_G z'_2=3$, observe that $F_4$ cannot be a 5-cycle.  Hence, the underlying graph of $F_4$ is \kk{} and by symmetry, we can
assume $z'_2$ is a nailed vertex in $F_4$.
Let $z''_2$ be the neighbor of $z'_2$ not belonging to $V(F_4)$ and distinct from $z_2$.  Note that $z''_2\not\in \{v_3,v_4\}$
by Lemma~\ref{lemma-conn}.
Let $S_5=\{v_3,v_4,z''_2\}$ and let $G_5$ be the e-graph consisting of $G_4$, $F_4-a$ and the edges $z_1z'_1$, $z_2z'_2$, and $z'_2z''_2$,
see Figure~\ref{fig-nok4-5}.  Then the \stex{G_5}{S_5}.
\end{subproof}
\begin{figure}
\begin{center}
\begin{tikzpicture}
\draw (0,0) node[vtx](a){} -- (0,1) node[vtx,label=left:$v_1$](b){} -- ++(0,1)node[vtx,label=left:$v_2$](c){}
-- ++(0,1)node[vtx](d){}
to[out=0,in=100] (6,1.5) node[vtx](e){} to[out=270,in=0] (a) (a) -- (-2,1.5) node[vtx](f){}--(d)
(e) to[bend right=90,looseness=1.5] 
node[vtxS,pos=0.3,label=below:$v_3$](v3){}
node[vtxS,pos=0.7,label=below:$v_4$](v4){}
(f)
;
\draw
(b) --++(1,0) node[vtx,label=below:{$v_1'=z_1$}](z1){}
(c) --++(1,0) node[vtx,label=above:{$v_2'=z_2$}](z2){}
(z1) -- ++(1,0) node[vtx,label=below:{$z_1'$}](z1'){}
(z2) -- ++(1,0) node[vtx,label=above right:{$z_2'$}](z2'){}
(z2') -- ++(0,2) node[vtxS,label=right:{$z_2''$}]{}
(1.5,1.5) node[vtx,label=right:{$z$}](z){} 
(z1')  -- ++(45:1) node[vtx](xx){}
(z1')  -- ++(-45:1) node[vtx](yy){}  --++(45:1) node[vtx](zz){}--(xx)
(xx) --node[vtx,pos=0.33]{} node[vtx,pos=0.66]{}  (yy)
(zz) to[out=80,in=0,looseness=1.5] (z2')
(z1)--(z)--(z2)
;
\end{tikzpicture}
\end{center}
\caption{The e-graph $G_5$.}\label{fig-nok4-5}
\end{figure}
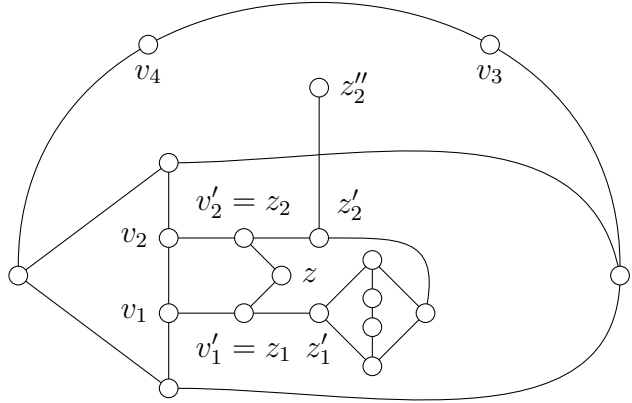

Therefore, we have $z'_1z'_2\in E(G)$; let us remark that the reduction from the case $z'_1z'_2\not\in E(G)$ does not apply,
as it would create a triangle.
For $i\in \{1,2\}$, let $w_i$ be the neighbor of $z'_i$ distinct from $z_i$ and $z'_{3-i}$.
If $\{w_1,w_2\}\cap \{v_3,v_4\}\neq\emptyset$, say $w_1=v_3$ by symmetry, then Corollary~\ref{cor-3conn} and Lemma~\ref{lemma-adj2}
imply either $w_2=v_4$, or $\deg w_2=2$ and $w_2v_4\in E(G)$; but both such graphs $G$ are $11/4$-colorable.
Hence, $\{w_1,w_2\}\cap \{v_3,v_4\}=\emptyset$.

\begin{figure}
\begin{center}
\begin{tikzpicture}
\draw (0,0) node[vtx](a){} -- (0,1) node[vtx,label=left:$v_1$](b){} -- ++(0,1)node[vtx,label=left:$v_2$](c){}
-- ++(0,1)node[vtx](d){}
to[out=0,in=100] (6,1.5) node[vtx](e){} to[out=270,in=0] (a) (a) -- (-2,1.5) node[vtx](f){}--(d)
(e) to[bend right=90,looseness=1.2] 
node[vtxS,pos=0.3,label=below:$v_3$](v3){}
node[vtxS,pos=0.7,label=below:$v_4$](v4){}
(f)
;
\draw
(b) --++(1,0) node[vtx,label=below:{$v_1'=z_1$}](z1){}
(c) --++(1,0) node[vtx,label=above:{$v_2'=z_2$}](z2){}
(z1) -- ++(1,0) node[vtx,label=below:{$z_1'$}](z1'){}
(z2) -- ++(1,0) node[vtx,label=above:{$z_2'$}](z2'){}
(1,1.5) node[vtx,label=right:{$z$}](z){} 
(z1)--(z)--(z2)
(z1')--(z2')
(z1') -- ++(1,0) node[vtxS,label=below:{$w_1$}](w1){}
(z2') -- ++(1,0) node[vtxS,label=above:{$w_2$}](w2){}

;
\draw (2,-1)node{$G_6$};
\end{tikzpicture}
\hskip 3em
\begin{tikzpicture}
\draw
(0,2) node[nailS,label=below:{$v_4$}]{}
--
(2,2) node[nailS,label=below:{$v_3$}]{}

(1,0) node[nail,label=left:$b$](b){} -- ++(0:1) node[vtxS,label=right:{$w_1$}]{}
(b) -- ++(90:1) node[vtx,label=left:$c$](a){}  -- ++(0:1) node[vtxS,label=right:{$w_2$}]{}
(0.5,-1)node{$H_4$}
;
\end{tikzpicture}
\end{center}
\caption{The e-graph $G_6$ and the replacement e-graph $H_6$.}\label{fig-nok4-6}
\end{figure}
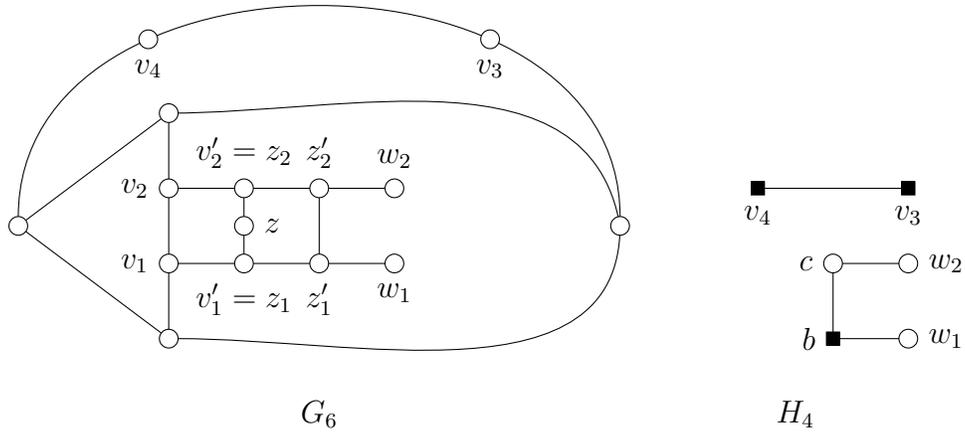
Let $S_6=\{w_1,w_2,v_3,v_4\}$ and let $G_6$ be the sub-e-graph of $G$ consisting of $G_4$ and the path
$w_1z'_1z'_2w_2$.  Let $H_6$ be the e-graph with the vertex set $S\cup\{b,c\}$
and edges $v_3v_4$, $w_1b$, $bc$, and $cw_2$, with $d_{H_6}(b)=3$ and $d_{H_6}(c)=2$;
see Figure~\ref{fig-nok4-6}.
The \starg{$H_6$ enforcing $|\varphi(w_1)\cap \varphi(w_2)|\le 2$}{G_6}{S_6}{G'''}{F_6}
(note that $F_6\neq G'''$ follows from (a0) and the fact that $G'''$ has three nailed vertices $b$, $v_3$, and $v_4$).
Clearly $b,c\in V(F_6)$, since $F_6$ is not an induced sub-e-graph of the critical graph $G$.  By Lemma~\ref{lemma-conn}, $G'''$ is connected, and thus by Observation~\ref{obs-ncnail}, $F_6$ contains
a nailed vertex $y$ distinct from $b$.  By (a0), the underlying graph of $F_6$ is $C_5$ or \kk.
If the underlying graph of $F_6$ were \kk, then replacing the path $w_1bcw_2$ of $F_6$ by the path $w_1z'_1z'_2w_2$ would yield
an induced sub-e-graph of $G$ with underlying graph \kk{} containing a vertex (a neighbor of $y$) whose degree in $G$ is two,
contradicting Lemma~\ref{lemma-k4three}.

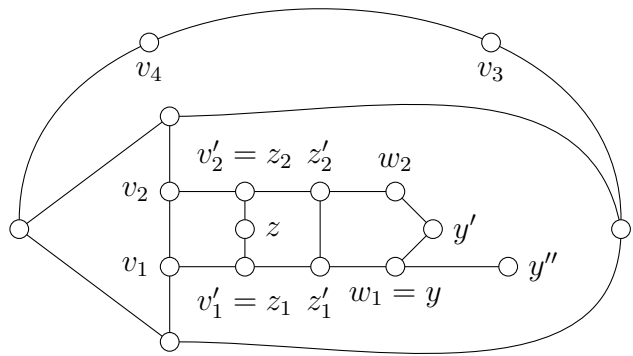
\begin{figure}
\begin{center}
\begin{tikzpicture}
\draw (0,0) node[vtx](a){} -- (0,1) node[vtx,label=left:$v_1$](b){} -- ++(0,1)node[vtx,label=left:$v_2$](c){}
-- ++(0,1)node[vtx](d){}
to[out=0,in=100] (6,1.5) node[vtx](e){} to[out=270,in=0] (a) (a) -- (-2,1.5) node[vtx](f){}--(d)
(e) to[bend right=90,looseness=1.2] 
node[vtxS,pos=0.3,label=below:$v_3$](v3){}
node[vtxS,pos=0.7,label=below:$v_4$](v4){}
(f)
;
\draw
(b) --++(1,0) node[vtx,label=below:{$v_1'=z_1$}](z1){}
(c) --++(1,0) node[vtx,label=above:{$v_2'=z_2$}](z2){}
(z1) -- ++(1,0) node[vtx,label=below:{$z_1'$}](z1'){}
(z2) -- ++(1,0) node[vtx,label=above:{$z_2'$}](z2'){}
(1,1.5) node[vtx,label=right:{$z$}](z){} 
(z1)--(z)--(z2)
(z1')--(z2')
(z1') -- ++(1,0) node[vtx,label=below:{$w_1=y$}](w1){}
(z2') -- ++(1,0) node[vtx,label=above:{$w_2$}](w2){}
(w1) -- ++(0.5,0.5)  node[vtx,label=right:{$y'$}](y'){}--(w2)
(w1) -- ++(1.5,0) node[vtxS,label=right:{$y''$}](y''){}
;
\end{tikzpicture}
\end{center}
\caption{The e-graph $G_7$.}\label{fig-nok4-7}
\end{figure}
Hence, $F_6$ is a 5-cycle, and by Lemma~\ref{lemma-2deg2} and symmetry, we can assume $y=w_1$, $\deg_G w_2=2$, and $y$ and $w_2$
have a common neighbor $y'$ of degree two in $G$.  Let $y''$ be the neighbor of $y$ distinct from $z'_1$ and $y'$;
we have $y''\not\in\{v_3,v_4\}$ by Lemma~\ref{lemma-conn} and $\deg y''=3$ by Lemma~\ref{lemma-2deg2}.
Let $S_7=\{y'',v_3,v_4\}$ and let $G_7$ be the sub-e-graph of $G$
obtained from $G_6$ by adding the path $yy'w_2$ and the edge $yy''$, see Figure~\ref{fig-nok4-7}.  Then the \stex{G_7}{S_7}.
\end{proof}

\section{Adjacent vertices of degree two}\label{sec-ad2}

Let $C=v_1v_2v_3v_4v_5$ be a 5-cycle in a minimum counterexample $G$, where $\deg_G v_4=\deg_G v_5=2$.
For $i\in\{1,2,3\}$, let $u_i$ be the neighbor of $v_i$ not in $C$.
By Lemma~\ref{lemma-5c3nailsdeg2}, we have $\deg_G u_1=\deg_G u_2=\deg_G u_3=3$.
Furthermore, the vertices $u_1$, $u_2$, and $u_3$ are pairwise distinct: Since $G$ is triangle-free,
we have $u_1\neq u_2\neq u_3$, and if $u_1=u_3$, then the $2$-edge-cut consisting of $v_2u_2$
and an edge incident with $u_1$ would contradict Corollary~\ref{cor-3conn}.
In this section, we exclude this configuration completely.

Our plan is to reduce this configuration by deleting $V(C)$ and adding an edge between two of the vertices $u_1$, $u_2$,
and $u_3$.  This course fails if the addition of any such edge creates a triangle, i.e., if each pair of these
vertices has a common neighbor.  Hence, we first need to deal with this situation, starting with the special
case that all three vertices have a common neighbor.

\begin{lemma}\label{lemma-five3common}
Let $C=v_1v_2v_3v_4v_5$ be a 5-cycle in a minimum counterexample~$G$, where $\deg_G v_4=\deg_G v_5=2$.
For $i\in\{1,2,3\}$, let $u_i$ be the neighbor of $v_i$ not in $C$.
Then the vertices $u_1$, $u_2$, and $u_3$ do not have a common neighbor.
\end{lemma}
\begin{proof}
Suppose $u_1$, $u_2$, and $u_3$ have a common neighbor $z$.
Let $w_i$ be the neighbor of $u_i$
distinct from $v_i$ and $z$.  Let $S=\{w_1,w_2,w_3\}$ and let $G_1$ be the subgraph of $G$ consisting
of the cycle $C$ and the paths $v_iu_iz$ and edges $u_iw_i$ for $i\in\{1,2,3\}$.
The \stex{G_1}{S_1}.
\end{proof}

Next, we deal with the more complicated case that each pair of the vertices $u_1$, $u_2$, and $u_3$
has a distinct common neighbor.

\begin{lemma}\label{lemma-five3nona}
Let $C=v_1v_2v_3v_4v_5$ be a 5-cycle in a minimum counterexample $G$, where $\deg_G v_4=\deg_G v_5=2$.
For $i\in\{1,2,3\}$, let $u_i$ be the neighbor of $v_i$ not in $C$.
Then there exist distinct $i,j\in\{1,2,3\}$ such that the vertices $u_i$ and $u_j$ do not have a common neighbor.
\end{lemma}
\begin{proof}
Suppose for a contradiction that for all distinct $i,j\in\{1,2,3\}$, $u_i$ and $u_j$ have a common neighbor $z_{i+j-2}$.
By Lemma~\ref{lemma-five3common}, the vertices $z_1$, $z_2$, and $z_3$ are pairwise distinct.

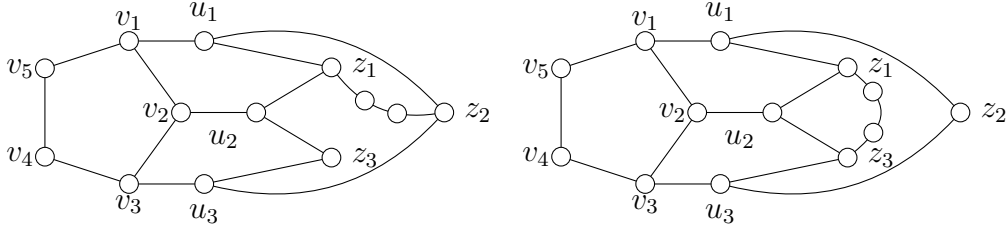
\begin{figure}
\begin{center}
\begin{tikzpicture}
\draw
\foreach \x in {1,...,5}{
(144-72*\x:1) node[vtx](v\x){}
}
(v2) node[left]{$v_2$}
(v1) node[above]{$v_1$}
(v3) node[below]{$v_3$}
(v4) node[left]{$v_4$}
(v5) node[left]{$v_5$}
(v1)--(v2)--(v3)--(v4)--(v5)--(v1)
(v1) -- ++(1,0) node[vtx,label=above:$u_1$](u1){}
(v2) -- ++(1,0) node[vtx,label=below left:$u_2$](u2){}
(v3) -- ++(1,0) node[vtx,label=below:$u_3$](u3){}
(3,0.6) node[vtx,label=right:$z_1$](z1){}
(3,-0.6) node[vtx,label=right:$z_3$](z2){}
(4.5,0) node[vtx,label=right:$z_2$](z3){}
(u2)--(z1)
(u2)--(z2)
(u1)--(z1)
(u3)--(z2)
(u1) to[bend left] (z3)
(u3) to[bend right] (z3)
(z1) to[bend right] node[pos=0.333,vtx]{} node[pos=0.666,vtx]{} (z3)
;
\end{tikzpicture}
\begin{tikzpicture}
\draw
\foreach \x in {1,...,5}{
(144-72*\x:1) node[vtx](v\x){}
}
(v2) node[left]{$v_2$}
(v1) node[above]{$v_1$}
(v3) node[below]{$v_3$}
(v4) node[left]{$v_4$}
(v5) node[left]{$v_5$}
(v1)--(v2)--(v3)--(v4)--(v5)--(v1)
(v1) -- ++(1,0) node[vtx,label=above:$u_1$](u1){}
(v2) -- ++(1,0) node[vtx,label=below left:$u_2$](u2){}
(v3) -- ++(1,0) node[vtx,label=below:$u_3$](u3){}
(3,0.6) node[vtx,label=right:$z_1$](z1){}
(3,-0.6) node[vtx,label=right:$z_3$](z2){}
(4.5,0) node[vtx,label=right:$z_2$](z3){}
(u2)--(z1)
(u2)--(z2)
(u1)--(z1)
(u3)--(z2)
(u1) to[bend left] (z3)
(u3) to[bend right] (z3)
(z1) to[bend left=50,looseness = 1.5] node[pos=0.22,vtx]{} node[pos=0.77,vtx]{} (z2)
;
\end{tikzpicture}
\end{center}
\caption{The critical e-graphs arising in Lemma~\ref{lemma-five3nona}.}\label{fig-five3nona-1}
\end{figure}

Suppose one of the vertices $z_1$, $z_2$, and $z_3$ has degree two.  By Lemma~\ref{lemma-2deg2},
the other two have degree three, and by Corollary~\ref{cor-3conn} and Lemma~\ref{lemma-C4only3vtxs},
they are joined by a path of two vertices of degree two.  However, the resulting e-graphs,
depicted in Figure~\ref{fig-five3nona-1}, belong to $\CC_0$.

Hence, we have $\deg z_1=\deg z_2=\deg z_3=3$.  For $i\in \{1,2,3\}$, let $w_i$ be the neighbor of $z_i$ not in $\{u_1,u_2,u_3\}$.
If $w_1=w_2$, then by Corollary~\ref{cor-3conn} and Lemma~\ref{lemma-adj2}, either $w_3=w_1$ or $\deg w_3=2$ and $w_3w_1\in E(G)$.
The case $w_2=w_3$ is symmetric.  If $w_1=w_3\neq w_2$, then
by Corollary~\ref{cor-3conn} and Lemma~\ref{lemma-adj2} we have $\deg w_2=2$ and $w_2w_1\in E(G)$.
However, all these graphs are $11/4$-colorable.  It follows that $w_1$, $w_2$, and $w_3$ are pairwise distinct.

\begin{figure}\begin{center}
\begin{tikzpicture}
\draw
\foreach \x in {1,...,5}{
(144-72*\x:1) node[vtx](v\x){}
}
(v2) node[left]{$v_2$}
(v1) node[above]{$v_1$}
(v3) node[below]{$v_3$}
(v4) node[left]{$v_4$}
(v5) node[left]{$v_5$}
(v1)--(v2)--(v3)--(v4)--(v5)--(v1)
(v1) -- ++(1,0) node[vtx,label=above:$u_1$](u1){}
(v2) -- ++(1,0) node[vtx,label=below left:$u_2$](u2){}
(v3) -- ++(1,0) node[vtx,label=below:$u_3$](u3){}
(3,0.6) node[vtx,label=above:$z_1$](z1){}
(3,-0.6) node[vtx,label=above:$z_3$](z2){}
(7,0) node[vtx,label=left:$z_2$](z3){}
(u2)--(z1)
(u2)--(z2)
(u1)--(z1)
(u3)--(z2)
(u1) to[out=20,in=90] (z3)
(u3) to[out=-20,in=270] (z3)
(z1)--++(1,0) node[vtx,label=below:{$w_1$}](w1){}--++(1,0) node[vtxS,label=below:{$w_1'$}](w1'){}

(z2)--++(1,0) node[vtxS,label=below:{$w_3=w_3'$}](w3){}
(z3)--++(1,0) node[vtxS,label=below:{$w_2=w_2'$}](w3){}
;
\end{tikzpicture}
\end{center}
\caption{The e-graph $G_1$ in the case $\deg w_1=2$.}\label{fig-five3nona-2}
\end{figure}
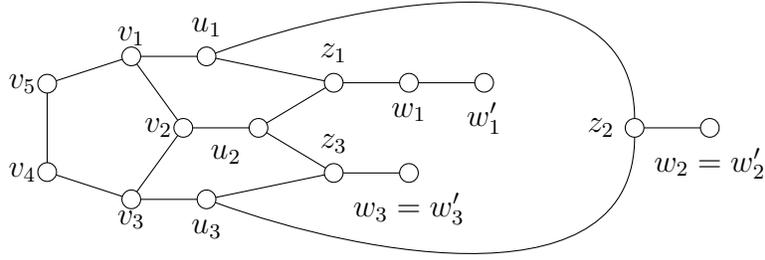

By Lemma~\ref{lemma-conn}, if for some distinct $i,j\in\{1,2,3\}$ the vertices $w_i$ and $w_j$ have degree two, then $w_iw_j\not\in E(G)$.
For $i=1,2,3$, let $w'_i=w_i$ if $\deg w_i=3$ and let $w'_i$ be the neighbor of $w_i$ distinct from $z_i$ otherwise.
Let $S_1=\{w'_1,w'_2,w'_3\}$ and let $G_1$ be the sub-e-graph of $G$ consisting of $G[V(C)\cup\{u_1,u_2,u_3,z_1,z_2,z_3\}]$,
the edges $z_iw_i$ for $i\in\{1,2,3\}$, and the edges $w_iw'_i$ for each $i\in\{1,2,3\}$ such that $\deg w_i=2$;
see Figure~\ref{fig-five3nona-2}.
If at least one of $w_1$, $w_2$, and $w_3$ has degree two, then the \stex{G_1}{S_1}.

\begin{figure}[h]
\begin{center}
\begin{tikzpicture}
\draw
\foreach \x in {1,...,5}{
(144-72*\x:1) node[vtx](v\x){}
}
(v2) node[left]{$v_2$}
(v1) node[above]{$v_1$}
(v3) node[below]{$v_3$}
(v4) node[left]{$v_4$}
(v5) node[left]{$v_5$}
(v1)--(v2)--(v3)--(v4)--(v5)--(v1)
(v1) -- ++(1,0) node[vtx,label=above:$u_1$](u1){}
(v2) -- ++(1,0) node[vtx,label=below left:$u_2$](u2){}
(v3) -- ++(1,0) node[vtx,label=below:$u_3$](u3){}
(3,0.6) node[vtx,label=above:$z_1$](z1){}
(3,-0.6) node[vtx,label=above:$z_3$](z2){}
(6,0) node[vtx,label=right:$z_2$](z3){}
(u2)--(z1)
(u2)--(z2)
(u1)--(z1)
(u3)--(z2)
(u1) to[out=20,in=90] (z3)
(u3) to[out=-20,in=270] (z3)
(z1)--++(1,0) node[vtxS,label=below:{$w_1$}](w1){}
(z2)--++(1,0) node[vtxS,label=below:{$w_3$}](w2){}
(z3)--++(-1,0) node[vtxS,label=below:{$w_2$}](w3){}
;
\draw (8,0) node {$G_1$};
\end{tikzpicture}
\hskip 3em
\begin{tikzpicture}
\draw
(0,0) node[vtx,label=below:$a$](a){}
(a) -- ++(135:1) node[vtxS,label=left:$w_1$](w1){}
(a) -- ++(225:1) node[vtxS,label=left:$w_3$](w2){}
(a)
 -- ++(1,0) node[nail,label=below:$b$](b){}
 -- ++(1,0) node[vtx,label=below:$c$](c){}
 -- ++(1,0) node[vtxS,label=below:$w_2$](w_3){}
;
\draw (5,0) node {$H_1$};
\end{tikzpicture}
\end{center}
\caption{The e-graph $G_1$ and the replacement e-graph $H_1$, the case $\{i,j\}=\{1,3\}$.}\label{fig-five3nona-3}
\end{figure}
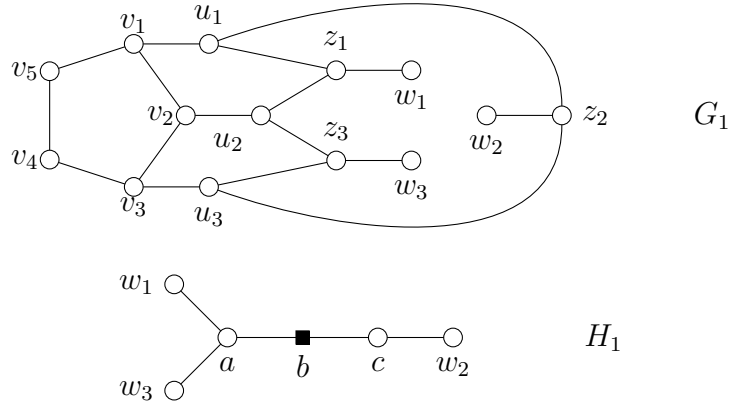
Therefore, we have $\deg w_1=\deg w_2=\deg w_3=3$.  Since $G$ is triangle-free,
there exist distinct $i,j\in\{1,2,3\}$ such that $w_iw_j\not\in E(G)$.
Let $H_1$ be the e-graph with the vertex set $S_1\cup \{a,b,c\}$, edges $w_ia$, $w_ja$, $ab$, $bc$, and $cw_{6-i-j}$, and with
$d_{H_1}(b)=3$ and $d_{H_1}(c)=2$; see Figure~\ref{fig-five3nona-3}.
The \starg{$H_1$ enforcing $|\varphi(w_1)\cup\varphi(w_2)\cup\varphi(w_3)|\le 9$ 
(as seen by Observation~\ref{obs-cpath} with $k=3$ applied to the path $abcw_2$ showing that $|\varphi(a) \cap \varphi(w_2)| \leq 2$
and consequently $|\varphi(a)\setminus\varphi(w_2)|\ge 2$,
and noticing that $|\varphi(a) \cap (\varphi(w_2) \cup \varphi(w_3))| = 0$)}{G_1}{S_1}{G'}{F_1}.
Corollary~\ref{cor-3conn} implies $G'$ is $2$-edge-connected, and thus if $b\in V(F_1)$, then
$F_1$ would contain at least three nailed vertices by Observation~\ref{obs-ncnail}, contradicting (a0).  Consequently $b,c\not\in V(F_1)$.
Since $G$ is critical, $F_1$ is not an induced sub-e-graph of $G$, and thus $a\in V(F_1)$ and $a$ is nailed in $F_1$.
Since $G'$ is $2$-edge-connected and $F_1$ contains at most two nailed vertices by (a0),
Observation~\ref{obs-ncnail} implies $F_1$ contains exactly one nailed vertex $z$ distinct from $a$.
If $z\neq w_{6-i-j}$, then the 2-edge-cut in $G$ formed by the edge $z_{6-i-j}w_{6-i-j}$
and the edge incident with $z$ not belonging to $F_1$ would contradict Corollary~\ref{cor-3conn}.
Hence, $z=w_{6-i-j}$, and thus $F_1=G'-\{b,c\}$.  By (a0), the underlying graph of
$F_1$ is $C_5$ or \kk.  However, this is not possible, since $\deg_{F_1} w_i=\deg_{F_1} w_j=3$ and $\deg_{F_1} a=2$.
\end{proof}

Let us now investigate how the reduction we proposed at the beginning of this section can result in a non-$11/4$-colorable e-graph.

\begin{lemma}\label{lemma-canaddedge}
Let $C=v_1v_2v_3v_4v_5$ be a 5-cycle in a minimum counterexample $G$, where $\deg v_4=\deg v_5=2$.
For $i\in\{1,2,3\}$, let $u_i$ be the neighbor of $v_i$ not in $C$.
Let $j\in \{2,3\}$ be an index such that $u_1$ and $u_j$ do not have a common neighbor.
Let $G'$ be the e-graph obtained from $G-V(C)$ by adding the edge $u_1u_j$ (if not already present) and setting $d_{G'}(u_{5-j})=2$.
If $F$ is a critical induced sub-e-graph of $G'$, then $u_1u_j\not\in E(F)$, $u_{5-j}\in V(F)$, and $F$ is a $5$-cycle with two nailed vertices.
\end{lemma}
\begin{proof}
Note that $d_{G'}(u_1)=d_{G'}(u_j)=3$ and $\deg_{G'} u_{5-j}=2$, and $G'$ is $2$-edge-connected by Corollary~\ref{cor-3conn}.
By the minimality of $G$, we have $F\in \CC_0$, and in particular $F$ has at most two nailed vertices by (a0).

Let us first consider the case that $u_1u_j\not\in E(F)$, and thus $F\neq G'$.
Since $G'$ is $2$-edge-connected, Observation~\ref{obs-ncnail} implies $F$ has two nailed vertices, and by (a0),
the underlying graph of $F$ is $C_5$ or \kk.  By Lemma~\ref{lemma-nok4}, it cannot be \kk, and thus $F$ is a $5$-cycle with two nailed vertices.
Moreover, since $G$ is critical, $F$ is not an induced sub-e-graph of $G$, and thus $u_{5-j}\in V(F)$, as required.

Suppose now for a contradiction that $u_1u_j\in E(F)$.  If $u_1u_j\in E(G)$, then $u_1$ and $u_j$
are nailed vertices of $F$, and by (a0), the underlying graph of $F$ is $C_5$ or $\kk$.  The latter is not possible by
Lemma~\ref{lemma-nok4}.  In the former case, no vertices of $F$ other than $u_1$ and $u_j$ can be nailed,
and thus Observation~\ref{obs-ncnail} implies that $G'=F$.  Note that two vertices of $F$ have degree two in $G$,
and thus by Lemma~\ref{lemma-2deg2}, $G-V(C)$ contains a path of length three with ends in $\{u_1,u_2,u_3\}$ and
remaining vertices of degree two in $G$, and $F$ consists of this path and a path of length two induced by $\{u_1,u_2,u_3\}$.
Therefore, $G$ is one of the e-graphs depicted in Figure~\ref{fig-canaddedge}.
These graphs are $11/4$-colorable, which is a contradiction.  Therefore, $u_1u_j\not\in E(G)$.

\begin{figure}\begin{center}
\begin{tikzpicture}
\draw
\foreach \x in {1,...,5}{
(144-72*\x:1) node[vtx](v\x){}
}
(v2) node[left]{$v_2$}
(v1) node[above]{$v_1$}
(v3) node[below]{$v_3$}
(v4) node[left]{$v_4$}
(v5) node[left]{$v_5$}
(v1)--(v2)--(v3)--(v4)--(v5)--(v1)

\foreach \x in {1,...,5}{
(3 cm, 0 cm)++(180+144-72*\x:1) node[vtx](z\x){}
}
(z1)--(z2)--(z3)--(z4)--(z5)--(z1)
(v2)--(z2)
(v1)--(z3)
(v3)--(z1)
(z1) node[below]{$u_3$}
(z2) node[right]{$u_2$}
(z3) node[above]{$u_1$}
;
\begin{scope}[xshift = 6cm]
\draw
\foreach \x in {1,...,5}{
(144-72*\x:1) node[vtx](v\x){}
}
(v2) node[left]{$v_2$}
(v1) node[above]{$v_1$}
(v3) node[below]{$v_3$}
(v4) node[left]{$v_4$}
(v5) node[left]{$v_5$}
(v1)--(v2)--(v3)--(v4)--(v5)--(v1)

\foreach \x in {1,...,5}{
(3 cm, 0 cm)++(180+144-72*\x:1) node[vtx](z\x){}
}
(z1)--(z2)--(z3)--(z4)--(z5)--(z1)
(v2)--(z2)
(v1)--(z3)
(v3) to[out=-40,in=-40, looseness=1.3](z4)
(z4) node[above]{$u_3$}
(z2) node[right]{$u_2$}
(z3) node[above]{$u_1$}
;
\end{scope}
\begin{scope}[xshift = 3cm, yshift=-3cm]
\draw
\foreach \x in {1,...,5}{
(144-72*\x:1) node[vtx](v\x){}
}
(v2) node[left]{$v_2$}
(v1) node[above]{$v_1$}
(v3) node[below]{$v_3$}
(v4) node[left]{$v_4$}
(v5) node[left]{$v_5$}
(v1)--(v2)--(v3)--(v4)--(v5)--(v1)

\foreach \x in {1,...,5}{
(3 cm, 0 cm)++(180+144-72*\x:1) node[vtx](z\x){}
}
(z1)--(z2)--(z3)--(z4)--(z5)--(z1)
(v2)--(z2)
(v1)  to[out=40,in=40, looseness=1.3] (z5)
(v3) -- (z1)
(z1) node[below]{$u_3$}
(z2) node[right]{$u_2$}
(z5) node[right]{$u_1$}
;
\end{scope}
\end{tikzpicture}
\end{center}
\caption{E-graphs from Lemma~\ref{lemma-canaddedge} in the case $u_1u_j\in E(F) \cap E(G)$.}\label{fig-canaddedge}
\end{figure}

If $u_{5-j}\in V(F)$, then since at most two vertices of $F$
are nailed, Observation~\ref{obs-ncnail} and Corollary~\ref{cor-3conn} imply that
either $F=G'$ or $G'$ is obtained from $F$ by adding a path of at most two vertices of degree two between the nailed
vertices of $F$.  By a computer-assisted enumeration, we verified
that for every e-graph in $\CC_0$, performing this transformation to obtain $G'$, deleting the edge $u_1u_j$
and adding $C$ to obtain $G$ results in an e-graph that either is not critical or belongs to $\CC_0$.  Since $G\not\in \CC_0$ is critical,
it follows that $u_{5-j}\not\in V(F)$.

In particular, $F\neq G'$. Since $G'$ is $2$-edge-connected, Observation~\ref{obs-ncnail} implies $F$ contains two nailed vertices, and by (a0),
the underlying graph of $F$ is $C_5$ or \kk.  If $u_1$ and $u_j$ were the nailed vertices of $F$,
then the partition $\{V(F)\setminus\{u_1,u_j\}, (V(G)\setminus V(F))\cup\{u_1,u_j\}\}$ of $V(G)$ would contradict
Corollary~\ref{cor-3conn}.  Consequently, $\max(\deg_F(u_1),\deg_F(u_j))=3$, and in particular the underlying graph of $F$ is \kk.

\begin{figure}
\begin{center}
\begin{tikzpicture}
\draw
\foreach \x in {1,...,5}{
(144-72*\x:1) node[vtx](v\x){}
}
(v2) node[left]{$v_2$}
(v1) node[above]{$v_1$}
(v3) node[below]{$v_3$}
(v4) node[left]{$v_4$}
(v5) node[left]{$v_5$}
(v1)--(v2)--(v3)--(v4)--(v5)--(v1)
(v1) -- ++(1.7,0) node[vtx,label=above:$u_1$](u1){}
(v2) -- ++(1,0) node[vtx,label=below:$u_2$](u2){}
(v3) -- ++(1,0) node[vtxS,label=below:$u_3$](u3){}
(u1) --
node[vtx,pos=0.33,label=above :$x_1$](x1){}
node[vtx,pos=0.66,label=below:$ $]{}
++(2,0)
 node[vtx,label=below:$ $](a){}
(u2) --
node[vtx,pos=0.33,label=below:$ $] {}
node[vtx,pos=0.66,label=below:$x_2$](x2){}
++(2,0)
 node[vtx,label=below:$ $](b){}
 (u1)--(b)--(a)--(u2)
 
 (x2)++(1.3,0.3) node[vtx,label=below:{$x_2'$}](x2'){}
 (x2)++(1.3,1.3) node[vtxS,label=above:{$x_1'=x_1''$}](x1'){}
 (x1)to[out=20,in=190](x1')
 (x2)to[out=80,in=180](x2')
 (x2') --++(1,0) node[vtxS,label=below:{$x_2''$}](x2''){}
;
\draw (8,0) node {$G_1$};
\end{tikzpicture}
\hskip 3em
\begin{tikzpicture}
\draw
(0,0) node[vtx,label=below:$a$](a){}
(a) -- ++(135:1) node[vtxS,label=left:$z_1$](w1){}
(a) -- ++(225:1) node[vtxS,label=left:$z_2$](w2){}
(a)
 -- ++(1,0) node[nail,label=below:$b$](b){}
 -- ++(1,0) node[nail,label=below:$c$](c){}
 -- ++(1,0) node[vtxS,label=below:$z_3$](z_3){}
;
\draw (5,0) node {$H_1$};
\end{tikzpicture}
\end{center}
\caption{The e-graph $G_1$ and the replacement e-graph $H_1$ in case $j=2$ and $\deg (x_2')=2$. In this case $S_1=\{x_1'',x_2'',u_3\}=\{z_1,z_2,z_3\}$ such that $z_1z_2$ is not an edge.}\label{fig-canaddedge-1}
\end{figure}
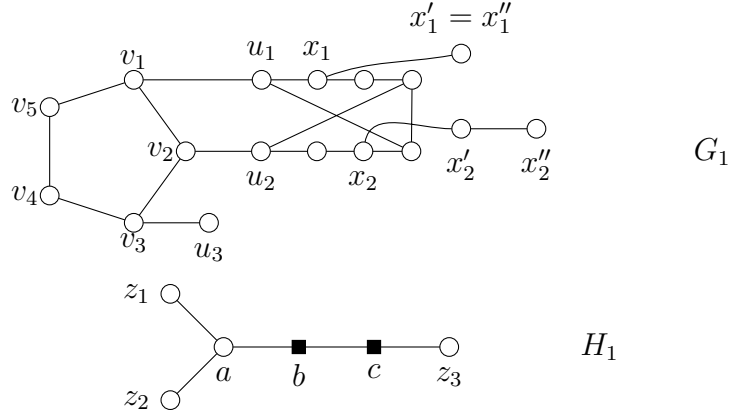
Let $x_1$ and $x_2$ be the nailed vertices of $F$ and let $x'_1$ and $x'_2$ be their neighbors in $G'-V(F)$.
For $i\in\{1,2\}$, if $\deg_G(x'_i)=2$, then let $x''_i$ be the neighbor of $x'_i$ distinct from $x_i$,
otherwise let $x''_i=x'_i$.  Let us remark that if $\deg_G(x'_1)=\deg_G(x'_2)=2$, then $x'_1\neq x'_2$ and $x'_1x'_2\not\in E(G)$,
as otherwise $v_{5-j}u_{5-j}$ would be a bridge in $G$, contradicting Lemma~\ref{lemma-conn}.
Let $S_1=\{x''_1,x''_2,u_{5-j}\}$ and let $G_1$ be the sub-e-graph of $G$ consisting of $C$, $F-u_1u_j$, the
edges $u_iv_i$ for $i\in\{1,2,3\}$, the edges $x_1x'_1$ and $x_2x'_2$, and the edges $x'_ix''_i$ for
each $i\in\{1,2\}$ such that $\deg x'_i=2$; see Figure~\ref{fig-canaddedge-1} illustrating one of the possible sub-e-graphs $G_1$.
If $x'_1$ or $x'_2$ has degree two, or if $x'_1$, $x'_2$, and $u_{5-j}$ are not
pairwise distinct, or if $u_1$ or $u_j$ has degree two in $F$, then
the \stex{G_1}{S_1}.  Hence, we can assume
the vertices $x'_1$, $x'_2$, and $u_{5-j}$ are pairwise distinct and have degree three in $G$, and the vertices $u_1$ and $u_j$ have degree
three in $F$.

\begin{figure}
\begin{center}
\newcommand{\baseplace}{
\clip (-1.3,-1.3) rectangle (7.5,2);
\draw
\foreach \x in {1,...,5}{
(144-72*\x:1) node[vtx](v\x){}
}
(v2) node[left]{$  $}
(v1) node[above]{$ $}
(v3) node[below]{$ $}
(v4) node[left]{$ $}
(v5) node[left]{$ $}
(v1)--(v2)--(v3)--(v4)--(v5)--(v1)
(v1)   ++(1.7,0) node[vtx,label=above:$  $](u1){}
(v2) -- ++(1,0) node[vtx,label=below:$ $](u2){}
(u1) --
node[vtx,pos=0.333,label=above:$  $](x1){}
node[vtx,pos=0.666,label=below:$  $](z2){}
++(2,0)
 node[vtx,label=below:$ $](a){}
(u2) --
node[vtx,pos=0.333,label=below:$ $] {}
node[vtx,pos=0.666,label=below:$ $](x2){}
++(2,0)
 node[vtx,label=below:$ $](b){}
 (u1)--(b)--(a)--(u2)
 ;
 
\draw
\foreach \x in {1,...,5}{
(6,0)+(180+144-72*\x:1) node[vtx,label=above:$ $](w\x){}
}
(w1)--(w2)--(w3)--(w4)--(w5)--(w1)
(v3)--(w1)
;
}

\begin{tikzpicture}[scale=0.6]
\baseplace
\draw
(v1)--(u1)
(x2) to[out=-40,in=180,looseness=1.5](w3)
(z2) to[out=40,in=120,looseness=1.5] (w2)
;
\end{tikzpicture}
\hskip 2em
\begin{tikzpicture}[scale=0.6]
\baseplace
\draw
(v1)--(u1)
(x2) to[out=-40,in=180,looseness=1.5](w2)
(z2) to[out=40,in=180,looseness=1.5] (w3)
;
\end{tikzpicture}
\hskip 2em
\begin{tikzpicture}[scale=0.6]
\baseplace
\draw
(v1)--(u1)
(x2) to[out=-40,in=180,looseness=1.5](w2)
(z2) to[out=40,in=60,looseness=1.5] (w5)
;
\end{tikzpicture}
\hskip 2em
\begin{tikzpicture}[scale=0.6]
\baseplace
\draw
(v4) to[out=120,in=160,looseness=1.5](u1)
(x2) to[out=-40,in=180,looseness=1.5](w2)
(z2) to[out=40,in=60,looseness=1.7] (w5)
;
\end{tikzpicture}
\end{center}
\caption{Critical e-graphs arising in Lemma~\ref{lemma-canaddedge}.}\label{fig-canaddedge-2}
\end{figure}
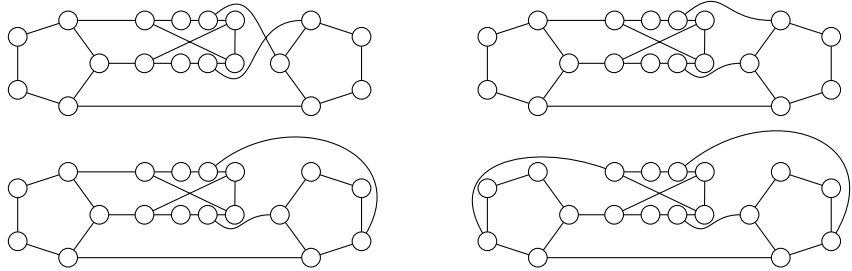

Since $G$ is triangle-free, there
exist distinct non-adjacent vertices $z_1,z_2\in S_1$; we let $H_1$ be the e-graph with the vertex set $S_1\cup\{a,b,c\}$,
edges $z_1a$, $z_2a$, $ab$, $bc$, and an edge from $c$ to the vertex $z_3\in S_1\setminus\{z_1,z_2\}$, and with $d_{H_1}(b)=d_{H_1}(c)=3$,
see Figure~\ref{fig-canaddedge-1}.
The \starg{$H_1$ enforcing $|\varphi(x'_1)\cup \varphi(x'_2)\cup\varphi(u_{5-j})|\le 10$ 
(as seen by Observation~\ref{obs-cpath} with $k=3$ applied to the path $abcz_3$,
showing that $|\varphi(a) \cap \varphi(z_3)| \leq 3$ and consequently $|\varphi(a)\setminus\varphi(z_3)|\ge 1$,
and noticing that $|\varphi(a) \cap (\varphi(z_1) \cup \varphi(z_2))| = 0$)
}{G_1}{S_1}{G''}{F_1}.
Let us remark that since $b$ and $c$ are nailed in $G''$ and $\deg_{G''}(a)=3$, to exclude the case $F_1=G''$
by (a0) we only need to consider the possiblity that the underlying graph of $G''$ is $\kk$---that is, $G$ consists of two disjoint $5$-cycles $C$
and $G''-\{a,b,c\}$, each containing three vertices of degree three, an edge between the two cycles, and of $F-u_1u_j$ (\kk{} without an edge)
joined by four edges to $C\cup G''-\{a,b,c\}$. Figure~\ref{fig-canaddedge-2} shows the elements of $\CC_0$ arising in this way,
while the rest of such e-graphs were shown to be $11/4$-colorable by computer-assisted enumeration.

Note that $G''$ is $2$-edge-connected by Lemma~\ref{lemma-conn}.  Since $F_1\neq G''$ and $F_1$ contains at most two nailed
vertices, we conclude by Observation~\ref{obs-ncnail} that $b,c\not\in V(F_1)$.  Since $G$ is critical, $F_1$ is not an induced sub-e-graph of $G$, and thus
$a\in V(F_1)$.  Since $G''$ is $2$-edge-connected, $F_1$ contains at least two nailed vertices by Observation~\ref{obs-ncnail}. By (a0),
it follows that the underlying graph of $F_1$ is $C_5$ or \kk{} and $F_1$ contains exactly two nailed vertices, one of which is $a$.
Since $\deg_G z_3=3$, Corollary~\ref{cor-3conn} implies $z_3$ is the other nailed vertex of $F_1$ and
$F_1=G''-\{b,c\}$.  However, this implies $\deg_{F_1}(z_1)=\deg_{F_1}(z_2)=3$, which is not possible since the underlying graph of $F_1$ is $C_5$ or \kk
and $\deg_{F_1}(a)=2$.
\end{proof}

The coloring properties of a 5-cycle with three vertices of degree three are described in the next lemma,
which follows from the measure version of Hall's theorem.

\begin{lemma}\label{lemma-hall-adj}
Let $C=v_1v_2v_3v_4v_5$ be a $5$-cycle with vertices $v_1$, $v_2$, and $v_3$ nailed.  For $i\in\{1,2,3\}$, let a set $S(v_i)\subseteq [0,11)$ have measure at least $7$.
There exists an $11/4$-coloring $\varphi$ of $C$ such that $\varphi(v_i)\subseteq S(v_i)$ for $i\in \{1,2,3\}$
if and only if $|S(v_1)\cup S(v_2)|\ge 8$, $|S(v_2)\cup S(v_3)|\ge 8$, and $|S(v_1)\cup S(v_2)\cup S(v_3)|=11$.
\end{lemma}
\begin{proof}
The conditions $|S(v_i)\cup S(v_2)|\ge 8$ for $i\in\{1,3\}$ are necessary, since $\varphi(v_2)$ and $\varphi(v_i)$ are disjoint sets of measure four.
Furthermore, by Observation~\ref{obs-constraints} with $A=\{v_1,v_3\}$, $z=v_2$, $B=\{v_1\}$, $C=\{v_3\}$, $z_1=v_5$, and $z_2=v_4$, we have $|\varphi(v_1)\cap\varphi(v_3)|=1$.
Consequently, $|\varphi(v_1)\cup \varphi(v_3)|=7$, and since $\varphi(v_2)$ is disjoint from $\varphi(v_1)\cup\varphi(v_3)$, this implies $|\varphi(v_1)\cup\varphi(v_2)\cup\varphi(v_3)|=11$, showing the necessity
of the last condition.

Conversely, let $X_1=S(v_1)$, $X_2=S(v_1)\cap S(v_3)$, $X_3=S(v_3)$, and $X_4=S(v_2)$, and suppose there exist pairwise disjoint sets $L_1\subseteq X_1$ and $L_3\subseteq X_3$ of measure $a_1=a_3=3$,
$L_2\subseteq X_2$ of measure $a_2=1$, and $L_4\subseteq X_4$ of measure $a_4=4$.
Then we can let $\varphi(v_1)=L_1\cup L_2$, $\varphi(v_2)=L_4$, and $\varphi(v_3)=L_2\cup L_3$;
this $11/4$-coloring extends to $v_4$ and $v_5$ by Observation~\ref{obs-cpath} with $k=3$, since $|\varphi(v_1)\cap\varphi(v_3)|=1$.
By Hall's theorem, such sets $L_1$, \ldots, $L_4$ exist if the inequality
$$\Bigl|\bigcup_{i\in I} X_i\Bigr|\ge \sum_{i\in I} a_i$$
holds for every $I\subseteq\{1,2,3,4\}$.  Since $|S(v_i)|\ge 7$ for $i\in\{1,2,3\}$ (and thus $|X_2|\ge |S(v_1)|+|S(v_3)|-11\ge 3$), $X_2\subseteq X_1$, and $X_2\subseteq X_3$,
all these inequalities are trivial or implied by the others, except for $|X_1\cup X_2\cup X_4|\ge 8$ (equivalent to $|S(v_1)\cup S(v_2)|\ge 8$),
$|X_2\cup X_3\cup X_4|\ge 8$ (equivalent to $|S(v_2)\cup S(v_3)|\ge 8$), 
and $|X_1\cup X_2\cup X_3\cup X_4|\ge 11$ (equivalent to $|S(v_1)\cup S(v_2)\cup S(v_3)|=11$).
\end{proof}

\begin{corollary}\label{cor-hall-adj}
Let $C=v_1v_2v_3v_4v_5$ be a $5$-cycle with vertices $v_1$, $v_2$, and $v_3$ nailed.
Let $f:\{1,2,3\}\to\{0,1\}$ be a function such that $f(1)+f(2)+f(3)=1$.
For $i\in\{1,2,3\}$, let a set $S(v_i)\subseteq [0,11)$ have measure at least $7-f(i)$.
If $|S(v_1)\cup S(v_2)|\ge 8-f(1)-f(2)$, $|S(v_2)\cup S(v_3)|\ge 8-f(2)-f(3)$, and $|S(v_1)\cup S(v_2)\cup S(v_3)|=11$, then
there exists an $11/4$-coloring $\varphi$ of $C$ such that $|\varphi(v_i)\setminus S(v_i)|\le f(i)$ holds for $i\in \{1,2,3\}$.
\end{corollary}
\begin{proof}
We prove the claim in case $f(2)=1$; the cases $f(1)=1$ or $f(3)=1$ are handled similarly.
Let $A_1=S(v_1)\setminus (S(v_2)\cup S(v_3))$, $A_2=[0,11)\setminus (S(v_1)\cup S(v_2)\cup S(v_3))$,
$A_3=S(v_3)\setminus (S(v_1)\cup S(v_2))$, and $A_4=(S(v_1)\cap S(v_3))\setminus S(v_2)$; note that the sets $A_1,\ldots,A_4$
form a partition of $[0,11)\setminus S(v_2)$.  For $i=1,\ldots,4$, let $a_i=|A_i|$.
By the assumptions, we have ($\star$)
\begin{align*}
a_1,a_2,a_3,a_4&\ge 0\\
a_1+\cdots+a_4&\le 5&&\text{since $|S(v_2)|\ge 6$}\\
a_2+a_3&\le 4&&\text{since $|S(v_1)\cup S(v_2)|\ge 7$}\\
a_1+a_2&\le 4&&\text{since $|S(v_2)\cup S(v_3)|\ge 7$}
\end{align*}
We claim there exist real numbers $m_1,\ldots, m_4$ such that $0\le m_i\le a_i$ for $i\in\{1,\ldots,4\}$, and
\begin{align*}
m_1+\cdots+m_4&\le 1\\
(a_1-m_1)+\cdots+(a_4-m_4)&\le 4\\
(a_2-m_2)+(a_3-m_3)&\le 3\\
(a_1-m_1)+(a_2-m_2)&\le 3
\end{align*}
It suffices to verify this is the case for the vertices of the polytope defined by ($\star$).
By symmetry between $v_1$ and $v_3$, we only need to consider the vertices satisfying $a_1\le a_3$.
These vertices and the corresponding values of $m_1$, \ldots, $m_4$ are:

\begin{center}
\begin{tabular}{cccc|cccc}
$a_1$&$a_2$&$a_3$&$a_4$&$m_1$&$m_2$&$m_3$&$m_4$\\
\hline
0&0&0&0&0&0&0&0\\
0&0&0&5&0&0&0&1\\
0&0&4&0&0&0&1&0\\
0&4&0&0&0&1&0&0\\
0&0&4&1&0&0&1&0\\
0&4&0&1&0&1&0&0\\
1&0&4&0&0&0&1&0\\
1&3&1&0&0&1&0&0\\
\end{tabular}
\end{center}

For $i\in\{1,\ldots,4\}$, choose a set $M_i\subseteq A_i$ of measure $m_i$ arbitrarily, and let $M=M_1\cup \ldots\cup M_4$.
Then
\begin{align*}
|S(v_2)\cup M|&=11-((a_1-m_1)+\ldots+(a_4-m_4))\ge 7\\
|S(v_1)\cup (S(v_2)\cup M)|&=11-((a_2-m_2)+(a_3-m_3))\ge 8\text{, and}\\
|S(v_3)\cup (S(v_2)\cup M)|&=11-((a_1-m_1)+(a_2-m_2))\ge 8.
\end{align*}
By Lemma~\ref{lemma-hall-adj}, since $|S(v_1)|,|S(v_3)|\ge 7$ and $|S(v_1)\cup S(v_2)\cup S(v_3)|=11$,
there exists an $11/4$-coloring $\varphi$ of $C$ such that $\varphi(v_1)\subset S(v_1)$,
$\varphi(v_2)\subset S(v_2)\cup M$, and $\varphi(v_3)\subset S(v_3)$.  Since $|M|=m_1+\ldots+m_4\le 1$, we have $|\varphi(v_2)\setminus S(v_2)|\le 1$,
as required.
\end{proof}

Finally, we are ready to restrict the distance between vertices of degree two in a minimum counterexample.

\begin{lemma}\label{lemma-noadj2}
If $G$ is a minimum counterexample, then the distance between any two vertices of $G$ of degree two is at least three.
\end{lemma}
\begin{proof}
Let $v_4$ and $v_5$ be distinct vertices of $G$ of degree two.  By Lemma~\ref{lemma-2deg2}, the distance between $v_4$ and $v_5$
is not exactly two.  Suppose for a contradiction that $v_4v_5\in E(G)$.

By Lemma~\ref{lemma-adj2},
$G$ contains a $5$-cycle $C=v_1v_2v_3v_4v_5$.  For $i\in\{1,2,3\}$, let $u_i$ be the neighbor of $v_i$ not in $C$;
as we argued at the beginning of this section, the vertices $u_1$, $u_2$, and $u_3$ have degree three and are pairwise distinct.
By Lemma~\ref{lemma-five3nona} and symmetry, we can assume there exists $j\in \{2,3\}$ such that
$u_1$ and $u_j$ do not have a common neighbor; moreover, the labels can be chosen so that either $j=2$ or $u_2$ has a common neighbor with
both $u_1$ and $u_3$.

Let $G'$ be the e-graph obtained from $G-V(C)$ by adding the edge $u_1u_j$ (if not already present) and setting $d_{G'}(u_{5-j})=2$.
Suppose first that $G'$ has an $11/4$-coloring $\varphi$.  Let $f(1)=f(j)=0$ and $f(5-j)=1$.
For $i\in\{1,2,3\}$, let $S(v_i)=[0,11)\setminus\varphi(u_i)$; then $|S(v_i)|=7-f(i)$.
Furthermore, the edge $u_1u_j$ ensures that $\varphi(u_1)$ is disjoint from $\varphi(u_j)$, and thus
$|S(v_1)\cup S(v_j)|=11$.  This implies
$|S(v_1)\cup S(v_2)|\ge 8-f(1)-f(2)$, $|S(v_2)\cup S(v_3)|\ge 8-f(2)-f(3)$, and $|S(v_1)\cup S(v_2)\cup S(v_3)|=11$.
By Corollary~\ref{cor-hall-adj}, there exists an $11/4$-coloring $\psi$ of $C$ such that
$|\psi(v_i)\setminus S(v_i)|\le f(i)$ holds for $i\in \{1,2,3\}$.  Let $\psi(x)=\varphi(x)$ for every $x\in V(G)\setminus (V(C)\cup\{u_{5-j}\})$
and let $\psi(u_{5-j})=\varphi(u_{5-j})\setminus\psi(v_{5-j})$; note that $|\psi(u_{5-j})|\ge |\varphi(u_{5-j})|-f(5-j)=4$.  Then $\psi$ is an $11/4$-coloring
of $G$, which is a contradiction.

Therefore, $G'$ is not $11/4$-colorable, and thus it contains a critical induced sub-e-graph $F$.  By Lemma~\ref{lemma-canaddedge},
$u_1u_j\not\in E(F)$, $u_{5-j}\in V(F)$, and $F$ is a $5$-cycle with two nailed vertices.
Since $F$ is an induced sub-e-graph of $G'$ and $u_1u_j\not\in E(F)$, we have $\{u_1,u_j\}\not\subseteq V(F)$.
If $|\{u_1,u_j\}\cap V(F)|=1$, then the vertex of $\{u_1,u_j\}\cap V(F)$ is nailed in $F$, and
the partition $\{V(C)\cup V(F),V(G)\setminus (V(C)\cup V(F))\}$ of $V(G)$ contradicts Corollary~\ref{cor-3conn}.
Therefore, $\{u_1,u_j\}\cap V(F)=\emptyset$.

\begin{figure}
\begin{center}
\begin{tikzpicture}
\draw
\foreach \x in {1,...,5}{
(144-72*\x:1) node[vtx](v\x){}

(3,0)++(180+144-72*\x:1) node[vtx](x\x){}
}
(v2) node[left]{$v_2$}
(v1) node[above]{$v_1$}
(v3) node[below]{$v_3$}
(v4) node[left]{$v_4$}
(v5) node[left]{$v_5$}
(v1)--(v2)--(v3)--(v4)--(v5)--(v1)

(v1) -- node[pos=0.5,vtx,label=above:$u_1$](u1){} (x3)
(v2) -- ++(1,0) node[vtx,label=right:$u_2$](u2){}
(v3) -- node[pos=0.5,vtx,label=below:$u_3$](u3){} (x1)

(x1)--(x2)--(x3)--(x4)--(x5)--(x1)
(u1)--(u3)
;
\end{tikzpicture}
\end{center}
\caption{The e-graph from the case $j=3$.}\label{fig-noadj2-1}
\end{figure}
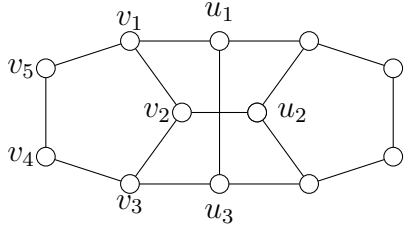
Suppose $j=3$, and thus $u_2$ has a common neighbor in $G$ with both $u_1$ and $u_3$ by the choice of the labels at the beginning of the proof;
these common neighbors must be the nailed vertices of $F$ since $u_2 \in V(F)$ implies the common neighbors are also in $V(F)$ while $\{u_1,u_j\}\cap V(F)=\emptyset$.
Since $u_1$ and $u_3$ do not have a common neighbor in $G$, Corollary~\ref{cor-3conn} and Lemma~\ref{lemma-adj2}
imply $u_1u_3\in E(G)$.  However, the corresponding graph $G$ (depicted in Figure~\ref{fig-noadj2-1}) is $11/4$-colorable, which is a contradiction.
Hence, we have $j=2$.

\begin{figure}
\begin{center}
\begin{tikzpicture}
\draw
\foreach \x in {1,...,5}{
(144-72*\x:1) node[vtx](v\x){}
}
(v2) node[left]{$v_2$}
(v1) node[above]{$v_1$}
(v3) node[below]{$v_3$}
(v4) node[left]{$v_4$}
(v5) node[left]{$v_5$}
(v1)--(v2)--(v3)--(v4)--(v5)--(v1)
(v1) -- ++(1.7,0) node[vtxS,label=above:$u_1$](u1){}
(v2) -- ++(1,0) node[vtxS,label=below:$u_2$](u2){}
(v3) -- ++(1.7,0) node[vtx,label=below left:$u_3$](u3){}

(u3) -- ++(1,0) node[vtx,label=below:$z_1$](z1){}
        -- ++(1,0) node[vtx,label=below right:$z_2$](z2){}
(z1) -- ++(0,1) node[vtxS,label=above:$w_1$]{}
(z2) -- ++(0,1) node[vtxS,label=above:$w_2$]{}
(z2) to[out=270,in=270] node[vtx,pos=0.333]{} node[vtx,pos=0.666]{} (u3)
;
\end{tikzpicture}
\hskip 2em
\begin{tikzpicture}
\draw

(0,0) node[vtxS,label=above:$u_1$](u1){}
-- ++(0,-1) node[vtxS,label=below:$u_2$](u2){}

 ++(1,-1) node[nail,label=below:$a$](z1){}
        -- ++(1,0) node[nail,label=below:$b$](z2){}
(z1) -- ++(0,1) node[vtxS,label=above:$w_1$]{}
(z2) -- ++(0,1) node[vtxS,label=above:$w_2$]{}
;
\end{tikzpicture}
\end{center}
\caption{The e-graph $G_1$ and the replacement e-graph $H_1$.}\label{fig-noadj2-2}
\end{figure}
Let $w_1$ and $w_2$ be the vertices of $G-V(F)$ with neighbors in $V(F)$ distinct from $v_3$.
Applying the observations from the beginning of the section to the 5-cycle in $G$ corresponding to $F$, we conclude
$w_1\neq w_2$ and $\deg w_1=\deg w_2=3$.  Let $z_1$ and $z_2$ be the neighbors of $w_1$ and $w_2$ in $V(F)$,
respectively; note that $\{z_1,z_2,u_3\}$ induces a subpath of $F$.
Let $S_1=\{u_1,u_2,w_1,w_2\}$ and let $G_1$ be the sub-e-graph of $G$ consisting of $G[V(C)\cup V(F)]$
and the edges $u_1v_1$, $u_2v_2$, $w_1z_1$ and $w_2z_2$.  Let $H_1$ be the graph with the vertex set $S_1\cup\{a,b\}$
and edges $u_1u_2$, $w_1a$, $ab$, and $bw_2$, such that $d_{H_1}(a)=d_{H_1}(b)=3$, see Figure~\ref{fig-noadj2-2} (showing the case that
$u_3z_1z_2$ is a path in $F$; another non-symmetric possibility is that $z_1u_3z_2$ is a path in $F$).
The \starg{$H_1$ enforcing $\varphi(u_1)\cap\varphi(u_2)=\emptyset$ and $|\varphi(w_1)\cap\varphi(w_2)|\le 3$}{G_1}{S_1}{G''}{F_1}.
Let us remark that since $a$ and $b$ are nailed, to see that $F_1\neq G''$ it suffices by (a0) to consider the cases that the underlying graph of $G''$ is $C_5$ or \kk.
However, $G''$ is not a 5-cycle since $\deg_{G''} u_1=3$.  If the underlying graph of $G''$ is \kk, then $\{u_1,u_2\}\cap \{w_1,w_2\}$ consists of a single vertex $y$
and $G$ is obtained from the $5$-cycles $C$, $F$, and $G''-\{a,b,y\}$ by adding an edge between each of the cycles and
an edge from $y$ to each of the cycles; by a computer-assisted enumeration, we verified that all such graphs are $11/4$-colorable.

By Lemma~\ref{lemma-conn}, the graph $G''$ is connected, and since $F_1\neq G''$, Observation~\ref{obs-ncnail} implies
$F_1$ contains a nailed vertex distinct from $a$ and $b$.
Since $F_1$ contains at most two nailed vertices by (a0), we conclude $a,b\not\in V(F_1)$.  However, this means that $F_1$ is an induced
sub-e-graph of $G'$ not containing $u_3$, which contradicts Lemma~\ref{lemma-canaddedge}.
\end{proof}

Since the distance between any two vertices of a 5-cycle is at most two, Lemma~\ref{lemma-noadj2} 
has the following useful consequence.

\begin{corollary}\label{cor-five4}
Every $5$-cycle in a minimum counterexample contains at least four vertices of degree three.
\end{corollary}

Moreover, using Lemma~\ref{lemma-noadj2} we can exclude one of the outcomes in Corollary~\ref{cor-3conn},
strengthening it as follows.

\begin{corollary}\label{cor-3conn1}
Every $2$-edge-cut in a minimum counterexample is formed by the edges incident with a single vertex of degree two.
\end{corollary}

\section{4-cycle with adjacent 2-vertex}

In this section, we show that vertices with a neighbor in a 4-cycle have degree three.
Recall that vertices inside a 4-cycle have degree three by Lemma~\ref{lemma-C4only3vtxs}.
As usual, we start the argument by discussing some degenerate cases.

\begin{lemma}\label{lemma-nofour3}
Let $C=v_1v_2v_3v_4$ be a $4$-cycle in a minimum counterexample $G$, and for $i\in\{1,2,3,4\}$,
let $u_i$ be the neighbor of $v_i$ not in $C$.  Then the vertices $u_1$, \ldots, $u_4$ are pairwise
distinct, and for $i\in\{1,2\}$, at most one of the vertices $u_i$ and $u_{i+2}$ has degree two.
\end{lemma}
\begin{proof}
Since $G$ is triangle-free, we have $u_1\neq u_2$.  If $u_1=u_3$, then $v_1$ and $v_3$ would have the
same neighbors, contradicting the criticality of $G$.  Symmetric arguments show that the vertices
$u_1$, \ldots, $u_4$ are pairwise distinct.  For $i\in\{1,2,3,4\}$, let $u'_i$ be the neighbor of $u_i$
distinct from $v_i$ if $\deg u_i=2$ and let $u'_i=u_i$ otherwise.  Let $S_1=\{u'_1,\ldots,u'_4\}$
and let $G_1$ be the sub-e-graph of $G$ consisting of $C$, the edges $u_iv_i$ for $i\in\{1,2,3,4\}$,
and the edges $u_iu'_i$ for those $i$ such that $\deg u_i=2$.  If $\deg u_1=\deg u_3=2$ or $\deg u_2=\deg u_4=2$,
then the \stex{G_1}{S_1}.
\end{proof}

A bit surprisingly, the case there are two consecutive neighbors of degree two is substantially more complicated.

\begin{lemma}\label{lemma-nofour2}
Let $C=v_1v_2v_3v_4$ be a $4$-cycle in a minimum counterexample $G$, and for $i\in\{1,2,3,4\}$,
let $u_i$ be the neighbor of $v_i$ not in $C$.  Then at most one of the vertices $u_1$, \ldots, $u_4$
has degree two.
\end{lemma}
\begin{proof}
Otherwise, we can by Lemma~\ref{lemma-nofour3} and symmetry assume $\deg u_1=\deg u_2=2$ and
$\deg u_3=\deg u_4=3$.  For $i\in\{1,2\}$, let $u'_i$ be the neighbor of $u_i$ distinct from $v_i$;
by Lemma~\ref{lemma-noadj2}, we have $\deg u'_i=3$ and $u'_1\neq u'_2$.

\begin{figure}
\begin{center}
\begin{tikzpicture}[scale=3]
\draw
(0,0) node[vtx,label=below:$v_3$](v3){}
(1,0) node[vtx,label=below:$v_4$](v4){}
(1,1) node[vtx,label=above:$v_1$](v1){}
(0,1) node[vtx,label=above:$v_2$](v2){} 
(v1)--(v2)--(v3)--(v4)--(v1)
(v3)--  node[pos=0.3,vtx,label=left:$u_3$](u3){}  node[pos=0.7,vtx,label=right:$u_1$](u1){}    (v1) 
(v4)--  node[pos=0.3,vtx,label=right:$u_4$](u4){}  node[pos=0.7,vtx,label=left:$u_2$](u2){}    (v2) 
(u3)--(u4)
;
\end{tikzpicture}
\end{center}
\caption{The critical e-graph arising in Lemma~\ref{lemma-nofour2}. 
}\label{fig-nofour2-1}
\end{figure}

We claim that $u'_1\neq u_3$.
\begin{subproof}
Suppose for a contradiction that $u'_1=u_3$.  If $u'_2=u_4$, then since $G$ does not contain \kk{} as an induced
subgraph by Lemma~\ref{lemma-nok4}, we have $u_3u_4\in E(G)$; but the resulting e-graph, depicted in Figure~\ref{fig-nofour2-1}, belongs to $\CC_0$.
Hence, we can assume $u'_2\neq u_4$.

\begin{figure}
\begin{center}
\vc{
\begin{tikzpicture}
\draw
(0,0) node[vtx,label=below:$v_3$](v3){}
(3,0) node[vtx,label=below:$v_4$](v4){}
(3,3) node[vtx,label=above:$v_1$](v1){}
(0,3) node[vtx,label=above:$v_2$](v2){} 
(v1)--(v2)--(v3)--(v4)--(v1)
(v3)--  node[pos=0.33,vtx,label=left:$u_3$](u3){}  node[pos=0.66,vtx,label=right:$u_1$](u1){}    (v1) 
(v4)--  ++(1,0)node[vtxS,label=below:$u_4$](u4){} 
 (v2)  -- node[pos=0.5,vtx,label=above:$u_2$](u2){}   ++(-2,0)   node[vtxS,label=left:{$u_2'$}](u2'){} 
(u3)--  ++(1,0) node[vtxS,label=below:{$u_3'$}]{} 
;
\end{tikzpicture}
}
\hskip 2em
\vc{
\begin{tikzpicture}
\draw
(0,2) node[vtxS,label=left:{$u_2'$}](u2){} 
(2,0) node[vtxS,label=below:$u_4$](u4){} 
(1,1) node[nailS,label=below:{$u_3'$}]{} 
(u2) to[out=0,in=90] 
node[pos=0.333,nail,label=above:{$b$}]{} 
node[pos=0.666,nail,label=right:{$a$}]{} 
(u4)
;
\end{tikzpicture}
}
\end{center}
\caption{The e-graph $G_1$ and the replacement e-graph $H_1$.}\label{fig-nofour2-2}
\end{figure}

Let $u'_3$ be the neighbor of $u_3$ distinct from $v_3$ and $u_1$.
Let $S_1=\{u'_2,u'_3,u_4\}$, let $G_1$ be the sub-e-graph of $G$ consisting of $C$,
the paths $v_1u_1u_3v_3$ and $v_2u_2u'_2$, and the edges $v_4u_4$ and $u_3u'_3$,
and let $H_1$ be the e-graph with the vertex set $S_1\cup\{a,b\}$, edges $u_4a$, $ab$, and $bu'_2$,
and $d_{H_1}(a)=d_{H_1}(b)=3$, see Figure~\ref{fig-nofour2-2}.  The \starg{$H_1$ enforcing $|\varphi(u'_2)\cap\varphi(u_4)|\le 3$}{G_1}{S_1}{G'}{F_1}
(note that $G'$ is triangle-free since $u'_2\neq u_4$, and $F_1\neq G'$ by (a0) since $a$, $b$, and $u'_3$ are nailed
in $G'$).  Since $G$ is critical, $F_1$ is not an induced sub-e-graph of $G$, and thus $a,b\in V(F_1)$.  But $G'$
is connected by Lemma~\ref{lemma-conn}, and thus by Observation~\ref{obs-ncnail}, $F_1\neq G''$ contains a nailed vertex distinct from $a$ and $b$.  This contradicts (a0).
\end{subproof}

\begin{figure}
\begin{center}
\vc{
\begin{tikzpicture}
\draw
(0,0) node[vtx,label=below:$v_3$](v3){}
(1.5,0) node[vtx,label=below:$v_4$](v4){}
(1.5,1.5) node[vtx,label=above:$v_1$](v1){}
(0,1.5) node[vtx,label=above:$v_2$](v2){} 
(v1)--(v2)--(v3)--(v4)--(v1)
(v4)--  ++(1,0)node[vtxS,label=below:$u_4$](u4){} 
(v3)--  ++(-1,0)node[vtxS,label=below:$u_3$](u3){} 
 (v2)  -- node[pos=0.5,vtx,label=above:$u_2$](u2){}   ++(-2,0)   node[vtxS,label=above:{$u_2'$}](u2'){} 
 (v1)  -- node[pos=0.5,vtx,label=above:$u_1$](u1){}   ++(2,0)   node[vtxS,label=above:{$u_1'$}](u1'){} 
;
\end{tikzpicture}
}
\hskip 1em
\vc{
\begin{tikzpicture}
\path
(0,0)  coordinate (v3)
(1.5,0) coordinate (v4) 
(1.5,1.5) coordinate(v1)
(0,1.5) coordinate(v2)
(v4)--  ++(1,0)node[vtxS,label=below:$u_4$](u4){} 
(v3)--  ++(-1,0)node[vtxS,label=below:$u_3$](u3){} 
 (v2)  -- node[pos=0.5,vtx,label=above:$u_2$](u2){}   ++(-2,0)   node[vtxS,label=above:{$u_2'$}](u2'){} 
 (v1)  -- node[pos=0.5,vtx,label=above:$u_1$](u1){}   ++(2,0)   node[vtxS,label=above:{$u_1'$}](u1'){} 
;
\draw
(0,0.75) node[vtx,label=left:$y$](y){}
(1.5,0.75) node[vtx,label=right:$x$](x){}
(u1')--(u1)--(x)--(u3)
(u2')--(u2)--(y)--(u4)
(x)--(y)
;
\end{tikzpicture}
}
\end{center}
\caption{The part of $G$ modified to create the e-graph $G''$.}\label{fig-nofour2-3}
\end{figure}
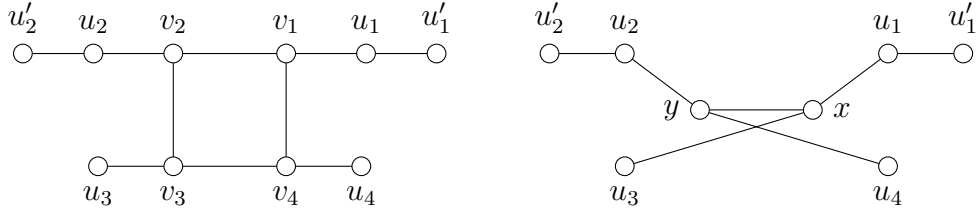

Therefore, we have $u'_1\neq u_3$, and symmetrically $u'_2\neq u_4$.  Let $G''$ be the graph obtained from $G$ by identifying $v_1$ with $v_3$
to a new vertex $x$ and $v_2$ with $v_4$ to a new vertex $y$, see Figure~\ref{fig-nofour2-3}. Then $G''$ is triangle-free.  Furthermore, $G''$ is not $11/4$-colorable,
since giving $v_1$ and $v_3$ the color set of $x$ and $v_2$ and $v_4$ the color set of $y$ would result in an $11/4$-coloring of $G$.
Hence, $G''$ contains a critical induced sub-e-graph $F_2$, which by the minimality of $G$ belongs to $\CC_0$.
If $F_2=G''$, then $G$ is obtained from $F_2\in \CC_0$ by uncontracting the edge $xy$ back to a $4$-cycle; however,
by a computer-assisted enumeration, we verified that all such graphs either belong to $\CC_0$ or are not critical, which is a contradiction.
Therefore, $F_2\neq G''$.  The graph $G''$ is $2$-edge-connected by Corollary~\ref{cor-3conn1}, and thus by Observation~\ref{obs-ncnail},
$F_2$ has at least two nailed vertices.  By (a0), $F_2$ has exactly two nailed vertices and the underlying graph of $F_2$ is $C_5$ or \kk.

Since $G$ is critical, it does not contain an e-graph isomorphic to $F_2$ as an induced sub-e-graph, and thus
$\{u_1,u_2\}\cap V(F_2)\neq \emptyset$.  By symmetry, we can assume $u_1\in V(F_2)$, and thus also $u'_1,x\in V(F_2)$.
Consequently, $F_2$ contains at least three distinct vertices whose degree in $G''$ is three, namely $u'_1$, $x$, and either $y$ or $u_3$.
Since $F_2$ has only two nailed vertices, the underlying graph of $F_2$ cannot be $C_5$, and thus it is \kk.

If $\deg_{F_2} x=3$, then since $\deg_{F_2} u_1=2$ and no vertex of \kk{} has two neighbors of degree two,
it follows that $\deg_{F_2} y=3$.  Since \kk{} contains two edges joining vertices of degree two, we have
$\deg_{F_2} u'_1=\deg_{F_2} u'_2=2$, and consequently $u'_1$
and $u'_2$ are the two nailed vertices of $F_2$.  Hence, precisely two edges of $G$ are leaving $(V(F_2)\setminus \{x,y\})\cup V(C)$,
and these edges are incident with $u'_1$ and $u'_2$.  By Corollary~\ref{cor-3conn1},
$u'_1$ and $u'_2$ have a common neighbor $z$ of degree two in $G$.  However, then $u'_1$ has two neighbors
of degree two in $G$, namely $u_1$ and $z$, contradicting Lemma~\ref{lemma-2deg2}.

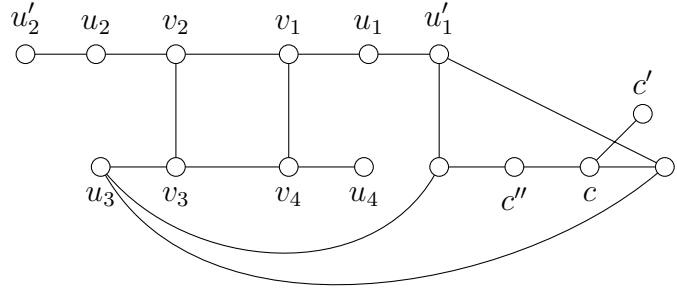
\begin{figure}
\begin{center}
\begin{tikzpicture}
\clip (-2.5,-2) rectangle (7,2.5);
\draw
(0,0) node[vtx,label=below:$v_3$](v3){}
(1.5,0) node[vtx,label=below:$v_4$](v4){}
(1.5,1.5) node[vtx,label=above:$v_1$](v1){}
(0,1.5) node[vtx,label=above:$v_2$](v2){} 
(v1)--(v2)--(v3)--(v4)--(v1)
(v4)--  ++(1,0)node[vtxS,label=below:$u_4$](u4){} 
(v3)--  ++(-1,0)node[vtx,label=below:$u_3$](u3){} 
 (v2)  -- node[pos=0.5,vtx,label=above:$u_2$](u2){}   ++(-2,0)   node[vtxS,label=above:{$u_2'$}](u2'){} 
 (v1)  -- node[pos=0.5,vtx,label=above:$u_1$](u1){}   ++(2,0)   node[vtx,label=above:{$u_1'$}](u1'){} 
 
 (u4)
  ++(1,0) node[vtx](a){}
  --
  ++(1,0) node[vtx,label=below:$c''$](x){}
  --
  ++(1,0) node[vtx,label=below:$c$](c){}
  --
  ++(1,0) node[vtx](b){}
 
 (u3) to[out=-50,in=240, looseness=1] (a) (a) -- (u1')
 (u3) to[out=-60,in=220, looseness=0.9] (b) (b) -- (u1')
 (c) -- ++(45:1) node[vtxS,label=above:$c'$]{}
;
\end{tikzpicture}
\end{center}
\caption{The e-graph $G_2$.}\label{fig-nofour2-4}
\end{figure}

Therefore, $\deg_{F_2} x=2$.  If $y\in V(F_2)$, then since both $x$ and $u_1$ have degree two, $y$ would have to
have degree three, and $y$ would be adjacent to two vertices $x$ and $u_2$ of degree two in $F_2$.
This does not happen in \kk.  Consequently $y\not\in V(F_2)$, and thus also $u_2\not\in V(F_2)$ and $u_3\in V(F_2)$.
Let $c$ be the nailed vertex of $F_2$ distinct from $x$ and using Lemma~\ref{lemma-conn}, observe that $u'_2\neq c\neq u_4$.
Let $c'$ be the neighbor of $c$ in $G$ not belonging to $V(F_2)$.
Since the underlying graph of $F_2$ is \kk, $c$ has a non-nailed neighbor $c''$ of degree two in $F_2$.
Note that $\deg_G c''=2$, and thus $\deg_G c'=3$ by Lemma~\ref{lemma-2deg2}.
Let $S_2=\{u_2',u_4,c'\}$ and let $G_2$ be the subgraph of $G$
consisting of $G[V(F-x)\cup V(C)\cup\{u_2\}]$ and the edges $u_2u'_2$, $v_4u_4$, and $cc'$,
see Figure~\ref{fig-nofour2-4}.  The \stex{G_2}{S_2}, which is a contradiction.
\end{proof}

Finally, let us deal with the case there is exactly one neighbor of degree two.

\begin{lemma}\label{lemma-nofour1}
Let $G$ be a minimum counterexample and suppose that a vertex $v\in V(G)$ has a neighbor in a $4$-cycle.
Then $\deg v=3$.
\end{lemma}
\begin{proof}
Let $C=v_1v_2v_3v_4$ be a $4$-cycle in $G$, and for $i\in\{1,2,3,4\}$,
let $u_i$ be the neighbor of $v_i$ not in $C$.  Suppose for a contradiction that $\deg u_1=2$.
By Lemma~\ref{lemma-nofour3}, the vertices $u_1$, \ldots, $u_4$ are pairwise distinct, and
by Lemma~\ref{lemma-nofour2}, the vertices $u_2$, $u_3$, and $u_4$ have degree three.
Let $u'_1$ be the neighbor of $u_1$ distinct from $v_1$; we have $\deg u'_1=3$ by Lemma~\ref{lemma-noadj2}.

\begin{figure}
\begin{center}
\begin{tikzpicture}
\draw
(0,0) node[vtx,label=below:$v_3$](v3){}
(1.5,0) node[vtx,label=below:$v_4$](v4){}
(1.5,1.5) node[vtx,label=above:$v_1$](v1){}
(0,1.5) node[vtx,label=above:$v_2$](v2){} 
(v1)--(v2)--(v3)--(v4)--(v1)
(v4)--  ++(1,0)node[vtxS,label=below:$u_4$](u4){} 
(v3)--  ++(-1,0)node[vtxS,label=below:$u_3$](u3){} 
 (v2)   -- ++(-1,0)   node[vtxS,label=above:{$u_2$}](u2){} 
 (v1)  -- ++(1,0) node[vtx,label=above:$u_1$](u1){}   -- ++(1,0)   node[vtxS,label=above:{$u_1'$}](u1'){} 
;
\end{tikzpicture}
\hskip 1em
\begin{tikzpicture}
\draw
(0,0) node[vtx,label=below:$v_3$](v3){}
(1.5,0) node[vtx,label=below:$v_4$](v4){}
(1.5,1.5) node[vtx,label=above:$v_1$](v1){}
(0,1.5) node[vtx,label=above:$v_2$](v2){} 
(v1)--(v2)--(v3)--(v4)
(v4)--  ++(1,0)node[vtxS,label=below:$u_4$](u4){} 
(v3)--  ++(-1,0)node[vtxS,label=below:$u_3$](u3){} 
 (v2)   -- ++(-1,0)   node[vtxS,label=above:{$u_2$}](u2){} 
 (v1)  -- ++(1,0) node[vtx,label=above:$u_1$](u1){}   -- ++(1,0)   node[vtxS,label=above:{$u_1'$}](u1'){} 
;
\end{tikzpicture}
\end{center}
\caption{The e-graph $G_1$ and the replacement e-graph $H$.}\label{fig-nofour1-1}
\end{figure}
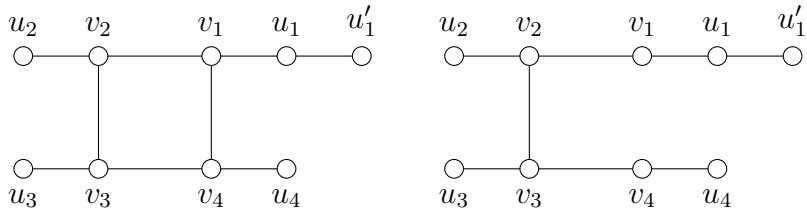

Let $S=\{u'_1,u_2,u_3,u_4\}$, let $G_1$ be the sub-e-graph of $G$ consisting of $C$, the path $v_1u_1u'_1$,
and the edges $v_iu_i$ for $i\in\{2,3,4\}$, and let $H$ be obtained from $G_1$ by deleting the edge $v_1v_4$
and setting $d_H(v_1)=d_H(v_4)=2$, see Figure~\ref{fig-nofour1-1}.
Note that in any $11/4$-coloring $\varphi$ of $H$, Observation~\ref{obs-constraints} implies
$|\varphi(v_2)\cap\varphi(u_4)|\le 2$, $|\varphi(v_2)\cap\varphi(u'_1)|\le 1$, and $|\varphi(v_2)\cap \varphi(u_2)|=0$,
and thus $|\varphi(v_2)\setminus (\varphi(u'_1)\cup\varphi(u_2)\cup\varphi(u_4))|\ge 1$.  This implies
$|\varphi(u'_1)\cup\varphi(u_2)\cup\varphi(u_4)|\le 10$.
The \starg{$H$ enforcing $|\varphi(u_3)\cap\varphi(u_4)|\le 2$, $|\varphi(u_2)\cap\varphi(u_3)|\le 3$,
and $|\varphi(u'_1)\cup\varphi(u_2)\cup\varphi(u_4)|\le 10$}{G_1}{S}{G'}{F}.

Note that $G'$ is $2$-edge-connected by Corollary~\ref{cor-3conn1}, and since $F\neq G'$, at least two vertices
of $F$ are nailed by Observation~\ref{obs-ncnail}.  By (a0), exactly two vertices of $F$ are nailed and the underlying graph of $F$ is $C_5$ or \kk.
Since $G$ is critical, $F$ is not an induced sub-e-graph of $G$, and thus $\{v_1,v_4\}\cap V(F)\neq \emptyset$.
If $F$ is a 5-cycle, then $F$ contains exactly one of $v_1$ and $v_4$, since $G$ is triangle-free;
but then the corresponding 5-cycle in $G$ has only three nailed vertices (the two nailed vertices
of $F$ and the vertex of $\{v_1,v_4\}\cap V(F)$), contradicting Corollary~\ref{cor-five4}.
If the underlying graph of $F$ is \kk, then $v_1,v_4\in V(F)$ by Lemma~\ref{lemma-nok4}, implying that $v_1$, $u_1$, and $v_4$
are vertices of $F$ of degree two.  This is a contradiction, since none of them is nailed in $F$, $F$ has two nailed vertices,
and \kk{} has only four vertices of degree two.
\end{proof}

\section{Vertices of degree two}

We now aim to get rid of vertices of degree two completely.  We need the following observation.
\begin{lemma}\label{lemma-sset}
Let $A_1$, $A_2$, $B$, and $C$ be measurable subsets of $[0,11)$ such that $|A_1|=|A_2|=4$, $|B|=5$ and $|C|=1$,
and $B$ is disjoint from $A_1\cup A_2\cup C$.  Then there exists a set $X\subset [0,11)\setminus (B\cup C)$ of measure $2$
such that $|A_1\cap X|\le 1$ and $|A_2\cap X|\le 1$.
\end{lemma}
\begin{proof}
Without loss of generality, we can assume $B=[6,11)$ and $C=[5,6)$.  Furthermore, we can assume that $A_1,A_2\subset [0,5)$,
since replacing $A_i$ by a subset of $[0,5)$ of measure $4$ containing $A_i\cap[0,5)$ only makes it harder to select $X$.
Let $t=|[0,5)\setminus (A_1\cup A_2)|$; then $|A_1\setminus A_2|=|A_2\setminus A_1|=1-t$.  Hence, we can let $X$
consist of $X_1=[0,5)\setminus (A_1\cap A_2)$ (a set of measure $2-t$) together with a subset $X_2$ of $A_1\cap A_2$ of measure $t$.
For $i\in \{1,2\}$, the set $X$ intersects $A_i$ in $(A_i\setminus A_{3-i})\cup X_2$, which has measure $(1-t)+t=1$.
\end{proof}

Let us constrain the neighborhoods of vertices of degree two.  To this end, we exploit the stronger statement we are
proving (obtaining larger color sets on vertices of degree two in the reduced graphs) as well as the possibility
to convexly combine colorings arising from several different reductions.
\begin{lemma}\label{lemma-all2bad}
Let $G$ be a minimum counterexample, let $v\in V(G)$ have degree two, and let $v_1$ and $v_2$ be the neighbors of $v$.
Then $G-\{v,v_1,v_2\}$ contains a $5$-cycle $K$ such that each of $v_1$ and $v_2$ has exactly one neighbor in $K$
and a vertex of $K$ has degree two in $G$.
\end{lemma}
\begin{proof}
We have $\deg v_1=\deg v_2=3$ by Lemma~\ref{lemma-noadj2}.  
For $i\in \{1,2\}$, let $u_{i,1}$ and $u_{i,2}$ be the neighbors
of $v_i$ distinct from $v$.  By Lemmas~\ref{lemma-C4only3vtxs} and \ref{lemma-noadj2},
the vertices $u_{i,j}$ for $i,j\in\{1,2\}$ are pairwise distinct and have degree three.

Firstly, we claim that $G$ has a set $11$-coloring $\varphi_1$ such that $|\varphi_1(v)|=1$,
for $i\in\{1,2\}$ we have $|\varphi_1(v_i)|=7$, and for $j\in\{1,2\}$, we have
$|\varphi_1(u_{i,j})|=3$.  Note that $\varphi_1$ is not an $11/4$-coloring: The measures of $\varphi_1(v)$ and $\varphi_1(u_{i,j})$ are smaller than needed,
while the measures of $\varphi(v_1)$ and $\varphi(v_2)$ are larger.
\begin{subproof}
Let $G'$ be the e-graph obtained from $G-v$ by
setting $d_{G'}(v_1)=d_{G'}(v_2)=2$.  Suppose that $G'$ has a critical induced sub-e-graph $F_1$, by the minimality of $G$ belonging to $\CC_0$.
By a computer-assisted enumeration, we verified that adding a common neighbor to a pair of non-nailed vertices in an e-graph from $\CC_0$
results either in an e-graph from $\CC_0$ or in a non-critical e-graph.  Since $G$ is obtained from $G'$ in this way, we have $G'\not\in \CC_0$,
and thus $F_1\neq G'$.  By Corollary~\ref{cor-3conn1}, $G'$ is $2$-edge-connected,
and thus $F_1$ contains at least two nailed vertices by Observation~\ref{obs-ncnail}.  By (a0), $F_1$ contains exactly two nailed
vertices and the underlying graph of $F_1$ is $C_5$ or \kk.  By Lemma~\ref{lemma-nok4},
$F_1$ is a 5-cycle.  By Corollary~\ref{cor-five4}, at least four vertices of $F_1$ have degree three in $G$,
and since only two vertices of $F_1$ are nailed, it follows that $v_1,v_2\in F_1$.  Since $G$ is triangle-free,
this implies $v_1$ and $v_2$ have a common neighbor distinct from $v$.  However,
this contradicts Lemma~\ref{lemma-C4only3vtxs}.

Therefore, $G'$ does not contain a critical induced sub-e-graph, and thus $G'$ has an $11/4$-coloring $\varphi'_1$.
Note that $|\varphi'_1(v_i)|=5$ for $i\in\{1,2\}$.
Let us now define the set $11$-coloring $\varphi_1$ of $G$ as follows.
Let $\varphi_1(v)$ be an arbitrary subset of $[0,11)\setminus (\varphi'_1(v_1)\cup \varphi'_1(v_2))$ of measure~$1$.
For every vertex $x$ at distance at least three from $v$, let $\varphi_1(x)=\varphi'_1(x)$.
For $i\in \{1,2\}$, let $X_i$ be the set obtained using Lemma~\ref{lemma-sset} with $C=\varphi_1(v)$, $B=\varphi'_1(v_i)$,
$A_1=\varphi'_1(u_{i,1})$, and $A_2=\varphi'_1(u_{i,2})$.  Let $\varphi_1(v_i)=\varphi'_1(v_i)\cup X_i$
and for $j\in \{1,2\}$, let $\varphi_1(u_{i,j})$ be a subset of $\varphi'_1(u_{i,j})\setminus X_i$ of measure $3$.
Let us remark that $|\varphi_1(v_i)|=|\varphi'_1(v_i)|+|X_i|=7$.
\end{subproof}

Let $G''$ be the e-graph obtained from $G-\{v_1,v,v_2\}$ by setting $d_{G''}(u_{i,j})=2$ for $i,j\in\{1,2\}$.
We claim $G''$ does not have any $11/4$-coloring.
\begin{subproof}
Suppose for a contradiction $G''$ has an $11/4$-coloring $\varphi_2$.
Note that $|\varphi_2(u_{i,j})|=5$ for $i,j\in\{1,2\}$.  For $i\in\{1,2\}$, choose $\varphi_2(v_i)$ as a subset of
$[0,11)\setminus (\varphi_2(u_{i,1})\cup \varphi_2(u_{i,2}))$ of measure~$1$, and let 
$\varphi_2(v)$ be a subset of $[0,11)\setminus (\varphi_2(v_1)\cup \varphi_2(v_2))$ of measure $9$;
again, this extension of $\varphi_2$ to $G$ is a set $11$-coloring but not an $11/4$-coloring.  However, consider the set $11$-coloring $\varphi=\tfrac{1}{2}\varphi_1+\tfrac{1}{2}\varphi_2$.
We now apply Observation~\ref{obs-convex}.
For a vertex $x$ at distance at least three from $v$, we have $|\varphi(x)|=|\varphi_1(x)|=|\varphi_2(x)|=7-d_G(x)$.
For $i\in \{1,2\}$, we have $|\varphi(v_i)|=\tfrac{1}{2}(|\varphi_1(v_i)|+|\varphi_2(v_i)|)=\tfrac{1}{2}(7+1)=4=7-d_G(v_i)$,
and for $j\in \{1,2\}$, we have
$|\varphi(u_{i,j})|=\tfrac{1}{2}(|\varphi_1(u_{i,j})|+|\varphi_2(u_{i,j})|)=\tfrac{1}{2}(3+5)=4=7-d_G(u_{i,j})$.  Finally,
$|\varphi(v)|=\tfrac{1}{2}(|\varphi_1(v)|+|\varphi_2(v)|)=\tfrac{1}{2}(1+9)=5=7-d_G(v)$.
Therefore, $\varphi$ is an $11/4$-coloring of $G$, which is a contradiction.
\end{subproof}

Since $G''$ does not have an $11/4$-coloring, it contains a critical induced sub-e-graph $F_2$, which by the minimality
of $G$ belongs to $\CC_0$.  By a computer-assisted enumeration, we verified that for any 4-tuple $(a_1,a_2,a_3,a_4)$ of distinct non-nailed vertices
of degree two in an e-graph from $\CC_0$ in which neither $a_1a_2$ nor $a_3a_4$ is an edge, adding a vertex $b_1$ adjacent to $a_1$ and $a_2$, a vertex $b_2$ adjacent to $a_3$ and $a_4$,
and a common neighbor of $b_1$ and $b_2$ results either in an e-graph from $\CC_0$ or in a non-critical e-graph.
Since $G$ arises from $G''$ in this way, we have $G''\not\in\CC_0$, and thus $F_2\neq G''$.
By Lemma~\ref{lemma-nok4}, the underlying graph of $F_2$ is not \kk.

If $F_2$ is a $5$-cycle, then
at least four vertices of $F_2$ have degree three in $G$ by Corollary~\ref{cor-five4}.
Since $F_2$ has at most two nailed vertices by (a0), it follows that $F_2$
contains at least two of the vertices $\{u_{i,j}:i,j\in\{1,2\}\}$.  Since $\deg_G v=2$, Lemma~\ref{lemma-nofour1} and the assumption that $G$ is triangle-free
imply that the distance between $u_{1,1}$ and $u_{1,2}$ and the distance between $u_{2,1}$ and $u_{2,2}$ in $G''$ is at least three.
Consequently $|V(F_2)\cap\{u_{i,1},u_{i,2}\}|=1$ for $i\in \{1,2\}$, exactly four vertices of $F_2$ have degree three in $G$
and exactly one vertex of $F_2$ has degree two in $G$.  Hence, the conclusion of this lemma holds with $K=F_2$.

Therefore, we can assume the underlying graph of $F_2$ is neither $C_5$ nor \kk, and by (a0), $F_2$ has at most one nailed vertex.
By Corollary~\ref{cor-3conn1}, the graph $G''$ is connected.
Since $F_2\neq G''$, Observation~\ref{obs-ncnail} implies $F_2$ has exactly one nailed vertex $x$.
Let $x'$ be the neighbor of $x$ in $V(G'')\setminus V(F_2)$; note that $xx'$ is the only edge of $G''$
between $V(F_2)$ and $V(G'')\setminus V(F_2)$.  Since $G$ is critical, $F_2$ is not an induced sub-e-graph of $G$,
and thus we can by symmetry assume $u_{1,1}\in V(F_2)$.  If $\{u_{2,1}, u_{2,2}\}\cap V(F_2)=\emptyset$,
then we obtain a contradiction with Corollary~\ref{cor-3conn1} by considering the two edges leaving either
$V(F_2)$ (if $u_{1,2}\not\in V(F_2)$) or $V(F_2)\cup \{v_1\}$ (if $u_{1,2}\in V(F_2)$) in $G$.  Therefore, by symmetry
we can assume $u_{2,1}\in V(F_2)$.  Similarly, if $\{u_{1,2},u_{2,2}\}\cap V(F_2)\neq\emptyset$,
then at most two edges leave $V(F_2)\cup \{v_1,v,v_2\}$ in $G$, and we obtain a contradiction with Corollary~\ref{cor-3conn1}.
Therefore, $u_{1,2},u_{2,2}\not\in V(F_2)$.

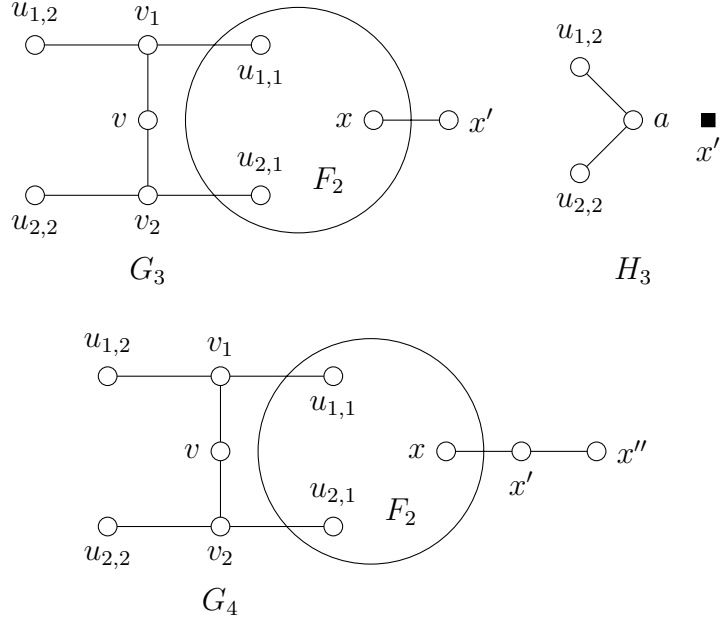
\begin{figure}
\begin{center}
\begin{tikzpicture}
\draw
(0,0) node[vtx,label=left:$v$](v){}
(v) --(90:1) node[vtx,label=above:$v_1$](v1){}
(v) --(-90:1) node[vtx,label=below:$v_2$](v2){}
(v1) --++(0:1.5) node[vtx,label=below:$u_{1,1}$](u11){}
(v1) --++(180:1.5) node[vtxS,label=above:$u_{1,2}$](u12){}
(v2) --++(0:1.5) node[vtx,label=above:$u_{2,1}$](u21){}
(v2) --++(180:1.5) node[vtxS,label=below:$u_{2,2}$](u22){}
(3,0) node[vtx,label=left:$x$](x){} -- ++(1,0) node[vtxS,label=right:$x'$](x'){}
(2,0) ellipse (1.5cm and 1.5cm) 
;
\draw (2.4,-0.8) node{$F_2$};
\draw (0,-2) node {$G_3$};
\end{tikzpicture}
\hskip 1em
\begin{tikzpicture}
\draw
(0,0) node[vtx,label=right:$a$](a){}
-- ++(135:1) node[vtxS,label=above:{$u_{1,2}$}]{}
(a) -- ++(225:1) node[vtxS,label=below:{$u_{2,2}$}]{}
(a)++(1,00) node[nailS,label=below:$x'$](x'){}
(0,-2) node {$H_3$};
;
\end{tikzpicture}\\[10pt]
\begin{tikzpicture}
\draw
(0,0) node[vtx,label=left:$v$](v){}
(v) --(90:1) node[vtx,label=above:$v_1$](v1){}
(v) --(-90:1) node[vtx,label=below:$v_2$](v2){}
(v1) --++(0:1.5) node[vtx,label=below:$u_{1,1}$](u11){}
(v1) --++(180:1.5) node[vtxS,label=above:$u_{1,2}$](u12){}
(v2) --++(0:1.5) node[vtx,label=above:$u_{2,1}$](u21){}
(v2) --++(180:1.5) node[vtxS,label=below:$u_{2,2}$](u22){}
(3,0) node[vtx,label=left:$x$](x){} -- ++(1,0) node[vtx,label=below:$x'$](x'){} -- ++(1,0) node[vtxS,label=right:$x''$](x'){}
(2,0) ellipse (1.5cm and 1.5cm) 
;
\draw (2.4,-0.8) node{$F_2$};
\draw (0,-2) node {$G_4$};
\end{tikzpicture}
\end{center}
\caption{The e-graphs $G_3$ and $G_4$ and the replacement e-graph $H_3$.}\label{fig-all2bad-1}
\end{figure}

Let $S_3=\{u_{1,2},u_{2,2},x'\}$ and let $G_3$ be the sub-e-graph of $G$ consisting of $G[V(F_2)]$, the path
$u_{1,1}v_1vv_2u_{2,1}$, and the edges $u_{1,2}v_1$, $u_{2,2}v_2$, and $xx'$.
In case $\deg x'=2$, let $x''$ be the neighbor of $x'$ distinct from $x$; we have $\deg x''=3$ by Lemma~\ref{lemma-noadj2}.
Let $S_4=\{u_{1,2},u_{2,2},x''\}$ and let $G_4$ be obtained from $G_3$ by adding the edge $x'x''$, see Figure~\ref{fig-all2bad-1}.
Since $F_2\in \CC_0$, there are only finitely many choices for $G_3$ and $G_4$.
The \stex{G_4}{S_4}, and thus $\deg x'=3$.  We claim that $u_{1,2}u_{2,2}\in E(G)$.
\begin{subproof}
Suppose for a contradiction that $u_{1,2}u_{2,2}\not\in E(G)$, and let $H_3$ be the e-graph
with the vertex set $S_3\cup\{a\}$, edges $u_{1,2}a$ and $u_{2,2}a$, and $d_{H_3}(a)=2$, see Figure~\ref{fig-all2bad-1}.
The \starg{$H_3$ enforcing $|\varphi(u_{1,2})\cup\varphi(u_{2,2})|\le 6$}{G_3}{S_3}{G'''}{F_3};
to verify that $F_3\neq G'''$, we consider all possible combinations of an e-graph from $\CC_0$ with one nailed vertex (corresponding to $x'$) minus a vertex of
degree two to represent $F_3-a$, another e-graph from $\CC_0$ with one nailed vertex (and with the underlying graph distinct from $C_5$ and \kk) to represent $F_2$,
and the neighborhood of a vertex of degree two (corresponding to the path $v_1vv_2$), and conclude by a computer-assisted enumeration that all graphs arising
in this way either belong to $\CC_0$ or are not critical.

By Corollary~\ref{cor-3conn1}, the graph $G'''$ is $2$-edge-connected.  Since $F_3\neq G'''$, Observation~\ref{obs-ncnail}
implies $F_3$ has at least two nailed vertices distinct from $x'$. By (a0), $F_3$ has exactly two nailed vertices and $x'\not\in V(F_3)$,
and the underlying graph of $F_3$ is either $C_5$ or \kk.
Furthermore, since $G$ is critical, $F_3$ is not an induced sub-e-graph of $G$, and thus $a\in V(F_3)$.
If $F_3$ were a $5$-cycle, then $u_{1,2}$ and $u_{2,2}$ would have to be its nailed vertices, since they have degree three in~$G$;
however, then $F_3$ would contain two adjacent vertices whose degree in $G$ is two, contradicting Lemma~\ref{lemma-noadj2}.

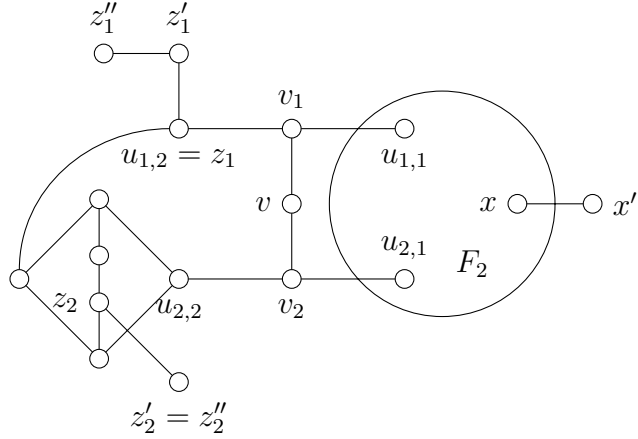
\begin{figure}
\begin{center}
\begin{tikzpicture}
\draw
(0,0) node[vtx,label=left:$v$](v){}
(v) --(90:1) node[vtx,label=above:$v_1$](v1){}
(v) --(-90:1) node[vtx,label=below:$v_2$](v2){}
(v1) --++(0:1.5) node[vtx,label=below:$u_{1,1}$](u11){}
(v1) --++(180:1.5) node[vtx,label=below:{$u_{1,2}=z_1$}](u12){}
(v2) --++(0:1.5) node[vtx,label=above:$u_{2,1}$](u21){}
(v2) --++(180:1.5) node[vtx,label=below:$u_{2,2}$](u22){}
(3,0) node[vtx,label=left:$x$](x){} -- ++(1,0) node[vtxS,label=right:$x'$](x'){}
(2,0) ellipse (1.5cm and 1.5cm) 
(u22) --++(135:1.5) node[vtx](a){}
(u22) --++(225:1.5) node[vtx](b){}  --++(135:1.5) node[vtx](c){} -- (a)
(c) to[out=90,in=180] (u12) 
(a) -- node[pos=0.333,vtx]{}   node[pos=0.666,vtx,label=left:$z_2$](z2){}  (b)
(z2) --++(135:-1.5) node[vtxS,label=below:${z_2'=z_2''}$]{}
(u12) --++(0,1) node[vtx,label=above:{$z_1'$}]{} --++(-1,0) node[vtxS,label=above:{$z_1''$}]{}
;
\draw (2.4,-0.8) node{$F_2$};
\end{tikzpicture}
\end{center}
\caption{The e-graph $G_5$.}\label{fig-all2bad-2}
\end{figure}
Hence, the underlying graph of $F_3$ is \kk.
Let $z_1$ and $z_2$ be the nailed vertices of $F_3$, and let $z'_1$ and $z'_2$ be their neighbors in $G$ not belonging to $F_3$.
For $i\in\{1,2\}$, if $\deg z'_i=2$, then let $z''_i$ be the neighbor of $z'_i$ distinct from $z_i$, otherwise let $z''_i=z'_i$.
Let $S_5=\{z''_1,z''_2,x'\}$ and let $G_5$ be the sub-e-graph of $G$ consisting of $G_3$, $G[V(F_3-a)]$, the edges $z_1z'_1$ and $z_2z'_2$,
and the edges $z'_iz''_i$ for $i\in\{1,2\}$ such that $\deg z'_i=2$, see Figure~\ref{fig-all2bad-2}.  The \stex{G_5}{S_5}.
\end{subproof}

\begin{figure}
\begin{center}
\begin{tikzpicture}
\draw
(0,0) node[vtx,label=left:$v$](v){}
(v) --(90:1) node[vtx,label=above:$v_1$](v1){}
(v) --(-90:1) node[vtx,label=below:$v_2$](v2){}
(v1) --++(0:1.5) node[vtx,label=below:$u_{1,1}$](u11){}
(v1) --++(180:1.5) node[vtx,label=above:$u_{1,2}$](u12){}
(v2) --++(0:1.5) node[vtx,label=above:$u_{2,1}$](u21){}
(v2) --++(180:1.5) node[vtx,label=below:$u_{2,2}$](u22){}
(3,0) node[vtx,label=left:$x$](x){} -- ++(1,0) node[vtxS,label=right:$x'$](x'){}
(2,0) ellipse (1.5cm and 1.5cm) 
;
\draw 
(u12)--(u22)
(u12) -- ++(180:1) node[vtxS,label=above left:{$w_1=w_1'$}](w1){}
(u22) -- ++(180:1) node[vtx,label=below:$w_2$](w2){}   -- ++(180:1) node[vtxS,label=below:$w_2'$](w2'){}
;
\draw (2.4,-0.8) node{$F_2$};
\draw (0,-2) node {$G_6$};
\end{tikzpicture}
\end{center}
\caption{The e-graph $G_6$.}\label{fig-all2bad-3with2}
\end{figure}

\begin{figure}
\begin{center}
\begin{tikzpicture}
\draw
(0,0) node[vtx,label=left:$v$](v){}
(v) --(90:1) node[vtx,label=above:$v_1$](v1){}
(v) --(-90:1) node[vtx,label=below:$v_2$](v2){}
(v1) --++(0:1.5) node[vtx,label=below:$u_{1,1}$](u11){}
(v1) --++(180:1.5) node[vtx,label=above:$u_{1,2}$](u12){}
(v2) --++(0:1.5) node[vtx,label=above:$u_{2,1}$](u21){}
(v2) --++(180:1.5) node[vtx,label=below:$u_{2,2}$](u22){}
(3,0) node[vtx,label=left:$x$](x){} -- ++(1,0) node[vtxS,label=right:$x'$](x'){}
(2,0) ellipse (1.5cm and 1.5cm) 
;
\draw 
(u12)--(u22)
(u12) -- ++(180:1) node[vtxS,label=above left:{$w_1=w_1'$}](w1){}
(u22) -- ++(180:1) node[vtxS,label=below left:{$w_2=w_2'$}](w2){}   
;
\draw (2.4,-0.8) node{$F_2$};
\draw (0,-2) node {$G_6$};
\end{tikzpicture}
\hskip 1em
\begin{tikzpicture}
\draw
(0,0.5) node[vtx,label=above:$b$](b){}
-- ++(180:1) node[vtxS,label=above:$w_1$](w1){}
(b) -- ++(0,-1) node[nail,label=below:$c$](c){}
-- ++(180:1) node[vtxS,label=below:$w_2$](w2){}
(c)++(1,0.5) node[nailS,label=below:$x'$](x'){}
(0,-2) node {$H_6$};
;
\end{tikzpicture}
\end{center}
\caption{The e-graph $G_6$ and the replacement e-graph $H_6$ in the case $\deg w_1 = \deg w_2 = 3$.}\label{fig-all2bad-3}
\end{figure}
Therefore, we have $u_{1,2}u_{2,2}\in E(G)$.  For $i\in \{1,2\}$, let $w_i$ be the neighbor of $u_{i,2}$ distinct from $v_i$ and $u_{3-i,2}$.
If $\deg w_i=2$, then let $w'_i$ be the neighbor of $w_i$ distinct from $u_{i,2}$, otherwise
let $w'_i=w_i$.  Let $S_6=\{w'_1,w'_2,x'\}$ and let $G_6$ be the sub-e-graph of $G$ obtained from $G_3$ by adding
the path $w_1u_{1,2}u_{2,2}w_2$, and for $i\in\{1,2\}$ such that $\deg w_i=2$, the edge $w_iw'_i$,
see Figure~\ref{fig-all2bad-3with2}.
If $\deg w_1=2$ or $\deg w_2=2$, then the \stex{G_6}{S_6}.  Therefore, $\deg w_1=\deg w_2=3$.

Let $H_6$ be the e-graph with the vertex set $S_6\cup\{b,c\}$, edges $w_1b$, $bc$, and $cw_2$, and $d_{H_6}(b)=2$
and $d_{H_6}(c)=3$, see Figure~\ref{fig-all2bad-3}.  The \starg{$H_6$ enforcing $|\varphi(w_1)\cap\varphi(w_2)|\le 2$}{G_6}{S_6}{G^\star}{F_6};
the case $F_6=G^\star$ is excluded as follows:  Note that $G^\star$ contains two nailed vertices $x'$ and $c$, and
thus by (a0), we would have that the underlying graph of $G^\star$ is $C_5$ or \kk.  Replacing the path $w_1bcw_2$
in $G^\star$ by the path $w_1u_{1,2}u_{2,2}w_2$, we obtain an induced sub-e-graph $G^\star_1$ in $G$ with the same underlying graph,
which by Lemma~\ref{lemma-nok4} cannot be \kk.  Moreover, only three vertices $x'$, $u_{1,2}$, and $u_{2,2}$ of $G^\star_1$
would be nailed, contradicting Corollary~\ref{cor-five4} when $G^\star_1$ is a 5-cycle.

Since $G$ is critical, $F_6$ is not an induced sub-e-graph of $G$, and thus $b,c\in V(F_6)$, and $c$ is a nailed vertex of $F_6$.
However, $G^\star$ is $2$-edge-connected by Corollary~\ref{cor-3conn1}, implying by Observation~\ref{obs-ncnail}
that $F_6\neq G^\star$ has at least two nailed vertices distinct from $c$; this contradicts (a0).
\end{proof}

Lemma~\ref{lemma-all2bad} shows that near to each vertex $v$ of degree two, there must be a 5-cycle containing another vertex $v'$
of degree two.  We plan to apply Lemma~\ref{lemma-all2bad} to $v'$, obtaining another $5$-cycle containing a vertex of degree two.
Let us now exclude two special combinations of 5-cycles that can arise in this way.
\begin{lemma}\label{lemma-symonedge}
Let $G$ be a smallest counterexample and let $C_1=v_1v_2v_3v_4v_5$ and $C_2=v_1v_2w_3w_4w_5$ be 5-cycles in $G$ intersecting in
the edge $v_1v_2$.  If $\deg v_4=2$, then $\deg w_4=3$.
\end{lemma}
\begin{proof}
Suppose for a contradiction that $\deg w_4=2$.  Note that $v_3w_5, w_3v_5\not\in E(G)$ by Lemma~\ref{lemma-nofour1}.
Since $G$ is triangle-free, we conclude that $C_1\cup C_2$ is an induced sub-e-graph of $G$.
Let $z_1$, $z_2$, $z_3$, and $z_4$ be the neighbors of $v_3$, $v_5$, $w_3$, and $w_5$, respectively, not belonging to $C_1\cup C_2$.
By Lemma~\ref{lemma-noadj2}, the vertices $z_1$, \ldots, $z_4$ have degree three.
Observe there exists a permutation $\pi$ of $\{1,2,3,4\}$ such that $z_{\pi(1)}z_{\pi(2)},z_{\pi(3)}z_{\pi(4)}\not\in E(G)$:
Otherwise, since $G$ is triangle-free, we could by symmetry assume $z_1z_2,z_1z_3,z_1z_4\in E(G)$,
and since $\deg z_1=3$ and $z_1v_3\in E(G)$, we would have $z_i=z_j$ for distinct $i,j\in\{2,3,4\}$.  We have
$z_2\neq z_4\neq z_3$ by Lemmas~\ref{lemma-C4only3vtxs} and \ref{lemma-nofour1}, and thus $z_2=z_3$.  However, then
$z_4$ would be incident with a bridge, contradicting Lemma~\ref{lemma-conn}.

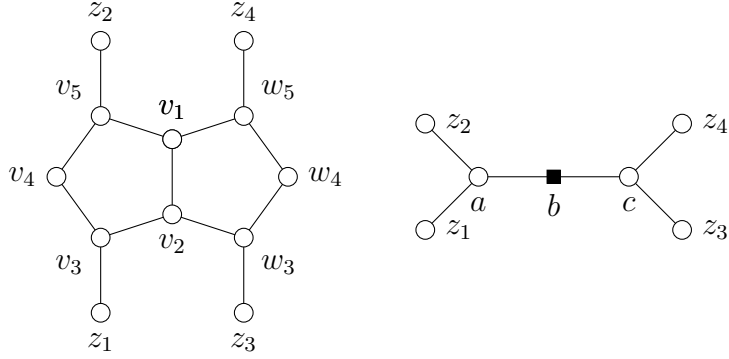
\begin{figure}
\begin{center}
\vc{
\begin{tikzpicture}
\draw
(0,0) node[vtx,label=above:$v_1$](v1){}
-- ++(90+72:1) node[vtx,label=above left:$v_5$](v5){}
-- ++(90+2*72:1) node[vtx,label=left:$v_4$](v4){}
-- ++(90+3*72:1) node[vtx,label=below left:$v_3$](v3){}
-- ++(90+4*72:1) node[vtx,label=below:$v_2$](v2){}
--(v1)
(v5) -- ++(90:1) node[vtxS,label=above:$z_2$]{}
(v3) -- ++(90:-1) node[vtxS,label=below:$z_1$]{}

(0,0) node[vtx,label=above:$v_1$](v1){}
-- ++(90-72:1) node[vtx,label=above right:$w_5$](w5){}
-- ++(90-2*72:1) node[vtx,label=right:$w_4$](v4){}
-- ++(90-3*72:1) node[vtx,label=below right:$w_3$](w3){}
--(v2)
(w5) -- ++(90:1) node[vtxS,label=above:$z_4$]{}
(w3) -- ++(90:-1) node[vtxS,label=below:$z_3$]{}
;
\end{tikzpicture}
}
\hskip 1em
\vc{
\begin{tikzpicture}
\draw
(0,0)  node[nail,label=below:$b$](b){}
-- ++(1,0)  node[vtx,label=below:$c$](c){}
-- ++(45:1) node[vtxS,label=right:$z_4$]{}
(c) -- ++(-45:1) node[vtxS,label=right:$z_3$]{}
(b) -- ++(-1,0)  node[vtx,label=below:$a$](a){}
-- ++(45:-1) node[vtxS,label=right:$z_1$]{}
(a) -- ++(-45:-1) node[vtxS,label=right:$z_2$]{}
;
\end{tikzpicture}
}
\end{center}
\caption{The e-graph $G_1$ and the replacement e-graph $H$ if $\pi$ is the identity.}\label{fig-symonedge-1}
\end{figure}
Let $S=\{z_1,z_2,z_3,z_4\}$ and let $G_1$ be the sub-e-graph of $G$ consisting of $C_1\cup C_2$ and the
edges $v_3z_1$, $v_5z_2$, $w_3z_3$, and $w_5z_4$.  Let $H$ be the e-graph with the vertex set $S\cup\{a,b,c\}$,
edges $z_{\pi(1)}a$, $z_{\pi(2)}a$, $ab$, $bc$, $cz_{\pi(3)}$, and $cz_{\pi(4)}$, and with $d_H(a)=d_H(b)=d_H(c)=3$, see Figure~\ref{fig-symonedge-1}.
Note that for any $11/4$-coloring $\varphi$ of $H$, Observation~\ref{obs-constraints}
implies $|(\varphi(z_{\pi(1)})\cup \varphi(z_{\pi(2)}))\cap \varphi(c)|\le 3$ and $|(\varphi(z_{\pi(3)})\cup \varphi(z_{\pi(4)}))\cap \varphi(c)|=0$,
and thus $\bigl|\varphi(c)\setminus \bigl(\bigcup_{i=1}^4 \varphi(z_i)\bigr)\bigr|\ge 1$, and
$\bigl|\bigcup_{i=1}^4 \varphi(z_i)\bigr|\le 10$.
The \starg{$H$ enforcing $\bigl|\bigcup_{i=1}^4 \varphi(z_i)\bigr|\le 10$}{G_1}{S}{G'}{F}
(note that in the computer-assisted argument to exclude $F=G'$, we can assume that $z_{\pi(1)}\neq z_{\pi(2)}$ and $z_{\pi(3)}\neq z_{\pi(4)}$,
since otherwise $H$ would contain two nailed vertices and by (a0), $F$ would be $C_5$ or \kk{}; and since the nailed vertices of $F$ would be
adjacent, $F$ would also contain two adjacent non-nailed vertices of degree two, implying that $G$ contains two adjacent
vertices of degree two in contradiction to Lemma~\ref{lemma-noadj2}).

By Corollary~\ref{cor-3conn1}, the graph $G'$
is $2$-edge-connected, implying by Observation~\ref{obs-ncnail} that $F$ contains at least two nailed vertices distinct from
the nailed vertex $b$ of $G''$.
By (a0), it follows that $b\not\in V(F)$ and the underlying graph of $F$ is $C_5$ or \kk.
Since $G$ is critical, $F$ is not an induced sub-e-graph of $G$, and thus we can assume $a\in V(F)$, $a$ is a nailed vertex of $F$,
and $z_{\pi(1)}\neq z_{\pi(2)}$.  Since both $z_{\pi(1)}$ and $z_{\pi(2)}$ have degree three in $G''$ and $F$ has only one nailed
vertex other than $a$, $F$ is not a 5-cycle.  If the underlying graph of $F$ is \kk, then $z_{\pi(1)}$ or $z_{\pi(2)}$ is the other nailed vertex of $F$
and $F$ contains two adjacent non-nailed vertices of degree two.  These vertices have degree two in $G$ as well,
contradicting Lemma~\ref{lemma-noadj2}.
\end{proof}

\begin{lemma}\label{lemma-symadj}
Suppose $C_1=v_1v_2v_3v_4v_5$ and $C_2=w_1w_2w_3w_4w_5$ are disjoint cycles in a smallest counterexample $G$
and $v_1w_1, v_3w_3\in E(G)$. If $\deg v_2=2$, then $\deg w_2=3$.
\end{lemma}
\begin{proof}
Suppose for a contradiction that $\deg w_2=2$.  Note that $v_4w_4, w_5v_5\not\in E(G)$ by Lemma~\ref{lemma-nofour1}.
If $v_4w_5\in E(G)$, then Corollary~\ref{cor-3conn1} implies that either $v_5w_4\in E(G)$, or 
$v_5$ and $w_4$ have a common neighbor of degree two; however, both such graphs are $11/4$-colorable.
Therefore $v_5w_4\not\in E(G)$, and symmetrically $v_4w_5\not\in E(G)$.
Let $z_1$, $z_2$, $z_3$, and $z_4$ be the neighbors of $v_4$, $v_5$, $w_4$, and $w_5$, respectively, not belonging to $C_1\cup C_2$.

Suppose $i,j\in \{1,2,3,4\}$ are distinct, $z_i=z_j$, and $\deg z_i=2$.  Letting $\{k,l\}=\{1,2,3,4\}\setminus\{i,j\}$,
Corollary~\ref{cor-3conn1} implies that $z_k=z_l$ also has degree two; but then $G$ is $11/4$-colorable (or contains a triangle),
which is a contradiction.  Hence, if $\deg z_i=2$, then $z_i$ is distinct from all vertices $z_j$ for $j\in\{1,2,3,4\}\setminus\{i\}$;
we let $z'_i$ be the neighbor of $z_i$ not in $C_1\cup C_2$.  If $\deg z_i=3$, we let $z'_i=z_i$.

\begin{figure}
\begin{center}
\vc{
\begin{tikzpicture}
\draw
(-2,0) node[vtx,label=left:$v_5$](v5){}
-- ++(90-72:1) node[vtx,label=above:$v_1$](v1){}
-- ++(90-2*72:1) node[vtx,label=left:$v_2$](v2){}
-- ++(90-3*72:1) node[vtx,label=below:$v_3$](v3){}
-- ++(90-4*72:1) node[vtx,label=left:$v_4$](v4){}
--(v5)

(v5) -- ++(90:1) node[vtxS,label=above:$z_2$]{}
(v4) -- ++(90:-1) node[vtxS,label=below:$z_1$]{}

(2,0) node[vtx,label=right:$w_5$](w5){}
-- ++(90+72:1) node[vtx,label=above:$w_1$](w1){}
-- ++(90+2*72:1) node[vtx,label=right:$w_2$](w2){}
-- ++(90+3*72:1) node[vtx,label=below:$w_3$](w3){}
-- ++(90+4*72:1) node[vtx,label=right:$w_4$](w4){}
--(w5)

(w5) -- ++(90:1) node[vtxS,label=above:$z_4$]{}
(w4) -- ++(90:-1) node[vtxS,label=below:$z_3$]{}

(v1)--(w1)
(v3)--(w3)
;
\end{tikzpicture}
}
\hskip 3em
\vc{
\begin{tikzpicture}
\draw
(0,0) node[nail,label=below:$a_1$](a1){} --
(0,1) node[nail,label=above:$a_2$](a2){}
(1,0) node[nail,label=below:$a_3$](a3){} --
(1,1) node[vtx,label=above:$a_4$](a4){}
(a1) --++(225:1) node[vtxS,label=left:$z_1$]{}
(a2) --++(135:1) node[vtxS,label=left:$z_2$]{}
(a3) --++(-45:1) node[vtxS,label=right:$z_3$]{}
(a4) --++(45:1) node[vtxS,label=right:$z_4$]{}
;
\end{tikzpicture}
}
\end{center}
\caption{The e-graph $G_1$ and the replacement e-graph $H$.}\label{fig-symadj-1}
\end{figure}
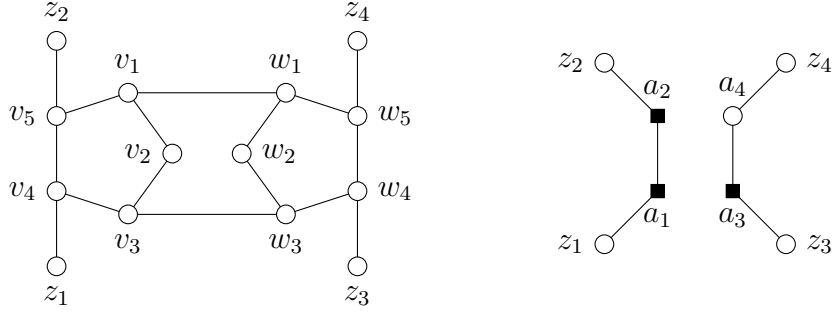
Let $S=\{z'_1,\ldots, z'_4\}$ and let $G_1$ be the sub-e-graph of $G$ consisting of $G[V(C_1)\cup V(C_2)]$,
the edges $v_4z_1$, $v_5z_2$, $w_4z_3$, $w_5z_4$, and the edges $z_iz'_i$ for $i\in \{1,2,3,4\}$ such that $\deg z_i=2$,
see Figure~\ref{fig-symadj-1}.
If at least one of the vertices $z_1$, \ldots, $z_4$ has degree two, then the \stex{G_1}{S}.  Hence, suppose
that all vertices $z_1$, \ldots, $z_4$ have degree three.  Let $H$ be the e-graph with the vertex set $S\cup\{a_1,a_2,a_3,a_4\}$,
edges of the paths $z_1a_1a_2z_2$ and $z_3a_3a_4z_4$, and with $d_H(a_i)=3$ for $i\in\{1,2,3\}$ and $d_H(a_4)=2$, see Figure~\ref{fig-symadj-1}.
The \starg{$H$ enforcing $|\varphi(z_1)\cap \varphi(z_2)|\le 3$ and $|\varphi(z_3)\cap \varphi(z_4)|\le 2$}{G_1}{S}{G'}{F}
(we have $F\neq G'$ by (a0), since $G'$ has three nailed vertices).

Since $G$ is critical, $F$ is not an induced sub-e-graph of $G$
(even after possibly replacing the path $z_1a_1a_2z_2$ by the path $z_1v_4v_5z_2$) and thus $a_3,a_4\in V(F)$.
Corollary~\ref{cor-3conn1} implies $G'$ is connected, and thus by Observation~\ref{obs-ncnail}, $F\neq G'$ contains another nailed vertex in addition to $a_3$.
Hence, by (a0) $F$ contains two nailed vertices and the underlying graph of $F$ is either $C_5$ or \kk.  However, if $F$ were a 5-cycle, then the corresponding
$5$-cycle in $G$ (obtained by replacing the path $z_3a_3a_4z_4$ by $z_3w_4w_5z_4$) would only contain three vertices of degree three,
contradicting Corollary~\ref{cor-five4}, and the case that the underlying graph of $F$ is \kk{} is excluded by Lemma~\ref{lemma-nok4}.
\end{proof}

We are now ready to prove the main result of this section.
\begin{lemma}\label{lemma-3reg}
Every minimum counterexample is $3$-regular.
\end{lemma}
\begin{proof}
Suppose for a contradiction that a minimum counterexample $G$ contains a vertex $v$ of degree two.  By Lemma~\ref{lemma-all2bad}
(considering the vertex of degree two contained in $K$ instead of $v$ if necessary),
we can without loss of generality assume that $v$ is contained in a 5-cycle $C_1=u_{1,1}v_1vv_2u_{2,1}$.
By Corollary~\ref{cor-five4}, all vertices of $C_1$ except for $v$ have degree three.
Let $u_{1,2}$ and $u_{2,2}$ be the neighbors of $v_1$ and $v_2$ not in $C_1$, respectively;
we have $u_{1,2}\neq u_{2,2}$ and $\deg u_{1,2}=\deg u_{2,2}=3$ by Lemmas~\ref{lemma-C4only3vtxs} and Lemma~\ref{lemma-noadj2}.
By Lemma~\ref{lemma-all2bad}, there exists a $5$-cycle $C_2\subset G-\{v_1,v_2,v\}$ containing another vertex of degree two
such that each of $v_1$ and $v_2$ has exactly one neighbor in $C_2$.  We claim that these neighbors of $v_1$ and $v_2$ are non-adjacent.

\begin{subproof}
Suppose for a contradiction that the neighbors of $v_1$ and $v_2$ in $C_2$ are adjacent, and thus we can (by relabelling the vertices if necessary)
assume these neighbors are $u_{1,1}$ and $u_{2,1}$; let $C_2=u_{1,1}u_{2,1}w_3w_4w_5$.  By Lemma~\ref{lemma-symonedge}, we have $\deg w_4=3$, and thus by symmetry
and Corollary~\ref{cor-five4}, we can assume $\deg w_3=2$ and $\deg w_5=3$.

\begin{figure}
\begin{center}
\begin{tikzpicture}
\draw
(0,0) node[vtx,label=above:$u_{1,1}$](u11){}
-- ++(90+72:1) node[vtx,label=above left:$v_1$](v1){}
-- ++(90+2*72:1) node[vtx,label=left:$v$](v){}
-- ++(90+3*72:1) node[vtx,label=below left:$v_2$](v2){}
-- ++(90+4*72:1) node[vtx,label=below:$u_{2,1}$](u21){}
--(u11)
(v1) -- ++(90:1) node[vtx,label=above:$u_{1,2}$](u12){}
(v2) -- ++(90:-1) node[vtxS,label=below:$u_{2,2}$]{}

(u11)
-- ++(90-72:1) node[vtx,label=above right:$w_5$](w5){}
-- ++(90-2*72:1) node[vtx,label=left:$w_4$](w4){}
-- ++(90-3*72:1) node[vtx,label=below:$w_3$](w3){}
--(u21)
(w5) -- ++(90:1) node[vtx,label=above:$x$]{}--(u12)
(w4) -- ++(0:1) node[vtxS,label=below:$w'_4$]{}
(u12) -- ++(180:1) node[vtxS,label=above:$x'$]{}
;
\end{tikzpicture}
\end{center}
\caption{The e-graph $G_1$.}\label{fig-3reg-1}
\end{figure}
We claim that $u_{1,2}$ and $w_5$ do not have a common neighbor $x$ of degree $2$;
indeed, if they did, let $x'$ be the neighbor of $u_{1,2}$ distinct from $v_1$ and $x$, let $w'_4$ be the neighbor of $w_4$ not in $C_2$,
let $S_1=\{x',w'_4,u_{2,2}\}$, and let $G_1$ be the sub-e-graph of $G$ consisting of $G[V(C_1)\cup V(C_2)]$, the path $u_{1,2}xw_5$,
and the edges $x'u_{1,2}$, $w'_4w_4$, and $u_{2,2}v_2$, see Figure~\ref{fig-3reg-1}. The \stex{G_1}{S_1}.

Let us now apply Lemma~\ref{lemma-all2bad} again, with $w_3$ playing the role of $v$, giving a $5$-cycle $K\subset G-\{u_{2,1},w_3,w_4\}$
containing a vertex of degree two and a neighbor of $u_{2,1}$ and of $w_4$.  If $u_{1,1}\in V(K)$, then $v_2\not\in V(K)$
since $u_{2,1}$ has only one neighbor in $K$, and thus $K$ contains the path $u_{1,2}v_1u_{1,1}w_5$; however,
then $w_5$ and $u_{1,2}$ would have a common neighbor of degree two, which is a contradiction.

Therefore $u_{1,1}\not\in V(K)$, $v_2\in V(K)$, and $K=u_{1,2}v_1vv_2u_{2,2}$.  The vertex $w_4$ has a neighbor in $K$.
If $w_4u_{2,2}\in E(G)$, then Corollary~\ref{cor-3conn1} implies that either $u_{1,2}w_5$ have a common neighbor of degree two (which
we already excluded), or $u_{1,2}w_5\in E(G)$, which contradicts Lemma~\ref{lemma-nofour1}.
It follows that $w_4u_{1,2}\in E(G)$.

Lemma~\ref{lemma-nofour1} then implies $w_5u_{2,2}\not\in E(G)$, as otherwise the vertex $w_3$
of degree two would have a neighbor in the 4-cycle $w_4w_5u_{2,2}u_{1,2}$.
By Corollary~\ref{cor-3conn1}, it follows that $w_5$ and $u_{2,2}$
have a common neighbor of degree two.  However, the corresponding graph $G$ is $11/4$-colorable, which is a contradiction.
\end{subproof}

Since the neighbors of $v_1$ and $v_2$ in $C_2$ are non-adjacent, we cannot have $u_{1,1}\in V(C_2)$;
otherwise, as $\deg_{G-\{v_1,v,v_2\}} u_{1,1}=2$, we would also have $u_{2,1}\in V(C_2)$.
Symmetrically, $u_{2,1}\not\in V(C_2)$.
Therefore, $u_{1,2},u_{2,2}\in V(C_2)$.  Let $C_2=u_{1,2}w_1u_{2,2}w_2w_3$.
By Lemma~\ref{lemma-symadj}, we have $\deg w_1=3$, and thus by symmetry and Corollary~\ref{cor-five4}, we can assume
$\deg w_2=2$ and $\deg w_3=3$.

We claim that $u_{1,1}$ and $w_1$ do not have a common neighbor $y$ of degree $2$:
Otherwise Corollary~\ref{cor-3conn1} and Lemma~\ref{lemma-noadj2} would imply $u_{2,1}w_3\in E(G)$; however,
the corresponding graph $G$ is $11/4$-colorable.

Let us now again apply Lemma~\ref{lemma-all2bad},
with $w_2$ playing the role of $v$, giving a $5$-cycle $K'\subset G-\{u_{2,2},w_2,w_3\}$
containing a vertex of degree two and a neighbor of $u_{2,2}$ and of $w_3$.  If $w_1\in V(K')$, then $v_2\not\in V(K')$,
and thus $K'$ contains the path $w_1u_{1,2}v_1u_{1,1}$; however, then $w_1$ and $u_{1,1}$ would have a common neighbor of
degree two, which is a contradiction.

Therefore, $w_1\not\in V(K')$ and $v_2\in V(K')$, which implies $K'=C_1$.
The vertex $w_3$ has a neighbor in $K'$, and $w_3$ is not adjacent to $u_{1,1}$ by Lemma~\ref{lemma-nofour1};
hence, $w_3u_{2,1}\in E(G)$.  However, then Corollary~\ref{cor-3conn1} and Lemma~\ref{lemma-nofour1} imply
$w_1$ and $u_{1,1}$ have a common neighbor of degree two, which is a contradiction.
\end{proof}

Consequently, we can strengthen Corollary~\ref{cor-3conn1}.

\begin{corollary}\label{cor-3connected}
Every minimum counterexample is $3$-edge-connected.
\end{corollary}

\section{4-cycles}

We now show that a minimum counterexample cannot contain a cycle of length four.

\begin{lemma}\label{lemma-girth}
Every minimum counterexample has girth at least five.
\end{lemma}
\begin{proof}
Suppose for a contradiction $C=v_1v_2v_3v_4$ is a 4-cycle in a minimum counterexample $G$.
Note that $C$ is an induced cycle, since $G$ is triangle-free.  For $i\in \{1,\ldots,4\}$, let $u_i$ be the neighbor of $v_i$ outside of $C$.
Note that the vertices $u_1$, \ldots, $u_4$ are pairwise distinct: otherwise, since $G$ is triangle-free, we could assume $u_1=u_3$;
but then $v_1$ and $v_3$ would have the same neighborhood, contradicting the criticality of $G$.
By Lemma~\ref{lemma-nofour1}, the vertices $u_1$, \ldots, $u_4$ have degree three.

Let $G'$ be the e-graph obtained from $G-V(C)$ by setting $d_{G'}(u_i)=2$ for $i\in\{1,\ldots,4\}$.
Suppose that $G'$ has an $11/4$-coloring $\varphi'$.  For $i\in\{1,\ldots,4\}$, let $L_i$ be a subset
of $[0,11)\setminus\varphi'(u_i)$ of measure 6.  Since the fractional choosability is equal to the fractional
chromatic number~\cite{alonlchoos} and $C$ has (fractional) chromatic number $2$, there exist sets
$A_i\subset L_i$ for $i\in \{1,\ldots, 4\}$ such that $|A_i|=3$ and $(A_1\cup A_3)\cap (A_2\cup A_4)=\emptyset$.
For $i\in \{1,\ldots, 4\}$, let $M_i$ be a subset of $[0,11)\setminus (A_{i-1}\cup A_i\cup A_{i+1})$ of measure $2$,
where $A_0=A_4$ and $A_5=A_1$.  Using the fractional choosability of $C$ again, for $i\in \{1,\ldots, 4\}$, there
exist sets $B_i\subset M_i$ of measure $1$ such that $(B_1\cup B_3)\cap (B_2\cup B_4)=\emptyset$.
Let $\varphi(x)=\varphi'(x)$ for any vertex $x$ at distance at least two from $C$, and for $i\in\{1,\ldots,4\}$,
let $\varphi(v_i)=A_i\cup B_i$ and let $\varphi(u_i)$ be a subset of $\varphi'(u_i)\setminus B_i$ of measure $4$.
Then $\varphi$ is an $11/4$-coloring of $G$, which is a contradiction.

Hence, $G'$ does not have an $11/4$-coloring and contains a critical induced sub-e-graph $F$,
belonging to $\CC_0$ by the minimality of $G$.
By a computer-assisted enumeration, we verified that for every graph in $\CC_0$ with no nailed vertices and with exactly
four vertices of degree two, attaching a 4-cycle adjacent to these vertices results in an e-graph that either belongs to $\CC_0$
or is not critical.  Since $G$ arises from $G'$ in this way, we have $G'\not\in \CC_0$, and thus $F\neq G'$.

By Corollary~\ref{cor-3connected}, the graph $G'$ is connected,
and thus $F$ contains at least one nailed vertex by Observation~\ref{obs-ncnail}.
If $F$ contains exactly one nailed vertex, then (b0) implies $F$ contains at least three non-nailed vertices of degree two,
and thus $|V(F)\cap \{u_1,\ldots, u_4\}|\ge 3$.  However, then $G$ contains at most two edges
with exactly one end in $V(F)\cup V(C)$, contradicting Corollary~\ref{cor-3connected}.

\begin{figure}
\begin{center}
\newcommand\basepiece{
(0,0) node[vtx,label=left:$v_1$](v1){}
(1,0) node[vtx,label=below:$v_2$](v2){}
(0,1) node[vtx,label=above:$v_4$](v4){}
(1,1) node[vtx,label=above:$v_3$](v3){}
(v1)--(v2)--(v3)--(v4)--(v1)
(v1) -- ++(45:-1) node[vtx,label=below left:$u_1$](u1){}
(v2) -- ++(315:1) node[vtx,label=below:$u_2$](u2){}
(v3) -- ++(45:1) node[vtx,label=above:$u_3$](u3){}
(v4) -- ++(135:1) node[vtxS,label=left:$u_4$](u4){}
}
\begin{tikzpicture}
\clip (-1.8,-3.4) rectangle (3.5,2.5);
\draw
\basepiece
(u1) --   node[pos=0.5,vtx,label=above:$x_2$](x2){}  (u2) -- node[pos=0.5,vtx,label=left:$x_1$](x1){} (u3)
(u1) to[out = 290, in = -20,looseness=2.8] (u3) 

(x1) -- ++(1,0) node[vtxS,label=below:$x_1'$]{}
(x2) -- ++(0,-1) node[vtxS,label=right:$x_2'$]{}
;
\end{tikzpicture}
\begin{tikzpicture}
\clip (-1.8,-3.4) rectangle (4.3,2.5);
\draw
\basepiece
(u1) --     (u2) --  (u3)
(u1) to[out = 290, in = -20,looseness=2.8]  node[pos=0.33,vtx,label=above left:$x_2$](x2){}  node[pos=0.666,vtx,label=above left:$x_1$](x1){} (u3) 

(x1) -- ++(1,0) node[vtxS,label=below:$x_1'$]{}
(x2) -- ++(0,-1) node[vtxS,label=right:$x_2'$]{}
;
\end{tikzpicture}
\end{center}
\caption{Some of the possibilities for the e-graph $G_1$.}\label{fig-girth-1}
\end{figure}

Therefore, $F$ contains at least two nailed vertices, and by (a0) and Lemma~\ref{lemma-nok4},
$F$ is a 5-cycle with exactly two nailed vertices.  Consequently, we can by symmetry assume $u_1,u_2,u_3\in V(F)$ and $u_4\not\in V(F)$.
Let $x_1$ and $x_2$ be the nailed vertices of $F$ and let $x'_1$ and $x'_2$ be their neighbors not in $F$.
Let $S=\{x'_1,x'_2,u_4\}$ and let $G_1$ be the sub-e-graph of $G$ consisting of $G[V(C)\cup V(F)]$ together with the edges $x_1x'_1$,
$x_2x'_2$, and $v_4u_4$, see Figure~\ref{fig-girth-1}.  The \stex{G_1}{S}.
\end{proof}

\section{3-regular graphs}\label{sec-final}

We now come to the core of the argument, very similar to the one used in~\cite{fracsub}.
For each vertex $v$ of a minimum counterexample, we delete $v$ and the neighbors of $v$
and find an $11/4$-coloring of the resulting e-graph.  We then convexly combine these colorings to obtain an $11/4$-coloring
of the whole graph, which gives a contradiction.  Of course, an issue is that we need to argue that no critical
subgraph is created in this way, which is generally straightforward with the exception of 5-cycles with two nailed vertices.
To deal with these 5-cycles, we exploit the fact that we can afford to keep a few of the vertices at distance two from $v$ nailed.
The following definition is used to determine which vertices to nail.

\begin{figure}
\begin{center}
\newcommand\basepiece{
\foreach \i in {0,...,4}{(90+72*\i:1.5) node[vtx](u\i){}}
\foreach \i in {1,...,4}{(90+72*\i:1.9) node{$u_{\i}$}}
(90:1.9) node{$u$}
(u0)--(u1)--(u2)--(u3)--(u4)--(u0)
(0,0) node[vtx,label=below:$v$](v){}
(v) -- node[vtx,pos=0.5]{}  (u0)
}

\begin{tikzpicture}
\draw
\basepiece
(v) -- node[vtx,pos=0.5]{}  (u1)
(v) -- node[vtx,pos=0.5]{}  (u4)
;
\end{tikzpicture}
\hskip 2em
\begin{tikzpicture}
\draw
\basepiece
(v) -- node[vtx,pos=0.5]{}  (u2)
(v) -- node[vtx,pos=0.5]{}  (u3)
;
\end{tikzpicture}
\hskip 2em
\begin{tikzpicture}
\draw(0,0) node[vtx,label=below:$v$](v){} -- (1,0) node[vtx,label=below:$u$](u){}
(0.5,-1) node{$\vec{Z}$}
;
\draw[-latex] (v)--(u)
;
\end{tikzpicture}
\end{center}
\caption{The edges of the difficult 5-cycle graph.}\label{fig-dcg}
\end{figure}

Let $G$ be a minimum counterexample. The \emph{difficult 5-cycle graph} of $G$ is defined as the directed
graph $\vec{Z}$ with the vertex set $V(G)$, where $(v,u)\in \vec{Z}$ if and only if $G$ contains a $5$-cycle $K=uu_1u_2u_3u_4$,
$v\not\in V(K)$, and $G$ contains three paths $P_1$, $P_2$, and $P_3$ from $v$ to $K$ of length two,
disjoint except for $v$ and intersecting $K$ only in their last vertices, such that the last vertices of $P_1$, $P_2$,
and $P_3$ are either $\{u,u_1,u_4\}$ or $\{u,u_2,u_3\}$, see Figure~\ref{fig-dcg}.  Let $W(v,K)$ denote the set of last vertices of $P_1$, $P_2$, and $P_3$
(this set is uniquely defined without the need to specify $P_1$, $P_2$, and $P_3$, since $G$ has girth at least five,
and thus each neighbor of $v$ can have only one neighbor in $K$), and let $u(v,K)$ denote the vertex $u$ (again, $u$ is uniquely
determined by $v$ and $K$).

The motivation for the definition of the difficult 5-cycle graph comes from the following Lemma.

\begin{lemma}\label{lemma-redugen}
Let $v$ be a vertex of a minimum counterexample $G$ and let $D$ be the set of outneighbors of $v$ in the difficult $5$-cycle graph $\vec{Z}$ of $G$.
Let $v_1$, $v_2$, and $v_3$ be the neighbors of $v$ in $G$, and let $N$ be the set of their neighbors distinct from $v$.
Let $G_v$ be the e-graph obtained from $G-\{v,v_1,v_2,v_3\}$ by setting
$d_{G_v}(x)=2$ for $x\in N\setminus D$ and $d_{G_v}(x)=3$ for $x\in D$.  Then $G_v$ is $11/4$-colorable.
\end{lemma}
\begin{proof}
Suppose for a contradiction that $G_v$ does not have an $11/4$-coloring, and thus it contains a critical induced sub-e-graph $F$,
belonging to $\CC_0$ by the minimality of $G$.  Note that $|N|=6$ by Lemma~\ref{lemma-girth}.
By a computer-assisted enumeration, we verified that for every graph in $\CC_0$
with exactly six vertices of degree two, adding a copy of $K_{1,3}$ and connecting them to form a $3$-regular graph
of girth six results either in an e-graph from $\CC_0$ or in a non-critical e-graph.  Therefore, $F\neq G_v$.

If $G_v$ is disconnected and $F$ is a connected component of $G_v$, then by Corollary~\ref{cor-3connected},
we have $|V(F)\cap N|\ge 3$ and $|(V(G_v)\setminus V(F))\cap N|\ge 3$.   Since $|N|=6$,
we have $|V(F)\cap N|=3$, and thus $F$ has exactly three vertices of degree two.  However, this contradicts (c0).
Therefore, $F$ is not a connected component of $G_v$, and thus by Observation~\ref{obs-ncnail},
$F$ contains a nailed vertex not belonging to $D$.

Suppose first that $F$ contains at least two nailed vertices;  by (a0) and Lemma~\ref{lemma-nok4}, $F$ is a 5-cycle
with exactly two nailed vertices.  The three non-nailed vertices of $F$ necessarily belong to $N\setminus D$.
Since $G$ has girth at least five,
each of the vertices $v_1$, $v_2$, and $v_3$ has at most one neighbor in $F$, and thus $|V(F)\cap N|\le 3$.
It follows that $|V(F)\cap N|=3$.  
Hence $V(F)\cup\{v,v_1,v_2,v_3\}$ form a subgraph of $G$ implying $(v,u)\in E(\vec{Z})$ for some $u\in V(F)\cap N$,
and thus $V(F)\cap D\neq\emptyset$ and $|V(F)\cap (N\setminus D)|<3$, which is a contradiction.

Therefore, $F$ has exactly one nailed vertex $x$, and this vertex does not belong to $D$; let $x'$ be the neighbor of $x$ not in $F$.
If both $v_1$ and $v_2$ had two neighbors in $F$, then at most two edges of $G$ have exactly one end
in either $V(F)\cup \{v_1,v_2,v\}$ (if $v_3$ has no neighbor in $F$) or $V(F)\cup \{v_1,v_2,v_3,v\}$ (if $v_3$ has a neighbor in $F$),
contradicting Corollary~\ref{cor-3connected}.  By symmetry, we conclude that at most one of the vertices $v_1$, $v_2$, and $v_3$ has
two neighbors in $F$, and in particular $|V(F)\cap N|\le 4$.

\begin{figure}
\begin{center}
\begin{tikzpicture}
\draw
(0,0) node[vtx,label=left:$v$](v){}
(1.5,1.5) node[vtx,label=below:$v_1$](v1){}
(1.5,0) node[vtx,label=below:$v_2$](v2){}
(1.5,-1.5) node[vtx,label=below:$v_3$](v3){}
(v)--(v1)
(v)--(v2)
(v)--(v3)
\foreach \x in {1,2,3}{
(v\x)--++(45:1) node[vtxS]{}
(v\x)--++(0:1.5) node[vtx]{}
}
(3.4,0) ellipse (1cm and 2cm)
(3.8,0) node[vtx,label=below:$x$](x){} -- ++(1.2,0)  node[vtxS,label=below:$x'$]{}
(4.7,1) node{$F$}
;
\end{tikzpicture}\hskip 1cm
\begin{tikzpicture}
\draw
(0,0) node[vtx,label=left:$v$](v){}
(1.5,1.5) node[vtx,label=below:$v_1$](v1){}
(1.5,0) node[vtx,label=below:$v_2$](v2){}
(1.5,-1.5) node[vtxS,label=below:$v_3$](v3){}
(v)--(v1)
(v)--(v2)
(v)--(v3)
(v1)--++(45:1) node[vtxS,label=left:$z$]{}
(v1)--++(0:1.5) node[vtx]{}
(v2)--++(20:1.5) node[vtx]{}
(v2)--++(-20:1.5) node[vtx]{}
(3.4,0) ellipse (1cm and 2cm)
(3.8,0) node[vtx,label=below:$x$](x){} -- ++(1.2,0)  node[vtxS,label=below:$x'$]{}
(4.7,1) node{$F$}
;
\end{tikzpicture}
\end{center}
\caption{The e-graph $G_1$.}\label{fig-redugen-1}
\end{figure}
On the other hand, $F$ contains at least three non-nailed vertices of degree two by (b0), and thus $|V(F)\cap N|\ge 3$.
If each of the vertices $v_1$, $v_2$, and $v_3$ has at least one neighbor in $F$,
then let $S=(N\setminus V(F))\cup \{x'\}$ and let $G_1$ be the sub-e-graph of $G$ consisting of $G[V(F)\cup\{v_1,v_2,v_3,v\}]$
and the edges from $S$ to $V(F)\cup\{v_1,v_2,v_3,v\}$.  If one of them (say $v_3$) does not have any neighbor in $F$,
then let $z$ be the unique neighbor of $v_1$ or $v_2$ distinct from $v$ and not belonging to $F$, let $S=\{v_3,z,x'\}$
and let $G_1$ be the sub-e-graph of $G$ consisting of $G[V(F)\cup\{v_1,v_2,v\}]$ and the edges from $S$ to $V(F)\cup\{v_1,v_2,v\}$,
see Figure~\ref{fig-redugen-1}.
The \stex{G_1}{S}.
\end{proof}

Next, we constrain the difficult $5$-cycle graph, ensuring that not too
many vertices get nailed in the reductions according to the previous Lemma.

\begin{lemma}\label{lemma-zindeg}
Let $G$ be a minimum counterexample and let $\vec{Z}$ be its difficult $5$-cycle graph.
Then $\vec{Z}$ has maximum indegree at most one.
\end{lemma}
\begin{proof}
Suppose for a contradiction that $(v,u),(v',u)\in E(\vec{Z})$ for two distinct vertices $v,v'\in V(G)$.
Let $v_1$, $v_2$, and $v_3$ be the neighbors of $v$, where $uv_2\in E(G)$.
Let $K=uu_1u_2u_3u_4$ be a $5$-cycle such that $u(v,K)=u$,
and for $i\in \{1,2,3\}$, let $z_i$ be the neighbor of $v_i$ distinct from $v$ and not belonging to $K$.
Let $K'$ be a $5$-cycle such that $u(v',K')=u$.  Let $G_0$ be the sub-e-graph of $G$ consisting of $K$, $K'$,
and all paths of length two from $v$ to $K$ and from $v'$ to $K'$.  If $G_0$ is an induced sub-e-graph, 
then let $S_1$ be the set of vertices in $V(G)\setminus V(G_0)$ with a neighbor in $G_0$ and let $G_1$ be the sub-e-graph of
$G$ consisting of $G_0$ and the edges between $S_1$ and $G_0$.

\begin{figure}
\begin{center}
\begin{tikzpicture}
\draw
(0,0)  node[vtx,label=left:$v$](v){}
(1,1)  node[vtx,label=left:$v_1$](v1){}
(1,0)  node[vtx,label=above:$v_2$](v2){}
(1,-1) node[vtx,label=left:$v_3$](v3){}
(2,0)  node[vtx,label=right:$u$](u){}
++(1,1) node[vtx,label=above:$u_1$](u1){}
++(0:1) node[vtx,label=above:$u_2$](u2){}
(u) ++(1,-1) node[vtx,label=below:$u_4$](u4){}
++(0:1) node[vtx,label=below:$u_3$](u3){}
(v1) ++(45:1) node[vtx,label=left:$z_1$](z1){}
(v3) ++(-45:1) node[vtx,label=below:$z_3$](z3){}
(v2) ++(-45:0.5) node[vtx,label=right:$z_2$](z2){}
;
\draw
(v)--(v1)--(u1)
(v)--(v2)--(u)
(v)--(v3)--(u4)
(v1)--(z1)
(v3)--(z3)
(v2)--(z2)
;
\draw[line width=1pt]
(u)--(u1)--(u2)--(u3)--(u4)--(u)
(u) ++ (1,0) node{$K$}
;
\end{tikzpicture}
\end{center}
\caption{The case $K[W(v,K)]$ is connected.}\label{fig-zindeg-1}
\end{figure}
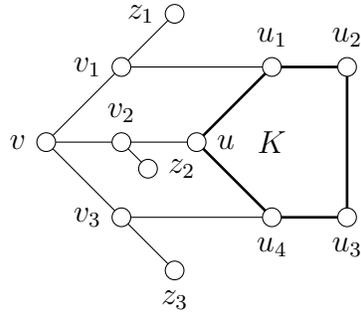

\begin{figure}
\begin{center}
\begin{tikzpicture}
\draw
(0,0)  node[vtx,label=left:$v$](v){}
(1,1)  node[vtx,label=left:$v_1$](v1){}
(1,0)  node[vtx,label=above:$v_2$](v2){}
(1,-1) node[vtx,label=below:$v_3$](v3){}
(2,0)  node[vtx,label=right:$u$](u){}
++(1,1) node[vtx,label=above:$u_1$](u1){}
++(0:1) node[vtx,label=above:$u_2$](u2){}
(u) ++(1,-1) node[vtx,label=below:$u_4$](u4){}
++(0:1) node[vtx,label=right:{$u_3=v'$}](u3){}
(v1) ++(45:1) node[vtxS,label=left:$z_1$](z1){}
(v3) ++(-45:1) node[vtxS,label=below:$z_3$](z3){}
(v2) ++(-45:0.5) node[vtx,label=right:$z_2$](z2){}
(u2) ++(1,0)  node[vtxS](Su2){}
(0,-1) node[vtxS](Sz2){}
;
\draw[line width=1pt]
(v)--(v1)--(u1)
(v)--(v2)--(u)
(v)--(v3)--(u4)
(u)--(u1)--(u2)--(u3)--(u4)--(u)
(v2)--(z2)
(u3) to[bend left=70,looseness=1.3] (z2)
;
\draw
(v1)--(z1)
(v3)--(z3)
(u2)--(Su2)
(z2)--(Sz2)
;
\end{tikzpicture}
\end{center}
\caption{The e-graph $G_1$ in the case $K[W(v,K)]$ is connected and $v'=u_3$. Thick lines show the e-graph $G_0$.}\label{fig-zindeg-2}
\end{figure}
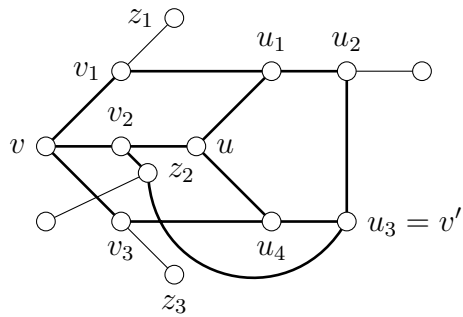

Suppose first that $K[W(v,K)]$ is connected,
and thus we can assume $v_1u_1,v_3u_4\in E(G)$, see Figure~\ref{fig-zindeg-1}.  The vertex $v'\neq v$ is at distance two from $u$,
and thus by symmetry we can assume $v'\in\{v_3,u_3,z_2\}$.  Let us discuss the three cases separately.
\begin{itemize}
\item Suppose $v'=v_3$.  Then $u_4$ and $v$ are vertices with neighbors in $K'$ and
$K'$ contains the path $u_1uv_2$.  Then $u,v_2\in W(v',K')$ and $u(v',K')=u$, and thus $u_1\in W(v',K')$.
But $z_3u_1\not\in E(G)$ since $G$ has girth at least five, which is a contradiction.
\item Suppose $v'=u_3$.  Then $u_2$ and $u_4$ have neighbors in $K'$, and thus $K'=uv_2vv_1u_1$.
Since $u,u_1\in W(v',K')$ and $u(v',K')=u$, we have $v_2\in W(v',K')$, and thus $u_3z_2\in E(G)$.
Note that since $G$ has girth at least five, $G_0$ is an induced sub-e-graph of $G$.
The e-graph $G_1$ in this situation is depicted in Figure~\ref{fig-zindeg-2}.
But the \stex{G_1}{S_1}.
\item Suppose $v'=z_2$.  By Lemma~\ref{lemma-girth}, $z_2v_2u$ is the only path of length two from $z_2$ to $u$,
and thus $u_1uu_4\subset K'$.  Using the fact that $G$ has girth at least five, observe that $K$ is the only $5$-cycle containing
this path; hence, we have $K=K'$.  Moreover, $z_2v_1,z_2v_3\not\in E(G)$, and thus $W(v',K)=\{u,u_2,u_3\}$.
For $i\in \{2,3\}$, let $u'_i$ be the common neighbor of $z_2$ and $u_i$.

\begin{figure}
\begin{center}
\begin{tikzpicture}
\draw
(0,0)  node[vtx,label=left:$v$](v){}
(1,1)  node[vtx,label=left:$v_1$](v1){}
(1,0)  node[vtx,label=above:$v_2$](v2){}
(1,-1) node[vtx,label=left:$v_3$](v3){}
(2,0)  node[vtx,label=above:$u$](u){}
++(1,1) node[vtx,label=above:$u_1$](u1){}
++(0:1) node[vtx,label=right:$u_2$](u2){}
(u) ++(1,-1) node[vtx,label=below:$u_4$](u4){}
++(0:1) node[vtx,label=right:$u_3$](u3){}
(v1) ++(45:1) node[vtx,label=above:{$z_1=u_2''$}](z1){}
(v3) ++(-45:1) node[vtx,label=below:{$z_3=u_3''$}](z3){}
(v2) ++(-45:0.5) node[vtx,label=right:$z_2$](z2){}
(u3) ++(0,-1) node[vtx,label=right:$u_3'$](u3'){}
(u2) ++(0,1) node[vtx,label=right:$u_2'$](u2'){}
;
\draw
(v)--(v1)--(u1)
(v)--(v2)--(u)
(v)--(v3)--(u4)
(v1)--(z1)
(v3)--(z3)
(v2)--(z2)
(z1) to[bend right=90,looseness=2.5] (z3)
(u2)--(u2')--(z1)
(u3)--(u3')--(z3)
(u2') to[bend right] (z2)
(u3') to[bend left=20] (z2)
;
\draw[line width=1pt]
(u)--(u1)--(u2)--(u3)--(u4)--(u)
(u) ++ (1.1,0) node{$K=K'$}
;
\end{tikzpicture}
\end{center}
\caption{The critical e-graph arising in the case $K[W(v,K)]$ is connected and $v'=z_2$.}\label{fig-zindeg-3}
\end{figure}
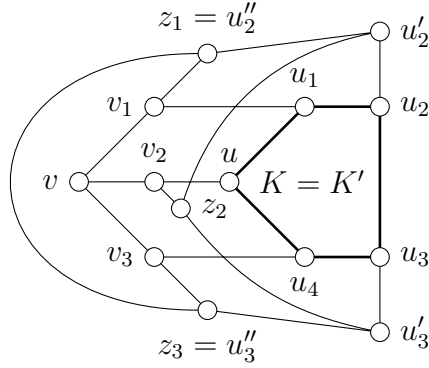

Since $G$ has girth at least five, $z_3\neq u'_3$ and $z_1\neq u'_2$.  If $z_3=u'_2$, then Corollary~\ref{cor-3connected}
implies $z_1=u'_3$; however, the resulting e-graph is $11/4$-colorable.  Hence $z_3\neq u'_2$ and symmetrically $z_1\neq u'_3$.
Consequently, $G_0$ is an induced sub-e-graph of $G$.  For $i\in\{2,3\}$, let $u''_i$ be the neighbor of $u'_i$ distinct from
$u_i$ and $z_2$.  If $u''_2=z_1$ and $u''_3=z_3$, then $z_1z_3\in E(G)$ by Corollary~\ref{cor-3connected} and the resulting e-graph
depicted in Figure~\ref{fig-zindeg-3}  belongs to $\CC_0$.  Hence, we can by symmetry assume $u''_2\neq z_1$.

\begin{figure}[h]
\begin{center}
\begin{tikzpicture}
\draw
(0,0)  node[vtx,label=left:$v$](v){}
(1,1)  node[vtx,label=left:$v_1$](v1){}
(1,0)  node[vtx,label=above:$v_2$](v2){}
(1,-1) node[vtx,label=left:$v_3$](v3){}
(2,0)  node[vtx,label=above:$u$](u){}
++(1,1) node[vtx,label=above:$u_1$](u1){}
++(0:1) node[vtx,label=right:$u_2$](u2){}
(u) ++(1,-1) node[vtx,label=below:$u_4$](u4){}
++(0:1) node[vtx,label=right:$u_3$](u3){}
(v1) ++(45:1) node[vtxS,label=above:$z_1$](z1){}
(v3) ++(-45:1) node[vtxS,label=below:$z_3$](z3){}
(v2) ++(-45:0.5) node[vtx,label=right:$z_2$](z2){}
(u3) ++(0,-1) node[vtx,label=right:$u_3'$](u3'){}
(u2) ++(0,1) node[vtx,label=right:$u_2'$](u2'){}
;
\draw
(v)--(v1)--(u1)
(v)--(v2)--(u)
(v)--(v3)--(u4)
(v1)--(z1)
(v3)--(z3)
(v2)--(z2)
(u2)--(u2') -- ++ (-1,0) node[vtxS,label=above:$u_2''$]{} 
(u3)--(u3') -- ++ (-1,0) node[vtxS,label=below:$u_3''$]{} 
(u2') to[bend right] (z2)
(u3') to[bend left=20] (z2)
;
\draw[line width=1pt]
(u)--(u1)--(u2)--(u3)--(u4)--(u)
(u) ++ (1.1,0) node{$K=K'$}
;

\begin{scope}[xshift=7cm,yshift=1cm]
\draw
 (0,0) node[vtxS,label=above:$z_1$]{}
 --
 (0,-1) node[nail,label=left:$a$]{}
 --
 (1,-1) node[nail,label=right:$b$]{}
 --
 (1,0) node[vtxS,label=above:$u_2''$]{}
 (0,-2) node[nailS,label=below:$z_3$]{}
 (1,-2) node[nailS,label=below:$u_3''$]{}
 ;

\end{scope}

\draw (2.5,-3) node {$G_1$};
\draw (7.5,-3) node {$H_2$};
\end{tikzpicture}
\end{center}
\caption{The e-graph $G_1$ in the case $K[W(v,K)]$ is connected and $v'=z_2$, and the replacement e-graph $H_2$.}\label{fig-zindeg-4}
\end{figure}
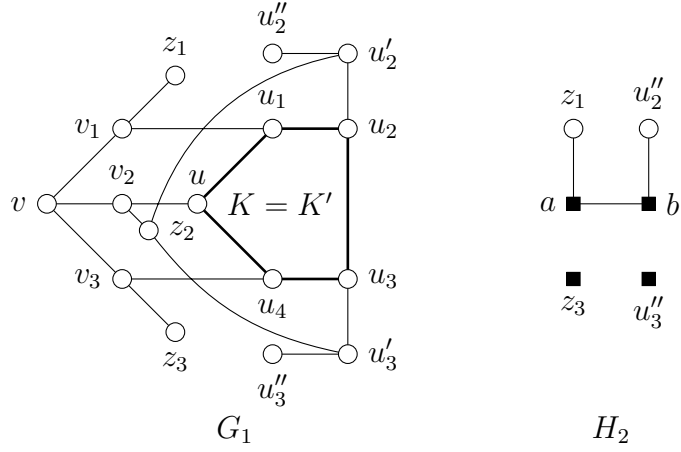

Let $H_2$ be the e-graph with the vertex set $S_1\cup \{a,b\}$,
the edges $z_1a$, $ab$, and $bu''_2$, and $d_{H_2}(a)=d_{H_2}(b)=3$, see Figure~\ref{fig-zindeg-4}.
The \starg{$H_2$ enforcing $|\varphi(z_1)\cap \varphi(u''_2)|\le 3$}{G_1}{S_1}{G'}{F_2};
note that $F_2\neq G'$ by (a0), since $G'$ has at least three nailed vertices.

Note that $G'$ is connected by Corollary~\ref{cor-3connected},
and that $a,b\in V(F_2)$ since $F_2$ is not an induced sub-e-graph of the critical e-graph $G$.
By Observation~\ref{obs-ncnail}, $F_2$ contains a nailed vertex distinct from the nailed vertices $a$ and $b$,
contradicting (a0).
\end{itemize}

\begin{figure}[h]
\begin{center}
\begin{tikzpicture}
\draw
(0,0)  node[vtx,label=left:$v$](v){}
(1,2)  node[vtx,label=above:$v_1$](v1){}
(1,0)  node[vtx,label=above:$v_2$](v2){}
(1,-2) node[vtx,label=below:$v_3$](v3){}
(2,0)  node[vtx,label=above:$u$](u){}
++(1,1) node[vtx,label=above:$u_1$](u1){}
++(0:1) node[vtx,label=right:$u_2$](u2){}
(u) ++(1,-1) node[vtx,label=below:$u_4$](u4){}
++(0:1) node[vtx,label=right:$u_3$](u3){}
(v1) ++(180:1) node[vtx,label=above:{$z_1$}](z1){}
(v3) ++(180:1) node[vtx,label=below:{$z_3$}](z3){}
(v2) ++(-45:0.5) node[vtx,label=below:$z_2$](z2){}
(u1) ++(180:1) node[vtx,label=above:$u_1'$](u1'){} 
(u4) ++(180:1) node[vtx,label=below:$u_4'$](u4'){} 
;
\draw
(v)--(v1) to[bend left] (u2)
(v)--(v2)--(u)
(v)--(v3) to[bend right] (u3)
(v1)--(z1)
(v3)--(z3)
(v2)--(z2)
(u1)--(u1')
(u4)--(u4')
;
\draw[line width=1pt]
(u)--(u1)--(u2)--(u3)--(u4)--(u)
(u) ++ (1.1,0) node{$K$}
;
\end{tikzpicture}
\end{center}
\caption{The case $K[W(v,K)]$ is not connected.}\label{fig-zindeg-5}
\end{figure}
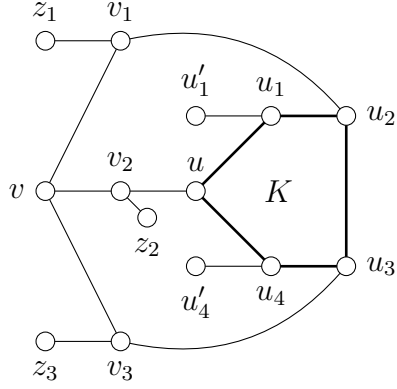

Therefore, $K[W(v,K)]$ is not connected, and symmetrically, $K'[W(v',K')]$ is not connected.
By symmetry, we can assume $v_1u_2,v_3u_3\in E(G)$.  For $i\in \{1,4\}$, let $u'_1$ and $u'_4$
be the neighbors of $u_1$ and $u_4$ distinct from $u$, $u_2$, and $u_3$, see Figure~\ref{fig-zindeg-5}.
The vertex $v'\neq v$ is at distance two from $u$, and thus by symmetry we can assume $v'\in\{u_3,z_2,u'_4\}$.
Let us discuss the three cases separately.
\begin{itemize}
\item Suppose $v'=u_3$.  Then $u,u_1\in W(v',K')$, and since $u=u(v',K')$ and $uu_1\in E(G)$,
it follows that $K'[W(v',K')]$ is connected,
which is a contradiction.
\item Suppose $v'=z_2$.  Then $u_1uu_4\subset K'$.  Since $G$ has girth at least five, observe that $z_2$ is not at distance two
from $u_2$ or $u_3$, and since $K'[W(v',K')]$ is not connected, we have $K'\neq K$.  Consequently,
$u'_1u'_4\in E(G)$, $K'=u'_1u_1uu_4u'_4$, and $W(v',K')=\{u'_1,u,u'_4\}$.  Hence, for $i\in\{1,4\}$,
$z_2$ and $u'_i$ have a common neighbor $w_i$.  

\begin{figure}[h]
\begin{center}
\begin{tikzpicture}
\draw
(-0.5,0)  node[vtx,label=left:$v$](v){}
(1,2)  node[vtx,label=above:$v_1$](v1){}
(1,0)  node[vtx,label=above:$v_2$](v2){}
(1,-2) node[vtx,label=below:$v_3$](v3){}
(3,0)  node[vtx,label=right:$u$](u){}
(2,0)++(1,1) node[vtx,label=above:$u_1$](u1){}
++(0:1) node[vtx,label=right:$u_2$](u2){}
(2,0) ++(1,-1) node[vtx,label=below:$u_4$](u4){}
++(0:1) node[vtx,label=right:$u_3$](u3){}
(v2) ++(-60:0.5) node[vtx,label=left:{$v'=z_2$}](z2){}
(u1) ++(180:1) node[vtx,label=above:$u_1'$](u1'){} 
(u4) ++(180:1) node[vtx,label=below:$u_4'$](u4'){} 
;
\draw
(v)--(v1) to[bend left] (u2)
(v)--(v2)--(u)
(v) to[bend right=50](v3) (v3) to[bend right] (u3)
(v2)--(z2)
(u1)--(u1')
(u4)--(u4')
(u1') to[bend left] (u4')
(z2) -- node[vtx,pos=0.7,label=above left:$w_1$](w1){}(u1')
(z2) -- node[vtx,pos=0.5,label=below left:$w_4$](w4){}(u4')
(v3)--(w4)
(v1)--(w1)
;
\draw
(u)--(u1)--(u2)--(u3)--(u4)--(u)
;
\end{tikzpicture}
\end{center}
\caption{The critical e-graph arising in the case $K[W(v,K)]$ is not connected and $v'=z_2$. It is isomorphic to $F_{14}^{(2)}$. }\label{fig-zindeg-6}
\end{figure}
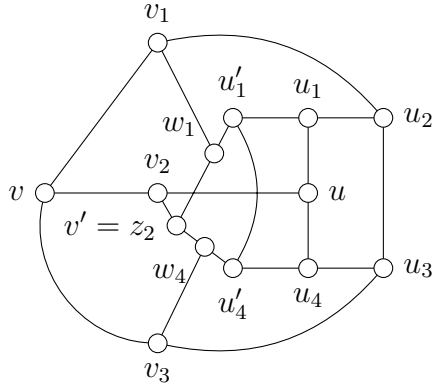

\begin{figure}[h]
\begin{center}
\begin{tikzpicture}
\draw
(-0.5,0)  node[vtx,label=left:$v$](v){}
(1,2)  node[vtx,label=above:$v_1$](v1){}
(1,0)  node[vtx,label=above:$v_2$](v2){}
(1,-2) node[vtx,label=below:$v_3$](v3){}
(3,0)  node[vtx,label=right:$u$](u){}
(2,0)++(1,1) node[vtx,label=above:$u_1$](u1){}
++(0:1) node[vtx,label=right:$u_2$](u2){}
(2,0) ++(1,-1) node[vtx,label=below:$u_4$](u4){}
++(0:1) node[vtx,label=right:$u_3$](u3){}
(v1) ++(180:1) node[vtxS,label=above:{$z_1$}](z1){}
(v3) ++(180:1) node[vtxS,label=below:{$z_3$}](z3){}
(v2) ++(-60:0.5) node[vtx,label=left:{$v'=z_2$}](z2){}
(u1) ++(180:1) node[vtx,label=above:$u_1'$](u1'){} 
(u4) ++(180:1) node[vtx,label=below:$u_4'$](u4'){} 
;
\draw
(v)--(v1) to[bend left] (u2)
(v)--(v2)--(u)
(v) to[bend right=50](v3) (v3) to[bend right] (u3)
(v1)--(z1)
(v3)--(z3)
(v2)--(z2)
(u1)--(u1')
(u4)--(u4')
(u1') to[bend left] (u4')
(z2) -- node[vtx,pos=0.7,label=above left:$w_1$](w1){}(u1')
(z2) -- node[vtx,pos=0.5,label=below left:$w_4$](w4){}(u4')
(w1) -- ++(100:0.9)node[vtxS]{}
(w4) -- ++(260:0.9)node[vtxS]{}
;
\draw
(u)--(u1)--(u2)--(u3)--(u4)--(u)
;
\end{tikzpicture}
\end{center}
\caption{The e-graph $G_1$ in the case $K[W(v,K)]$ is not connected and $v'=z_2$.}\label{fig-zindeg-7}
\end{figure}

By Corollary~\ref{cor-3connected}, $\{v_1,v_3,w_1,w_4\}$
is either an independent set in $G$ or contains two edges with pairwise distinct endpoints.
In the latter case, one of the graphs arising in this way is $11/4$-colorable and the other one (depicted in Figure~\ref{fig-zindeg-6})
belongs to $\CC_0$.
Therefore, we can assume $G_0$ is an induced sub-e-graph of $G$.  The e-graph $G_1$ in this case is depicted in Figure~\ref{fig-zindeg-7}.
However, the \stex{G_1}{S_1}.
\item Suppose $v'=u'_4$.  Then $u_1uv_2\subset K'$.  The vertices $v_2$ and $u_2$ do not have a common neighbor, and
the vertices $u_1$ and $v$ do not have a common neighbor, and thus $z_2u'_1\in E(G)$ and $K'=z_2u'_1u_1uv_2$.
Since $K'[W(v',K')]$ is not connected, $u'_4$ has a common neighbor $y_1$ with $u'_1$ and a common neighbor $y_2$ with $z_2$.

\begin{figure}
\begin{center}
\begin{tikzpicture}
\draw
(-0.5,0)  node[vtx,label=left:$v$](v){}
(1,2)  node[vtx,label=above:$v_1$](v1){}
(1,0)  node[vtx,label=below:$v_2$](v2){}
(1,-2) node[vtx,label=below:$v_3$](v3){}
(3,0)  node[vtx,label=right:$u$](u){}
(2,0)++(1,1) node[vtx,label=above:$u_1$](u1){}
++(0:1) node[vtx,label=right:$u_2$](u2){}
(2,0) ++(1,-1) node[vtx,label=below:$u_4$](u4){}
++(0:1) node[vtx,label=right:$u_3$](u3){}
(v2) ++(60:0.5) node[vtx,label=left:{$z_2$}](z2){}
(u1) ++(180:1) node[vtx,label=above:$u_1'$](u1'){} 
(u4) ++(180:1) node[vtx,label=below:{$v'=u_4'$}](u4'){} 
;
\draw
(v)--(v1) to[bend left] (u2)
(v)--(v2)--(u)
(v) to[bend right=50](v3) (v3) to[bend right] (u3)
(v2)--(z2)
(u1)--(u1')
(u4)--(u4')
(z2)--(u1')
(z2) to node[pos=0.5,vtx,label=below:$y_2$](y2){} (u4')
(u1') to node[pos=0.3,vtx,label=right:$y_1$](y1){} (u4')
(v3) to[bend left] (y2)
(v1)--(y1)
;
\draw
(u)--(u1)--(u2)--(u3)--(u4)--(u)
;
\end{tikzpicture}
\end{center}
\caption{The critical e-graph arising in the case $K[W(v,K)]$ is not connected and $v'=u'_4$. It is isomorphic to $F_{14}^{(1)}$.}\label{fig-zindeg-8}
\end{figure}

\begin{figure}[h!]
\begin{center}
\begin{tikzpicture}
\draw
(-0.5,0)  node[vtx,label=left:$v$](v){}
(1,2)  node[vtx,label=above:$v_1$](v1){}
(1,0)  node[vtx,label=below:$v_2$](v2){}
(1,-2) node[vtx,label=below:$v_3$](v3){}
(3,0)  node[vtx,label=right:$u$](u){}
(2,0)++(1,1) node[vtx,label=above:$u_1$](u1){}
++(0:1) node[vtx,label=right:$u_2$](u2){}
(2,0) ++(1,-1) node[vtx,label=below:$u_4$](u4){}
++(0:1) node[vtx,label=right:$u_3$](u3){}
(v1) ++(180:1) node[vtxS,label=above:{$z_1$}](z1){}
(v3) ++(180:1) node[vtxS,label=below:{$z_3$}](z3){}
(v2) ++(60:0.5) node[vtx,label=left:{$z_2$}](z2){}
(u1) ++(180:1) node[vtx,label=above:$u_1'$](u1'){} 
(u4) ++(180:1) node[vtx,label=below:{$v'=u_4'$}](u4'){} 
;
\draw
(v)--(v1) to[bend left] (u2)
(v)--(v2)--(u)
(v) to[bend right=50](v3) (v3) to[bend right] (u3)
(v1)--(z1)
(v3)--(z3)
(v2)--(z2)
(u1)--(u1')
(u4)--(u4')
(z2)--(u1')
(z2) to node[pos=0.5,vtx,label=below:$y_2$](y2){} (u4')
(u1') to node[pos=0.3,vtx,label=right:$y_1$](y1){} (u4')
(y1) -- ++(45:0.5)node[vtxS]{}
(y2) -- ++(225:1)node[vtxS]{}

;
\draw
(u)--(u1)--(u2)--(u3)--(u4)--(u)
;
\end{tikzpicture}
\end{center}
\caption{The e-graph $G_1$ in the case $K[W(v,K)]$ is not connected and $v'=u'_4$.}\label{fig-zindeg-9}
\end{figure}

By Corollary~\ref{cor-3connected}, $\{v_1,v_3,y_1,y_2\}$ is either an independent set in $G$ or contains two edges with pairwise distinct endpoints.
In the latter case, one of the graphs arising in this way is $11/4$-colorable and the other one (depicted in Figure~\ref{fig-zindeg-8}) belongs to $\CC_0$.
Therefore, we can assume $G_0$ is an induced sub-e-graph of $G$.  The e-graph $G_1$ in this case is depicted in Figure~\ref{fig-zindeg-9}.
However, the \stex{G_1}{S_1}.
\end{itemize}
\end{proof}

We are now ready to prove our main result.

\begin{proof}[Proof of Lemma~\ref{lemma-main}]
Suppose for a contradiction that there exists an e-graph $G\in \CC$ not belonging to $\CC_0$;
we can assume $G$ is a minimum counterexample.  Consequently, $G$ is $3$-regular by Lemma~\ref{lemma-3reg}
and has girth at least five by Lemma~\ref{lemma-girth}.  Let $\vec{Z}$ be the difficult $5$-cycle graph of $G$;
by Lemma~\ref{lemma-zindeg}, $\vec{Z}$ has maximum indegree at most one.

For each vertex $v\in V(G)$, let
the e-graph $G_v$ be defined as in the statement of Lemma~\ref{lemma-redugen} and let $\varphi_v$ be an $11/4$-coloring
of $G_v$.  Let $v_1$, $v_2$, and $v_3$ be the neighbors of $v$ in $G$, and let us extend $\varphi_v$ to a set $11$-coloring of $G$
as follows. For $i\in\{1,2,3\}$, denoting by $u_{i,1}$ and $u_{i,2}$ the neighbors of $v_i$ distinct from $v$,
choose $\varphi_v(v_i)$ as a subset of $[0,11)\setminus (\varphi_v(u_{i,1})\cup\varphi_v(u_{i,2}))$ of measure $1$.
Then, select $\varphi_v(v)$ as a subset of $[0,11)\setminus (\varphi_v(v_1)\cup\varphi_v(v_2)\cup\varphi_v(v_3))$ of measure $8$.

Let $n=|V(G)|$ and let $\varphi=\sum_{v\in V(G)}\tfrac{1}{n}\varphi_v$.  Consider a vertex $v\in V(G)$ whose indegree in $\vec{Z}$ is $d\le 1$.
We have $|\varphi_v(v)|=8$ and for each neighbor $x$ of $v$, we have $|\varphi_x(v)|=1$.  For each vertex $u$ at distance two from $v$,
we have $|\varphi_u(v)|=4$ if $(u,v)\in E(\vec{Z})$ and $|\varphi_u(v)|=5$ if $(u,v)\not\in E(\vec{Z})$.  Finally, for each vertex $y$ at
distance greater than two from $v$, we have $|\varphi_y(v)|=4$.  Since $G$ has girth five, exactly $6$ vertices of $G$ are at distance
exactly two from $v$.  By Observation~\ref{obs-convex}, we have
\begin{align*}
|\varphi(v)|&=\frac{(n-10)\times 4+8+3\times 1+(6-d)\times 5+d\times 4}{n}\\
&=4+\frac{4-3\times 3+(6-d)}{n}\ge 4.
\end{align*}
Consequently, an $11/4$-coloring $\varphi'$ of $G$ can be obtained from $\varphi$ by choosing $\varphi'(v)\subseteq \varphi(v)$
of measure $4$ for every $v\in V(G)$.  This is a contradiction.
\end{proof}

\clearpage

\bibliographystyle{amsplain}
\bibliography{fracsub}

\providecommand{\bysame}{\leavevmode\hbox to3em{\hrulefill}\thinspace}
\providecommand{\MR}{\relax\ifhmode\unskip\space\fi MR }
\providecommand{\MRhref}[2]{%
  \href{http://www.ams.org/mathscinet-getitem?mr=#1}{#2}
}
\providecommand{\href}[2]{#2}
\begin{thebibliography}{10}

\bibitem{akomsem}
Mikl{\'o}s Ajtai, J{\'a}nos Koml{\'o}s, and Endre Szemer{\'e}di, \emph{A note
  on {R}amsey numbers}, Journal of Combinatorial Theory, Series A \textbf{29}
  (1980), no.~3, 354--360.

\bibitem{alonlchoos}
Noga Alon, Zsolt Tuza, and Margit Voigt, \emph{Choosability and fractional
  chromatic numbers}, Discrete Mathematics \textbf{165--166} (1997), 31 -- 38,
  Graphs and Combinatorics.

\bibitem{bajnok1998independence}
B{\'e}la Bajnok and Gunnar Brinkmann, \emph{On the independence number of
  triangle free graphs with maximum degree three}, Journal of Combinatorial
  Mathematics and Combinatorial Computing \textbf{26} (1998), 237--254.

\bibitem{van2019large}
Wouter Cames~van Batenburg, Jan Goedgebeur, and Gwena{\"e}l Joret, \emph{Large
  independent sets in triangle-free cubic graphs: beyond planarity}, Advances
  in Combinatorics (2020), 13667.

\bibitem{fracsub}
Zden\v{e}k Dvo\v{r}\'ak, Jean-S{\'e}bastien Sereni, and Jan Volec,
  \emph{Subcubic triangle-free graphs have fractional chromatic number at most
  14/5}, Journal of the London Mathematical Society \textbf{89} (2014),
  641--662.

\bibitem{fajtlowicz1978size}
Simeon Fajtlowicz, \emph{On the size of independent sets in graphs}, Congr.
  Numer \textbf{21} (1978), 269--274.

\bibitem{fralo}
Kathryn Fraughnaugh and Stephen~C. Locke, \emph{11/30 (finding large
  independent sets in connected triangle-free 3-regular graphs)}, J. Comb.
  Theory, Ser. {B} \textbf{65} (1995), 51--72.

\bibitem{thoheck}
Christopher~Carl Heckman and Robin Thomas, \emph{A new proof of the
  independence ratio of triangle-free cubic graphs}, Discrete Math.
  \textbf{233} (2001), 233--237.

\bibitem{HeTh06}
\bysame, \emph{Independent sets in triangle-free cubic planar graphs}, J.
  Combin. Theory Ser. B \textbf{96} (2006), 253--275.

\bibitem{planfr5}
Anthony Hilton, Richard Rado, and S.~Scott, \emph{A $(<5)$-colour theorem for
  planar graphs}, Bull. London Math. Soc. \textbf{5} (1973), 302--306.

\bibitem{johansson1996asymptotic}
Anders Johansson, \emph{Asymptotic choice number for triangle free graphs},
  DIMACS Technical Report \textbf{91-4, 1196} (1996).

\bibitem{kim1995ramsey}
Jeong~Han Kim, \emph{The {R}amsey number {$R(3, t)$} has order of magnitude
  $t^2/\log t$}, Random Structures \& Algorithms \textbf{7} (1995), no.~3,
  173--207.

\bibitem{martinsson2021simplified}
Anders Martinsson, \emph{A simplified proof of the {J}ohansson-{M}olloy
  {T}heorem using the {R}osenfeld counting method}, arXiv \textbf{2111.06214}
  (2021).

\bibitem{molloy2019list}
Michael Molloy, \emph{The list chromatic number of graphs with small clique
  number}, Journal of Combinatorial Theory, Series B \textbf{134} (2019),
  264--284.

\bibitem{ScheinermanUllman2011}
Edward~R. Scheinerman and Daniel~H. Ullman, \emph{Fractional graph theory},
  Dover Publications Inc., Mineola, NY, 2011. \MR{2963519}

\bibitem{Sta79}
William Staton, \emph{Some {R}amsey-type numbers and the independence ratio},
  Trans. Amer. Math. Soc. \textbf{256} (1979), 353--370.

\end{thebibliography}

\section*{Appendix}

Here we list all elements of the set $\CC$.  The format is as follows.  The description of a graph
is a sequence of parts of form ``$a:xyz;$'', meaning that the edges $ax$, $ay$, and $az$ should be added to the
graph.  The final part of the sequence is of form ``$x1y1$'', meaning that $d_G(x)=\deg x+1$ and $d_G(y)=y+1$; all other
vertices are non-nailed.

{\scriptsize 
\begin{Verbatim}
5 vertices 0 nails a:bd;b:c;c:e;d:e;
5 vertices 1 nails a:bd;b:c;c:e;d:e;d1
10 vertices 0 nails a:ei;b:cg;c:d;d:j;e:gh;f:ghi;h:j;
10 vertices 1 nails a:de;b:fj;c:dh;e:i;f:hi;g:hij;e1
15 vertices 0 nails a:en;b:go;c:di;d:m;e:ik;f:ikn;g:jl;h:jlo;j:m;k:l;
18 vertices 0 nails a:bg;b:h;c:di;d:j;e:fk;f:l;g:im;h:in;j:oq;k:pr;l:qr;m:pq;n:op;o:r;
18 vertices 0 nails a:fr;b:ch;c:k;d:ei;e:j;f:hl;g:hmr;i:np;j:op;k:no;l:nq;m:oq;p:q;
10 vertices 1 nails a:be;b:f;c:gj;d:ei;f:gh;g:i;h:ij;e1
15 vertices 0 nails a:en;b:go;c:im;d:kl;e:ik;f:ikn;g:jl;h:jlo;j:m;
10 vertices 1 nails a:bf;b:c;c:j;d:eh;e:i;f:hi;g:hij;e1
15 vertices 0 nails a:fn;b:ch;c:d;d:o;e:jk;f:jl;g:jln;h:km;i:kmo;l:m;
18 vertices 0 nails a:bg;b:h;c:ei;d:fj;e:f;g:ko;h:mo;i:lp;j:np;k:nq;l:mr;m:n;o:r;p:q;q:r;
18 vertices 0 nails a:fr;b:dh;c:ei;d:e;f:jn;g:knr;h:lo;i:mo;j:lp;k:mp;l:q;m:q;n:q;o:p;
12 vertices 0 nails a:gk;b:ci;c:e;d:fl;e:f;g:ij;h:ijk;j:l;
12 vertices 0 nails a:bi;b:l;c:dk;d:l;e:gj;f:hk;g:h;i:jk;j:l;
12 vertices 0 nails a:cg;b:ei;c:d;d:l;e:f;f:k;g:ij;h:ijl;j:k;
15 vertices 0 nails a:bh;b:i;c:dj;d:m;e:gk;f:gl;h:kn;i:ln;j:lo;k:o;m:no;
15 vertices 0 nails a:bg;b:i;c:jo;d:fk;e:fl;g:hm;h:no;i:km;j:kn;l:mn;
15 vertices 0 nails a:bg;b:i;c:jo;d:fk;e:fl;g:hm;h:no;i:km;j:ln;k:n;l:m;
15 vertices 0 nails a:ch;b:di;c:k;d:l;e:fj;f:n;g:mn;h:jo;i:mo;j:k;k:m;l:no;
15 vertices 0 nails a:bh;b:i;c:dj;d:k;e:fl;f:n;g:mn;h:jo;i:mo;j:l;k:lm;n:o;
15 vertices 0 nails a:dh;b:ci;c:j;d:l;e:fk;f:n;g:mn;h:io;i:k;j:km;l:no;m:o;
15 vertices 0 nails a:bh;b:i;c:ej;d:fk;e:f;g:lm;h:jn;i:ln;j:o;k:mo;l:o;m:n;
15 vertices 0 nails a:bh;b:i;c:ej;d:fk;e:f;g:lm;h:jn;i:ln;j:o;k:lo;m:no;
15 vertices 0 nails a:bh;b:i;c:ej;d:fk;e:f;g:lm;h:lo;i:mo;j:mn;k:no;l:n;
DANGEROUS! 5 vertices 2 nails a:bc;b:e;c:d;d:e;c1d1
DANGEROUS! 5 vertices 2 nails a:ce;b:cd;d:e;c1d1
7 vertices 1 nails a:bf;b:d;c:eg;d:e;f:g;f1
7 vertices 0 nails a:bf;b:d;c:eg;d:e;f:g;
9 vertices 0 nails a:bh;b:d;c:ei;d:f;e:g;f:g;h:i;
8 vertices 0 nails a:bh;b:ce;c:fg;d:egh;e:f;f:h;
8 vertices 0 nails a:ch;b:eg;c:ef;d:efh;f:g;
8 vertices 1 nails a:dh;b:cf;c:g;d:fg;e:fgh;c1
13 vertices 0 nails a:dl;b:fm;c:hk;d:hi;e:hil;f:jk;g:jkm;i:j;
13 vertices 1 nails a:dl;b:fm;c:hk;d:hi;e:hil;f:jk;g:jkm;i:j;c1
18 vertices 0 nails a:dq;b:fr;c:ho;d:km;e:kmq;f:ln;g:lnr;h:jp;i:jop;j:k;l:p;m:n;
18 vertices 0 nails a:er;b:cg;c:h;d:ij;e:im;f:jmr;g:kn;h:ln;i:o;j:p;k:oq;l:pq;m:n;o:p;
18 vertices 0 nails a:be;b:g;c:hr;d:ij;e:im;f:jnr;g:km;h:ln;i:o;j:p;k:pq;l:oq;m:n;o:p;
18 vertices 0 nails a:be;b:g;c:hr;d:ij;e:im;f:jnr;g:km;h:ln;i:o;j:p;k:oq;l:pq;m:n;o:p;
13 vertices 1 nails a:em;b:dg;c:hi;d:j;e:il;f:ilm;g:hk;h:j;j:k;k:l;d1
18 vertices 0 nails a:dq;b:fr;c:hj;d:jn;e:jnq;f:ko;g:kor;h:il;i:mp;k:p;l:mo;m:n;
21 vertices 0 nails a:eu;b:cg;c:h;d:ij;e:io;f:iou;g:kp;h:lp;j:mn;k:mr;l:ms;n:qs;o:q;p:t;q:r;r:t;s:t;
21 vertices 0 nails a:eu;b:cg;c:h;d:ij;e:io;f:iou;g:kp;h:lp;j:mn;k:ms;l:st;m:q;n:rt;o:r;p:q;q:t;r:s;
21 vertices 0 nails a:eu;b:cg;c:h;d:ij;e:io;f:iou;g:kp;h:lp;j:mn;k:mt;l:st;m:q;n:rt;o:r;p:q;q:s;r:s;
15 vertices 0 nails a:ho;b:dg;c:ej;d:m;e:m;f:kl;g:jn;h:ln;i:lno;j:k;k:m;
15 vertices 0 nails a:go;b:ci;c:l;d:ek;e:l;f:jk;g:jn;h:jno;i:km;l:m;m:n;
18 vertices 0 nails a:bf;b:h;c:dg;d:i;e:jr;f:lq;g:mq;h:oq;i:mo;j:np;k:npr;l:mn;o:p;
18 vertices 0 nails a:be;b:g;c:hq;d:ir;e:fk;f:lq;g:km;h:lm;i:no;j:nor;k:p;l:p;m:n;o:p;
18 vertices 0 nails a:be;b:g;c:hq;d:ir;e:fm;f:nq;g:km;h:ln;i:op;j:opr;k:no;l:mp;
18 vertices 0 nails a:gr;b:cf;c:i;d:km;e:jl;f:jq;g:mo;h:mor;i:nq;j:p;k:lp;l:q;n:op;
18 vertices 0 nails a:fr;b:ch;c:i;d:jl;e:km;f:lo;g:lor;h:kq;i:nq;j:kp;m:pq;n:op;
18 vertices 0 nails a:bf;b:g;c:hr;d:jl;e:km;f:jn;g:ln;h:mo;i:mor;j:p;k:lp;n:q;o:q;p:q;
18 vertices 0 nails a:bf;b:g;c:hr;d:jk;e:lm;f:jq;g:kq;h:lo;i:lor;j:n;k:p;m:pq;n:op;
18 vertices 0 nails a:bf;b:g;c:hr;d:kl;e:jm;f:jq;g:kq;h:lo;i:lor;j:n;k:p;m:pq;n:op;
18 vertices 0 nails a:bf;b:g;c:hr;d:jl;e:km;f:jq;g:kq;h:lo;i:lor;j:p;k:p;m:nq;n:op;
DANGEROUS! 8 vertices 2 nails a:cg;b:eh;c:ef;d:efg;f:h;b1h1
DANGEROUS! 8 vertices 2 nails a:bd;b:f;c:gh;d:eg;e:fh;f:g;b1c1
16 vertices 1 nails a:be;b:f;c:dg;d:h;e:jm;f:km;g:il;h:in;i:o;j:ln;k:no;l:p;m:p;o:p;d1
24 vertices 0 nails a:de;b:cf;c:g;d:h;e:im;f:jn;g:kn;h:lm;i:ot;j:ps;k:qs;l:rt;m:u;n:v;o:su;p:vw;q:vx;r:ux;t:w;w:x;
24 vertices 0 nails a:dx;b:cf;c:g;d:hl;e:ilx;f:jw;g:kw;h:mq;i:nq;j:or;k:pr;l:s;m:su;n:sv;o:tu;p:tv;q:r;t:w;u:v;
16 vertices 1 nails a:be;b:f;c:dg;d:h;e:jm;f:km;g:il;h:in;i:o;j:ln;k:lo;m:p;n:p;o:p;d1
24 vertices 0 nails a:bd;b:f;c:gx;d:hl;e:imx;f:jl;g:km;h:nr;i:os;j:pr;k:qs;l:t;m:u;n:ot;o:u;p:tw;q:uv;r:v;s:w;v:w;
24 vertices 0 nails a:bd;b:f;c:gx;d:hl;e:imx;f:jl;g:km;h:nr;i:os;j:pr;k:qs;l:t;m:u;n:ot;o:u;p:tv;q:uv;r:w;s:w;v:w;
16 vertices 1 nails a:be;b:f;c:dg;d:h;e:jm;f:km;g:il;h:in;i:o;j:lo;k:no;l:p;m:p;n:p;d1
10 vertices 1 nails a:cg;b:dh;c:j;d:j;e:fi;f:g;g:h;h:i;i:j;f1
18 vertices 0 nails a:bh;b:i;c:eg;d:fj;e:k;f:k;g:lm;h:np;i:op;j:no;k:l;l:q;m:nr;o:q;p:r;q:r;
18 vertices 0 nails a:bh;b:j;c:eg;d:fi;e:k;f:k;g:lm;h:nq;i:op;j:pq;k:l;l:r;m:np;n:r;o:qr;
18 vertices 0 nails a:bh;b:j;c:eg;d:fi;e:k;f:k;g:lo;h:np;i:mq;j:pq;k:l;l:r;m:op;n:or;q:r;
FULL! 14 vertices 0 nails a:dfn;b:ehn;c:gin;d:gj;e:fk;f:l;g:m;h:im;i:l;j:kl;k:m;
FULL! 14 vertices 0 nails a:cgm;b:efm;c:ei;d:fjn;e:l;f:k;g:jk;h:iln;i:j;k:l;m:n;
10 vertices 1 nails a:bg;b:i;c:dh;d:i;e:fh;f:j;g:hj;i:j;f1
18 vertices 0 nails a:bi;b:j;c:eg;d:fh;e:k;f:k;g:hl;h:m;i:np;j:op;k:q;l:or;m:no;n:q;p:r;q:r;
18 vertices 0 nails a:hr;b:df;c:eg;d:j;e:j;f:gk;g:l;h:mo;i:nor;j:p;k:mq;l:nq;m:p;n:p;o:q;
16 vertices 0 nails a:bd;b:f;c:gp;d:hl;e:imp;f:jl;g:km;h:kn;i:jo;j:k;l:o;m:n;n:o;
18 vertices 0 nails a:fr;b:dh;c:ei;d:k;e:j;f:hl;g:imr;h:n;i:o;j:lo;k:mn;l:p;m:q;n:p;o:q;p:q;
18 vertices 0 nails a:bg;b:i;c:ej;d:fh;e:k;f:l;g:hm;h:q;i:jp;j:o;k:no;l:pq;m:or;n:pr;q:r;
18 vertices 0 nails a:cg;b:dh;c:i;d:j;e:fk;f:l;g:hp;h:q;i:mp;j:nr;k:or;l:qr;m:oq;n:op;
18 vertices 0 nails a:bf;b:h;c:ir;d:ej;e:k;f:lp;g:mqr;h:np;i:oq;j:op;k:nq;l:mo;m:n;
21 vertices 0 nails a:bf;b:i;c:dg;d:h;e:jk;f:jr;g:lq;h:mq;i:nr;j:o;k:pr;l:ns;m:nt;o:ps;p:t;q:u;s:u;t:u;
21 vertices 0 nails a:eu;b:cg;c:h;d:ij;e:io;f:jou;g:kp;h:lp;i:m;j:n;k:qr;l:qs;m:nr;n:s;o:q;p:t;r:t;s:t;
21 vertices 0 nails a:bf;b:g;c:dh;d:i;e:jk;f:jp;g:kp;h:lq;i:mq;j:n;k:o;l:nu;m:tu;n:r;o:su;p:s;q:r;r:t;s:t;
21 vertices 0 nails a:bf;b:i;c:dg;d:h;e:jk;f:jr;g:lq;h:nq;i:mr;j:o;k:pr;l:ou;m:pt;n:tu;o:s;p:u;q:s;s:t;
21 vertices 0 nails a:bf;b:i;c:dg;d:h;e:jk;f:jr;g:lq;h:nq;i:mr;j:o;k:pr;l:pu;m:ot;n:tu;o:u;p:s;q:s;s:t;
21 vertices 0 nails a:bf;b:g;c:dh;d:i;e:jk;f:jp;g:kp;h:lq;i:mq;j:n;k:o;l:rt;m:st;n:os;o:t;p:r;q:u;r:u;s:u;
21 vertices 0 nails a:bf;b:i;c:dg;d:h;e:jk;f:jr;g:lq;h:nq;i:mr;j:o;k:pr;l:mu;m:s;n:tu;o:pt;p:u;q:s;s:t;
21 vertices 0 nails a:bf;b:i;c:dg;d:h;e:jk;f:jr;g:lq;h:nq;i:mr;j:o;k:pr;l:mu;m:s;n:tu;o:pu;p:t;q:s;s:t;
21 vertices 0 nails a:bf;b:i;c:dg;d:h;e:jk;f:jr;g:lq;h:mq;i:nr;j:o;k:pr;l:ps;m:pt;n:ot;o:s;q:u;s:u;t:u;
21 vertices 0 nails a:bf;b:i;c:dg;d:h;e:jk;f:jr;g:lq;h:mq;i:nr;j:o;k:pr;l:os;m:ot;n:pt;p:s;q:u;s:u;t:u;
21 vertices 0 nails a:bf;b:g;c:dh;d:i;e:jk;f:jp;g:kp;h:lq;i:mq;j:n;k:o;l:ns;m:nt;o:rt;p:r;q:u;r:s;s:u;t:u;
16 vertices 0 nails a:dp;b:cf;c:g;d:hl;e:ilp;f:jo;g:ko;h:jn;i:kn;j:m;k:m;l:m;n:o;
18 vertices 0 nails a:fr;b:dh;c:ei;d:j;e:k;f:hl;g:imr;h:p;i:q;j:np;k:oq;l:nq;m:op;n:o;
18 vertices 0 nails a:cg;b:dh;c:k;d:l;e:fi;f:j;g:hn;h:o;i:mq;j:nq;k:pq;l:op;m:pr;n:r;o:r;
21 vertices 0 nails a:bf;b:g;c:dh;d:i;e:jk;f:jp;g:kp;h:lq;i:mq;j:n;k:o;l:nt;m:tu;n:r;o:su;p:s;q:r;r:u;s:t;
21 vertices 0 nails a:bf;b:i;c:dg;d:h;e:jk;f:jr;g:lq;h:nq;i:mr;j:o;k:pr;l:ou;m:pu;n:tu;o:s;p:t;q:s;s:t;
21 vertices 0 nails a:bf;b:i;c:dg;d:h;e:jk;f:jr;g:lq;h:nq;i:mr;j:o;k:pr;l:pu;m:ou;n:tu;o:t;p:s;q:s;s:t;
13 vertices 1 nails a:bf;b:h;c:dg;d:i;e:jk;f:jm;g:kl;h:km;i:lm;j:l;e1
18 vertices 0 nails a:bh;b:i;c:dg;d:j;e:kn;f:lm;g:mq;h:kr;i:lr;j:qr;k:p;l:o;m:p;n:oq;o:p;
18 vertices 0 nails a:bh;b:i;c:dg;d:j;e:kn;f:lm;g:kq;h:lr;i:mr;j:qr;k:p;l:o;m:p;n:oq;o:p;
18 vertices 0 nails a:bh;b:i;c:dg;d:j;e:kn;f:lm;g:lq;h:kr;i:mr;j:qr;k:p;l:o;m:p;n:oq;o:p;
13 vertices 1 nails a:be;b:g;c:hm;d:ij;e:fk;f:lm;g:ik;h:il;j:kl;d1
18 vertices 0 nails a:bf;b:h;c:ir;d:jk;e:lm;f:gp;g:qr;h:jp;i:kq;j:n;k:o;l:op;m:nq;n:o;
18 vertices 0 nails a:bf;b:h;c:ir;d:jl;e:km;f:gp;g:qr;h:jp;i:kq;j:n;k:o;l:op;m:nq;n:o;
18 vertices 0 nails a:bf;b:h;c:ir;d:jk;e:lm;f:gp;g:qr;h:jp;i:kq;j:n;k:o;l:np;m:oq;n:o;
18 vertices 0 nails a:bf;b:h;c:ir;d:jm;e:kl;f:gp;g:qr;h:jp;i:kq;j:n;k:o;l:np;m:oq;n:o;
18 vertices 0 nails a:bf;b:h;c:ir;d:jl;e:km;f:gp;g:qr;h:jp;i:kq;j:n;k:n;l:op;m:oq;n:o;
18 vertices 0 nails a:bf;b:h;c:ir;d:jm;e:kl;f:gp;g:qr;h:jp;i:kq;j:n;k:n;l:op;m:oq;n:o;
13 vertices 1 nails a:be;b:g;c:hm;d:ij;e:fk;f:lm;g:ik;h:jl;i:l;j:k;d1
13 vertices 1 nails a:bf;b:g;c:eh;d:ij;e:k;f:im;g:km;h:jl;i:l;j:m;k:l;e1
13 vertices 1 nails a:bf;b:g;c:eh;d:ij;e:k;f:im;g:km;h:il;j:lm;k:l;e1
13 vertices 1 nails a:bf;b:g;c:eh;d:ij;e:k;f:il;g:jl;h:jm;i:m;k:lm;e1
13 vertices 1 nails a:bf;b:g;c:eh;d:ij;e:k;f:im;g:jm;h:lm;i:k;j:l;k:l;e1
13 vertices 1 nails a:bf;b:g;c:eh;d:ij;e:k;f:hm;g:im;h:l;i:k;j:lm;k:l;e1
13 vertices 1 nails a:bf;b:g;c:eh;d:ij;e:k;f:hm;g:im;h:l;i:l;j:km;k:l;e1
10 vertices 0 nails a:bf;b:h;c:ij;d:eh;e:i;f:gi;g:hj;
10 vertices 1 nails a:bg;b:i;c:dh;d:j;e:fj;f:g;g:h;h:i;i:j;f1
15 vertices 0 nails a:bg;b:h;c:io;d:ek;e:m;f:lm;g:km;h:kn;i:ln;j:lno;
12 vertices 0 nails a:eh;b:il;c:fi;d:gj;e:k;f:k;g:k;h:ij;j:l;
12 vertices 0 nails a:bh;b:j;c:kl;d:fj;e:gk;f:g;h:ik;i:jl;
12 vertices 0 nails a:bh;b:j;c:jl;d:fh;e:gi;f:k;g:k;h:i;i:l;j:k;
15 vertices 0 nails a:bh;b:i;c:dj;d:m;e:fl;f:m;g:kl;h:ko;i:lo;j:no;k:n;m:n;
15 vertices 0 nails a:bi;b:j;c:dh;d:m;e:fl;f:m;g:kl;h:in;i:o;j:lo;k:no;m:n;
15 vertices 0 nails a:bh;b:j;c:di;d:m;e:fl;f:m;g:kl;h:io;i:n;j:ko;k:n;l:o;m:n;
15 vertices 0 nails a:bh;b:k;c:ei;d:fj;e:n;f:n;g:lm;h:mo;i:jl;j:m;k:no;l:o;
15 vertices 0 nails a:go;b:di;c:ej;d:m;e:m;f:kl;g:kn;h:lno;i:jk;j:l;m:n;
15 vertices 0 nails a:bi;b:j;c:eh;d:fk;e:n;f:n;g:lm;h:ko;i:lo;j:mo;k:m;l:n;
15 vertices 0 nails a:bh;b:i;c:ej;d:fk;e:n;f:n;g:lm;h:jo;i:lo;j:k;k:m;l:n;m:o;
15 vertices 0 nails a:bh;b:i;c:ej;d:fk;e:n;f:n;g:lm;h:jo;i:lo;j:k;k:l;m:no;
9 vertices 0 nails a:fi;b:eh;c:fh;d:gh;e:fg;g:i;
10 vertices 0 nails a:fj;b:df;c:eg;d:i;e:i;f:h;g:hj;h:i;
10 vertices 1 nails a:bg;b:i;c:dh;d:i;e:fh;f:g;g:j;h:j;i:j;f1
15 vertices 0 nails a:go;b:ci;c:l;d:ek;e:l;f:jk;g:jn;h:jno;i:mn;k:m;l:m;
12 vertices 0 nails a:bh;b:j;c:jl;d:fh;e:gi;f:g;h:k;i:kl;j:k;
15 vertices 0 nails a:bi;b:j;c:dh;d:m;e:fl;f:m;g:kl;h:kn;i:ko;j:lo;m:n;n:o;
15 vertices 0 nails a:bi;b:j;c:dh;d:m;e:fl;f:m;g:kl;h:kn;i:lo;j:no;k:o;m:n;
15 vertices 0 nails a:bh;b:j;c:di;d:m;e:fl;f:m;g:kl;h:kn;i:ko;j:no;l:n;m:o;
13 vertices 0 nails a:bf;b:h;c:dg;d:i;e:jk;f:jm;g:kl;h:km;i:lm;j:l;
13 vertices 1 nails a:bf;b:g;c:eh;d:ij;e:k;f:im;g:lm;h:jl;i:k;j:m;k:l;e1
18 vertices 0 nails a:gr;b:cf;c:i;d:km;e:jl;f:jq;g:mo;h:mor;i:pq;j:n;k:lp;l:q;n:op;
13 vertices 1 nails a:bf;b:g;c:eh;d:ij;e:k;f:im;g:lm;h:il;j:km;k:l;e1
18 vertices 0 nails a:fr;b:ch;c:i;d:jl;e:km;f:lo;g:lor;h:kq;i:pq;j:kp;m:nq;n:op;
15 vertices 0 nails a:bi;b:j;c:eh;d:fk;e:n;f:n;g:lm;h:mo;i:kl;j:km;l:o;n:o;
15 vertices 0 nails a:bh;b:k;c:ei;d:fj;e:f;g:lm;h:mn;i:lo;j:mo;k:no;l:n;
18 vertices 0 nails a:cg;b:dh;c:i;d:j;e:lm;f:kn;g:lp;h:kr;i:mp;j:mr;k:o;l:q;n:qr;o:pq;
18 vertices 0 nails a:cg;b:dh;c:i;d:j;e:km;f:ln;g:kp;h:lr;i:mp;j:mr;k:q;l:q;n:or;o:pq;
18 vertices 0 nails a:bg;b:h;c:di;d:j;e:km;f:ln;g:lp;h:np;i:kr;j:mr;k:o;l:q;m:q;n:r;o:pq;
18 vertices 0 nails a:bg;b:i;c:dh;d:j;e:km;f:ln;g:kr;h:lq;i:lr;j:oq;k:p;m:or;n:pq;o:p;
18 vertices 0 nails a:bg;b:i;c:dh;d:j;e:km;f:ln;g:kr;h:lq;i:lr;j:oq;k:o;m:pr;n:pq;o:p;
18 vertices 0 nails a:bg;b:h;c:di;d:j;e:mn;f:kl;g:kr;h:lr;i:mq;j:oq;k:o;l:p;m:r;n:pq;o:p;
18 vertices 0 nails a:bh;b:i;c:dg;d:j;e:kn;f:lm;g:kq;h:lr;i:mr;j:oq;k:p;l:o;m:p;n:qr;o:p;
18 vertices 0 nails a:bg;b:i;c:dh;d:j;e:ln;f:km;g:lp;h:kr;i:np;j:qr;k:o;l:q;m:nr;o:pq;
18 vertices 0 nails a:bg;b:h;c:di;d:j;e:kl;f:mn;g:kp;h:lp;i:mr;j:qr;k:q;l:m;n:or;o:pq;
18 vertices 0 nails a:bf;b:h;c:ir;d:jl;e:km;f:jp;g:kqr;h:lp;i:mq;j:n;k:o;l:m;n:oq;o:p;
18 vertices 0 nails a:bh;b:i;c:dg;d:j;e:kn;f:lm;g:lq;h:kr;i:mr;j:oq;k:p;l:p;m:q;n:or;o:p;
18 vertices 0 nails a:bh;b:i;c:dg;d:j;e:kn;f:lm;g:lq;h:kr;i:mr;j:oq;k:o;l:p;m:q;n:pr;o:p;
18 vertices 0 nails a:bg;b:h;c:di;d:j;e:kl;f:mn;g:kp;h:lp;i:mr;j:or;k:q;l:m;n:qr;o:pq;
18 vertices 0 nails a:bf;b:h;c:ir;d:jl;e:km;f:jp;g:kqr;h:np;i:nq;j:k;l:op;m:oq;n:o;
18 vertices 0 nails a:cg;b:dh;c:i;d:j;e:kl;f:mn;g:kq;h:mr;i:oq;j:or;k:p;l:mq;n:pr;o:p;
18 vertices 0 nails a:bf;b:h;c:ir;d:jl;e:km;f:jp;g:kqr;h:np;i:nq;j:o;k:o;l:mp;m:q;n:o;
18 vertices 0 nails a:bg;b:i;c:dh;d:j;e:kn;f:lm;g:kp;h:lr;i:np;j:or;k:q;l:q;m:nr;o:pq;
13 vertices 0 nails a:be;b:g;c:hm;d:ij;e:fk;f:lm;g:ik;h:il;j:kl;
15 vertices 0 nails a:bg;b:k;c:lo;d:ek;e:l;f:ij;g:im;h:ino;j:mn;k:m;l:n;
18 vertices 0 nails a:cg;b:dh;c:j;d:i;e:kl;f:mn;g:kq;h:mr;i:or;j:pq;k:o;l:mq;n:pr;o:p;
18 vertices 0 nails a:bf;b:h;c:ir;d:jl;e:km;f:jp;g:kqr;h:np;i:oq;j:k;l:op;m:nq;n:o;
13 vertices 0 nails a:be;b:g;c:hm;d:ij;e:fk;f:lm;g:ik;h:jl;i:l;j:k;
15 vertices 0 nails a:bg;b:k;c:lo;d:ek;e:l;f:ij;g:im;h:jno;i:n;j:m;k:m;l:n;
18 vertices 0 nails a:bf;b:h;c:ir;d:jl;e:km;f:jp;g:kqr;h:np;i:oq;j:o;k:n;l:mp;m:q;n:o;
11 vertices 0 nails a:bj;b:ck;c:df;d:hj;e:ijk;f:gi;g:hk;h:i;
13 vertices 0 nails a:be;b:g;c:hm;d:il;e:ij;f:ikm;g:jl;h:kl;j:k;
13 vertices 1 nails a:be;b:g;c:hm;d:il;e:ij;f:ikm;g:jl;h:kl;j:k;d1
18 vertices 0 nails a:be;b:g;c:hr;d:ip;e:ln;f:lor;g:mn;h:mo;i:kq;j:kpq;k:l;m:q;n:o;
13 vertices 1 nails a:bf;b:g;c:eh;d:ij;e:k;f:im;g:jm;h:jl;i:k;k:l;l:m;e1
18 vertices 0 nails a:bf;b:g;c:hr;d:jl;e:km;f:jp;g:lp;h:mq;i:mqr;j:n;k:lo;n:oq;o:p;
15 vertices 0 nails a:bf;b:h;c:io;d:ej;e:n;f:jl;g:jmo;h:kl;i:km;k:n;l:m;
15 vertices 0 nails a:go;b:di;c:ej;d:e;f:kl;g:km;h:lmo;i:kn;j:ln;m:n;
15 vertices 0 nails a:fo;b:ch;c:i;d:jn;e:lm;f:jl;g:jmo;h:kl;i:km;k:n;
18 vertices 0 nails a:bf;b:h;c:ir;d:jl;e:km;f:jn;g:kor;h:ln;i:mo;j:p;k:p;l:m;n:q;o:q;p:q;
\end{Verbatim}}

\begin{aicauthors}
\begin{authorinfo}[zd]
  Zdeněk Dvořák\\
  Computer Science Institute of Charles University\\
  Prague, Czech Republic \\
  rakdver\imageat{}iuuk\imagedot{}mff\imagedot{}cuni\imagedot{}cz \\
  \url{https://iuuk.mff.cuni.cz/~rakdver/}
\end{authorinfo}
\begin{authorinfo}[bl]
  Bernard Lidický\\
  Department of Mathematics, Iowa State University\\
  Ames, IA, USA \\
  lidicky\imageat{}iastate\imagedot{}edu \\
  \url{https://lidicky.name/}
\end{authorinfo}
\begin{authorinfo}[bl]
  Luke Postle\\
  Department of Combinatorics and Optimization, University of Waterloo\\
  Waterloo, Ontario, Canada \\
  lpostle\imageat{}uwaterloo\imagedot{}ca \\
  \url{https://www.math.uwaterloo.ca/~lpostle/}
\end{authorinfo}
\end{aicauthors}

\end{document}